\documentclass[moor,nonblindrev,dvipsnames]{informs1} 

\usepackage{natbib}
 \NatBibNumeric
 \def\bibfont{\small}%
 \bibpunct[, ]{[}{]}{,}{n}{}{,}%

\usepackage[colorlinks=true,breaklinks=true,bookmarks=true,urlcolor=blue,
     citecolor=blue,linkcolor=blue,bookmarksopen=false,draft=false]{hyperref}

\def\EMAIL#1{\href{mailto:#1}{#1}}

\TheoremsNumberedThrough     

\EquationsNumberedThrough    

\usepackage{enumitem}

\usepackage{amsmath,amssymb,amsfonts} 

\usepackage{tikz}
\usepackage{ifthen}
\usepackage{stmaryrd}

\usetikzlibrary{shapes}
\usetikzlibrary{calc}

\newcommand{\bx}{\boldsymbol{x}}

\newcommand{\ty}{\widetilde{y}}
\newcommand{\tby}{\widetilde{\by}}
\newcommand{\bu}{\boldsymbol{u}}
\newcommand{\bw}{\boldsymbol{w}}
\newcommand{\by}{\boldsymbol{y}}
\newcommand{\bz}{\boldsymbol{z}}

\newcommand{\bZ}{\boldsymbol{Z}}
\newcommand{\bv}{\boldsymbol{v}}
\newcommand{\bxi}{\boldsymbol{\xi}}
\newcommand{\bpi}{\boldsymbol{\pi}}

\newcommand{\bPhi}{\boldsymbol{\Phi}}
\newcommand{\tPhi}{\widetilde{\Phi}}
\newcommand{\tbPhi}{\widetilde{\bPhi}}
\newcommand{\bLambda}{\boldsymbol{\Lambda}}
\newcommand{\bPsi}{\boldsymbol{\Psi}}

\newcommand{\tbPsi}{\widetilde{\bPsi}}
\newcommand{\bOmega}{\boldsymbol{\Omega}}
\newcommand{\tOmega}{\widetilde{\Omega}}
\newcommand{\tbOmega}{\widetilde{\bOmega}}

\newcommand{\bq}{\boldsymbol{q}}
\newcommand{\bQ}{\boldsymbol{Q}}
\newcommand{\tQ}{\widetilde{Q}}
\newcommand{\tR}{\widetilde{R}}
\newcommand{\tbQ}{\widetilde{\bQ}}
\newcommand{\tbR}{\widetilde{\bR}}
\newcommand{\tbA}{\widetilde{\bA}}
\newcommand{\tbB}{\widetilde{\bB}}
\newcommand{\tK}{\widetilde{K}}
\newcommand{\bK}{\boldsymbol{K}}
\newcommand{\tbK}{\widetilde{\bK}}
\newcommand{\bP}{\boldsymbol{P}}
\newcommand{\bS}{\boldsymbol{S}}

\newcommand{\bD}{\boldsymbol{D}}
\newcommand{\bF}{\boldsymbol{F}}
\newcommand{\tbG}{\widetilde{\bG}}
\newcommand{\tbF}{\widetilde{\bF}}
\newcommand{\bG}{\boldsymbol{G}}
\newcommand{\bd}{\boldsymbol{d}}
\newcommand{\bp}{\boldsymbol{p}}

\newcommand{\bA}{\boldsymbol{A}}

\newcommand{\bB}{\boldsymbol{B}}

\newcommand{\bE}{\boldsymbol{E}}
\newcommand{\bH}{\boldsymbol{H}}
\newcommand{\tH}{\widetilde{H}}

\newcommand{\tbH}{\widetilde{\bH}}
\newcommand{\bR}{\boldsymbol{R}}

\newcommand{\bM}{\boldsymbol{M}}
\newcommand{\br}{\boldsymbol{r}}

\newcommand{\bI}{\boldsymbol{I}}

\newcommand{\bSigma}{\boldsymbol{\Sigma}}
\newcommand{\osigma}{\overline{\sigma}}
\newcommand{\usigma}{\underline{\sigma}}

\newcommand{\ulambda}{\underline{\lambda}}

\newcommand{\bzero}{\boldsymbol{0}}
\newcommand{\st}{\mathop{\text{\normalfont s.t.}}}
\newcommand{\diag}{\mathop{\text{\normalfont diag}}}

\newcommand{\cG}{\mathcal{G}}

\newcommand{\cV}{\mathcal{V}}

\newcommand{\cE}{\mathcal{E}}

\newcommand{\cI}{\mathcal{I}}
\newcommand{\cJ}{\mathcal{J}}
\newcommand{\cT}{\mathcal{T}}

\newcommand{\tz}{\widetilde{z}}
\newcommand{\tA}{\widetilde{A}}
\newcommand{\tB}{\widetilde{B}}

\newcommand{\tx}{\widetilde{x}}
\newcommand{\tu}{\widetilde{u}}
\newcommand{\td}{\widetilde{d}}
\newcommand{\tp}{\widetilde{p}}
\newcommand{\tq}{\widetilde{q}}
\newcommand{\tr}{\widetilde{r}}
\newcommand{\bXi}{\boldsymbol{\Xi}}

\newcommand{\tbx}{\widetilde{\bx}}
\newcommand{\tbz}{\widetilde{\bz}}

\newcommand{\tbu}{\widetilde{\bu}}
\newcommand{\tbq}{\widetilde{\bq}}
\newcommand{\tbr}{\widetilde{\br}}
\newcommand{\tbd}{\widetilde{\bd}}
\newcommand{\tbp}{\widetilde{\bp}}

\usepackage{algorithm}
\usepackage{algpseudocode}

\algdef{SE}[SUBALG]{Indent}{EndIndent}{}{\algorithmicend\ }%
\algtext*{Indent}
\algtext*{EndIndent}

\usepackage[utf8]{inputenc} 
\usepackage[T1]{fontenc}    
\usepackage{hyperref}       
\usepackage{url}            
\usepackage{booktabs}       
\usepackage{amsfonts}       
\usepackage{nicefrac}       
\usepackage{microtype}      

\usepackage{tikz}
\usepackage{ifthen}

\usetikzlibrary{shapes}
\usetikzlibrary{calc}

\usepackage[utf8]{inputenc} 
\usepackage[T1]{fontenc}    
\usepackage{hyperref}       
\usepackage{url}            
\usepackage{booktabs}       
\usepackage{amsfonts}       
\usepackage{nicefrac}       
\usepackage{microtype}      

\allowdisplaybreaks

\newcommand{\rbr}[1]{\left(#1\right)} 
\newcommand{\cbr}[1]{\left\{#1\right\}}

\usepackage{mathtools}

\newcommand{\qedhere}{\hfill\ensuremath{\blacksquare}}
\usetikzlibrary[positioning,arrows,decorations.pathmorphing,backgrounds,fit,calc]

\RRHFirstLine{}
\RRHSecondLine{} 
\LRHFirstLine{}
\LRHSecondLine{} 
\newcommand{\red}[1]{{\color{black}#1}}

\begin{document} 
\TITLE{\Large Near-Optimal Performance of \red{Stochastic Model Predictive Control}}
\ARTICLEAUTHORS{%
\AUTHOR{Sungho Shin\footnote{Most of this work was done while at the Mathematics and Computer Science Division, Argonne National Laboratory.}}
\AFF{Department of Chemical Engineering, Massachusetts Institute of Technology, \EMAIL{sushin@mit.edu}}
\AUTHOR{Sen Na}
\AFF{School of Industrial and Systems Engineering, Georgia Institute of Technology, \EMAIL{senna@gatech.edu}}
\AUTHOR{Mihai Anitescu}
\AFF{
Mathematics and Computer Science Division, Argonne National Laboratory,\\
Department of Statistics, University of Chicago, \EMAIL{anitescu@mcs.anl.gov}
}
}

\ABSTRACT{%
  This article presents a \red{dynamic regret analysis} for \red{stochastic model predictive control (\red{SMPC})} in linear systems with quadratic performance index and additive and multiplicative uncertainties. Under a finite support assumption, the problem can be cast as a finite-dimensional quadratic program, but the problem becomes quickly intractable as the problem size grows exponentially in the horizon length. \red{SMPC} aims to compute approximate solutions by solving a sequence of problems with truncated prediction horizons and committing the solution in a receding-horizon fashion. While this approach is widely used in practice, its performance relative to the optimal solution is not well understood. This article reports for the first time a rigorous near-optimal performance guarantee of \red{SMPC}: Under stabilizability and detectability conditions, the dynamic regret of \red{SMPC} is exponentially small in the prediction horizon length, allowing \red{SMPC} to achieve near-optimal performance at a substantially reduced computational expense.
}%


\maketitle

\section{Introduction}\label{sec:intro}
 
Sequential decision-making under uncertainty is a classical problem studied across several disciplines, including operations research, control theory, and machine learning. It has been studied in different contexts, such as stochastic control \cite{aastrom2012introduction}, robust control \cite{green2012linear}, stochastic programming \cite{birge1985decomposition}, dynamic programming \cite{bertsekas2012dynamic}, and reinforcement learning \cite{bertsekas2019reinforcement}, with different settings, objectives, and strategies to deal with stochasticity. An overarching goal of these studies is to design a decision policy that optimizes the expected performance index of interest over a given period of time, subject to stochastic dynamics.

While solving the stochastic sequential decision problems is typically intractable, the problem may reduce to tractable forms in a number of particular situations. One of the well-known tractable cases is the stochastic linear-quadratic regulator, where the dynamics are linear, the performance index is quadratic, and the uncertainty is additive. In such a case, the certainty equivalence principle allows the formulation of a~deterministic equivalent of the stochastic control problem \cite{theil1957note,simon1956dynamic}. This setting can be further generalized to problems with exponential-quadratic performance criteria \cite{jacobson1973optimal}. Another arguably tractable setting is the Markov decision process (MDP) with finite state-action space. It is well known that when the state-action space is reasonably small, the optimal policy can be computed in a tractable manner, by exploiting the stationary nature of the dynamics, via value or policy iterations \cite{bertsekas2012dynamic}. Even when the underlying dynamics are unknown, the optimal policy can be learned through Q-learning \cite{mnih2015human} or policy optimization \cite{schulman2015trust}. Moreover, it has been recently reported that the periodicity in the uncertainty allows to drastically reduce the complexity of multi-stage stochastic programs \cite{kumar2019hierarchical,shapiro2020periodical}.

If such a desirable structure (certainty equivalence, stationarity, or periodicity) is not present, the computation of optimal decisions often becomes intractable. To rigorously account for the uncertainty, one needs to explicitly consider every possible realization of the uncertainty and perform explicit planning within a single optimization~problem. When the support of the uncertainty is finite, or the sample average approximation technique \cite{kleywegt2002sample,shapiro2006complexity} is used, this~problem can be cast as a finite-dimensional optimization problem, called a multistage stochastic program \cite{birge1985decomposition,pereira1991multi,dantzig1993multi,shapiro2011analysis}. This problem is solvable in principle but becomes quickly intractable as the size of the scenario tree grows exponentially in the horizon length. To overcome such intractability, various decomposition strategies, such as nested Benders decomposition \cite{ho1974nested}, progressive hedging \cite{mulvey1991applying,rockafellar1991scenarios}, and stochastic dual dynamic programming \cite{pereira1991multi}, have been investigated. 
In the context of control, the intractability of multistage stochastic programs is typically addressed by means of stochastic model predictive control (\red{SMPC}, also called receding-horizon control, rolling-horizon heuristics, or look-ahead tree policy) \cite{mesbah2016stochastic}. In particular, multi-stage stochastic programming-based \red{SMPC} has been studied along the lines of work on multistage nonlinear model predictive control \cite{lucia2013multi,thangavel2018dual,yu2019advanced,lucia2020stability} and scenario-based model predictive control \cite{de2005stochastic,bernardini2009scenario,bernardini2011stabilizing,schildbach2014scenario}. In these works, the full multistage problem is sought to be approximately solved by using a sequence of multistage stochastic programs with {\it truncated} prediction horizons. In each time stage, a truncated problem is solved, and the solution is actuated in a {\it receding-horizon} fashion---in each stage, only the first-stage decision is actuated, and in the next stage, the problem is formulated with a shifted prediction window. In this way, we attempt to mimic the behavior of optimal policy with a sequence of truncated-horizon predictive decision policies. This method is becoming increasingly popular in different applications (e.g., battery storage \cite{kumar2020stochastic}, smart buildings \cite{pedersen2018investigating}, HVAC system \cite{kumar2018stochastic}, and microgrids \cite{hans2015scenario}), and a specialized numerical solver is recently developed \cite{frison2020hpipm}.

\red{
  \paragraph{Related Work}
  Prior to the advent of \red{SMPC}, the stability and performance properties of receding horizon controllers were studied in the linear-quadratic control setting with additive uncertainties. In~this setting, the stochastic optimal control policy can be represented as an affine function of uncertainty realizations in \cite{goulartInputtostateStabilityRobust2008}. 
  Therein, sufficient conditions for the receding horizon controllers to satisfy~input-to-state stability have been established, leveraging a technique that parameterizes the control~law~as~an affine state feedback.

The stability and bounded average performance properties of the \red{SMPC} control policy have~primarily been investigated over the past decade. These properties were first analyzed in \cite{chatterjee2014stability}, extending the standard assumptions used in MPC stability and average performance analysis, particularly those concerning the terminal cost and terminal policy that ensure the descent property of the value function. Subsequently, \cite{lorenzenConstraintTighteningStabilityStochastic2017} established asymptotic average performance (i.e., the boundedness of the expected norm of the state variable) and asymptotic stability in probability (i.e., the existence of a region of attraction within which the states remain bounded with high probability)~for~unconstrained~LQ optimal controllers under additive uncertainty.

There has been significant progress in characterizing the stability and bounded average performance properties of nonlinear \red{SMPC} algorithms \cite{mcallisterStochasticRobustnessNominal2023, mcallisterNonlinearStochasticModel2023, mcallisterInherentDistributionalRobustness2024}. The stability results for nonlinear \red{SMPC} were first established in \cite{mcallisterNonlinearStochasticModel2023}, where the authors demonstrated robust asymptotic stability in expectation and bounded average performance.
This essentially guarantees the boundedness of the expected norm of the closed-loop state trajectory, and the bound is expressed in terms of the initial condition and~the covariance of independent and identically distributed (i.i.d.) random disturbances. Furthermore, the stability of \red{SMPC} controllers has been explored in greater depth, revealing that \red{SMPC} solutions are not necessarily robustly asymptotically stable; that is, there is no uniform upper bound on state trajectories for all realizations of uncertainties \cite{mcallisterStochasticRobustnessNominal2023}. Finally, a distributional robustness property~of \red{SMPC} was established in \cite{mcallisterInherentDistributionalRobustness2024}, where state trajectory bounds are expressed in terms of the Wasserstein distance between the true and nominal distributions of the disturbances.

Although stability and bounded average performance properties have been extensively studied in the literature, \textit{the impact of the prediction horizon length on performance has not been sufficiently studied}. More specifically, the performance loss of the \red{SMPC} policy compared to the optimal policy~has not been rigorously quantified. In classical MPC literature, the use of a truncated horizon is often~justified by empirical observations, suggesting that the performance of the MPC~controller~improves~as~the horizon length becomes adequately long \cite{morari1999model}. However, such empirical observations are not rigorously substantiated in \red{SMPC} settings. Unless the exact value function is known and incorporated into the MPC formulation in the form of terminal penalty, the performance of the MPC control policy incurs constant suboptimality at each stage, leading to an overall dynamic regret of $O(T)$. 
Characterizing this suboptimality is important, as most practical MPC policies utilize terminal penalty functions that do not exactly match the optimal value function \cite{mayneStabilizingConditionsModel2019}. Even in the deterministic case, a rigorous characterization of the performance gap caused by truncated horizons has only been made recently \cite{Li2019Online, Na2021Global, Na2020Superconvergence, lin2021perturbation}. This situation motivates us to investigate an important open question: {\it What is the price of truncating the prediction horizon in \red{SMPC}?}

}

We aim to address this question by analyzing the dynamic regret of \red{SMPC}. In particular, we characterize the relationship between the performance loss caused by truncation and the prediction horizon length. Under~mild assumptions, we show that \red{SMPC} policy is exponentially stabilizing in expectation if the prediction horizon is sufficiently long and the performance loss compared with the optimal policy (referred to as {\it dynamic regret}) decays exponentially in the prediction horizon length. This result rigorously substantiates the empirical observation that an \red{SMPC} scheme with a sufficiently long prediction horizon closely approximates the optimal policy. \red{Furthermore, this result reveals that \red{SMPC} can achieve {\it near-optimal} performance, in the sense that one can make its performance exponentially small by controlling the prediction horizon length of \red{SMPC}. }

This paper is close in spirit to the recent work on the \red{dynamic regret analysis} of time-varying DMPC \cite{lin2021perturbation}, which follows the line of work on the regret analysis of linear predictive control \cite{Li2019Online, yu2020power, zhang2021regret}. Under controllability and a positive definite stage cost assumption, Lin et al.~\cite{lin2021perturbation} have proved that DMPC enjoys exponential input-to-state stability and that the dynamic regret decays exponentially with the prediction horizon length. 
\red{
This result establishes a sharper characterization of performance compared to the classical performance studies of DMPC.
}
Prior to \cite{lin2021perturbation}, it was shown in \cite{keerthi1988optimal} that the performance of receding-horizon control approaches that of the infinite-horizon optimal solution for disturbance-free linear-quadratic setting. The performance of nonlinear receding-horizon control has been also studied in a number of different settings \cite{hu2002toward,grune2008infinite,angeli2011average}.

\red{
  The exponential perturbation bound, the key technical tool used in this paper, is related to the well-known Turnpike property in optimal control theory \cite{gruneEconomicModelPredictive2020}. The Turnpike property \cite[Definition 6]{gruneEconomicModelPredictive2020} states that, for a sufficiently long time horizon, the optimal trajectory of the system will remain close to the steady-state solution for most of the time. This property has been widely adopted in the stability analysis of economic MPC \cite{faulwasserEconomicNonlinearModel2018}. The Turnpike property can be considered a special case of the exponential perturbation bound when all disturbances are zero, and the only perturbations arise from the initial state and the terminal cost gradient. Based on this result, one can derive the near-optimality property \cite[Theorem 3]{gruneEconomicModelPredictive2020}. As the Turnpike property is primarily concerned with the unperturbed (zero disturbance) system and focuses solely on the impact of the horizon, this near-optimality result only applies to the unperturbed system. In contrast, our result applies to the general linear-quadratic control setting with additive and multiplicative uncertainties.
}

\paragraph{Contributions}{
Our main contribution is the \red{dynamic regret analysis} of \red{SMPC}. We show that under finite~support, stabilizability, and detectability assumptions, the dynamic regret of the \red{SMPC} policy decays exponentially with the prediction horizon length (Theorem \ref{thm:perf}). In other words, \red{SMPC} can achieve near-optimal performance. Our result generalizes the DMPC performance results of \cite{lin2021perturbation} by allowing the problem formulation to explicitly account for the uncertainty; that is, we do not need to assume that the perfect future information is available.
\red{
The key technical novelty introduced in the proof lies in enabling the analysis of the perturbation bound for optimal control problems with scenario tree structures. To the best of our knowledge, this work is the first to introduce perturbation analysis of the Karush-Kuhn-Tucker (KKT) system embedded within scenario trees.~Specifically,~the uniform perturbation bound for the KKT system is derived by scaling the system according to the probabilities associated with each scenario.
}
Furthermore, we relax the controllability and positive definite cost assumption to stabilizability and detectability assumptions,~while~deriving a performance guarantee that matches the results in \cite{lin2021perturbation}. Our results are obtained by reformulating the stochastic control problem as a multistage stochastic program and leveraging the state-of-the-art perturbation bound for graph-structured optimization problems \cite{demko1977inverses,shin2022exponential,shin2021controllability}.
}

\red{
  \begin{remark}
    This paper presents a dynamic regret analysis for linear-quadratic control settings. However, it is crucial to emphasize that this simple framework primarily serves as a benchmark for analyzing the \red{SMPC} policy, rather than as a practical application. Practical problems related to the aforementioned applications typically involve nonlinear dynamics and/or equality and inequality constraints, which are not addressed within the linear-quadratic control framework. 
    In control theory and related~literature, linear-quadratic control problems have frequently been employed as benchmarks for examining~various theoretical properties of control methods, including stability, sample complexity, regret, and algorithm convergence \cite{lin2021perturbation, linBoundedRegretMPCPerturbation, shin2022near, maniaCertaintyEquivalenceEfficient2019, maniaSimpleRandomSearch2018, deanSampleComplexityLinear2020, luoDynamicRegretMinimization2022, fazelGlobalConvergencePolicy2018}.
    This is primarily because the linear-quadratic setting is analytically tractable and does not require assumptions that are difficult to verify in practice. The simplicity of this framework enables clear and rigorous analysis, providing insights that can often be generalized to more complex situations. For instance, one might anticipate that the performance~results presented in this paper can be generalized to nonlinear systems when they are operated around a small perturbation of the steady state \cite{shinControllabilityObservabilityImply2021}, or extended to inequality-constrained scenarios by assuming suitable controllability conditions for various active sets of inequality constraints \cite{xuExponentiallyAccurateTemporal2018, shinOverlappingSchwarzDecomposition2020}.
    Thus, this work concentrates on the linear-quadratic framework to offer a clear and rigorous analysis of the performance of the \red{SMPC} policy, avoiding the additional complexities associated with nonlinearities and constraints.
    We consider the extension of our results to more general settings (nonlinear and constrained) an intriguing avenue for future research.
  \end{remark}

}

\paragraph{Notation}{
  We denote $a\wedge b = \min(a,b)$ and $a \vee b = \max(a,b)$. The set of real numbers and the set of integers are denoted by $\mathbb{R}$ and $\mathbb{I}$, respectively. The set of symmetric matrices in $\mathbb{R}^{n\times n}$ are denoted by $\mathbb{S}_n$. We define $\mathbb{I}_{A}\coloneqq\mathbb{I}\cap A$, $\mathbb{I}_{\geq 0}\coloneqq\mathbb{I}_{[0,\infty)}$, and $\mathbb{I}_{> 0}\coloneqq\mathbb{I}_{(0,\infty)}$. The identity matrix is denoted by $\bI$, and the zero matrix or vector is denoted by $\bzero$. We use $\Vert\cdot\Vert$ to denote 2-norm for vectors and induced 2-norm for matrices. For matrices $A$ and $B$, $A\succ(\succeq) B$ indicates that $A-B$ is symmetric positive (semi)-definite. We say $A$ is $L$-bounded if $\|A\|\leq L$ and say it is $\gamma$-positive definite if $A\succeq \gamma \bI$ for $\gamma>0$. We use the syntax $[M_1;\cdots;M_n]\coloneqq[M_1^\top\,\cdots\,M_n^\top]^\top$; $\{M_i\}_{i\in \cI}\coloneqq[M_{i_1}; \cdots; M_{i_{m}}]$; $\{M_{i,j}\}_{i\in \cI,j\in \cJ}\coloneqq\{\{M_{i,j}^\top\}_{j\in \cJ}^\top\}_{i\in \cI}$ for $\cI\coloneqq\{i_1<\cdots<i_{m}\}$ and $\cJ\coloneqq\{j_1<\cdots<j_{n}\}$. Also, we use the syntax $M[\cI,\cJ]:=\{M[i,j]\}_{i\in \cI,j\in \cJ}$, where $M[i,j]$ is the $(i,j)$-th component of $M$. For $\{v_t\}_{t\in\mathbb{I}_{[0,T]}}$ and $\{M_{t,t'}\}_{t\in\mathbb{I}_{[0,T]},t'\in\mathbb{I}_{[0,T]}}$, we use the following convention: $\bv_{a:b}\coloneqq\{v_t\}_{t\in\mathbb{I}_{[a,b]}}$; $\bM_{a:b,c:d}\coloneqq\{M_{t,t'}\}_{t\in\mathbb{I}_{[a,b]},t'\in\mathbb{I}_{[c,d]}}$, where $a,b,c,d\in \mathbb{I}_{[0,T]}$. 
}

\section{Settings}\label{sec:setting}

\subsection{Model}\label{sec:setting-1}
We consider a discrete-time {\it stochastic process} $\bxi\coloneqq\{\xi_t\}_{t=0}^T$, where $t$ is the time index, $T\in\mathbb{I}_{>0}$ is the full horizon length, and $\xi_t$ is a random variable taking a value in some measurable set $\Xi$. We consider a discrete-time linear system with additive and multiplicative uncertainties:
\begin{equation}\label{eqn:dyn}
  x_t = f(x_{t-1},u_{t-1};\xi_{t})\coloneqq A(\xi_t) x_{t-1} + B(\xi_t) u_{t-1} + d(\xi_t). 
\end{equation}
Here, $A(\xi_t)$, $B(\xi_t)$, and $d(\xi_t)$ are the random data that take values in $\mathbb{R}^{n_x\times n_x}$, $\mathbb{R}^{n_x\times n_u}$, and $\mathbb{R}^{n_x}$, respectively. Moreover, we consider a stagewise {\it performance index} with additive and multiplicative uncertainties:
\begin{equation}\label{eqn:pi}
  \ell(x_t,u_t;\xi_t)  \coloneqq
  \frac{1}{2}  \begin{bmatrix}
    x_t\\
    u_t
  \end{bmatrix}^\top
  \begin{bmatrix}
    Q(\xi_t)\\
    & R(\xi_t)
  \end{bmatrix}\begin{bmatrix}
    x_t\\
    u_t
  \end{bmatrix} - \begin{bmatrix}
    q(\xi_t)\\
    r(\xi_t)
  \end{bmatrix}^\top
  \begin{bmatrix}
    x_t\\
    u_t
  \end{bmatrix}
\end{equation}
Here, $x_t\in\mathbb{R}^{n_x}$ and $u_t\in\mathbb{R}^{n_u}$ are the state and control variables at time $t$; and $Q(\xi_t)$, $R(\xi_t)$, $q(\xi_t)$, and $r(\xi_t)$ are the random data that take values in $\mathbb{S}_{n_x}$, $\mathbb{S}_{n_u}$, $\mathbb{R}^{n_x}$, and $\mathbb{R}^{n_u}$, respectively. In general, $d(\xi_t)$ in the system \eqref{eqn:dyn} is referred to as {\it disturbance}, and $q(\xi_t)$ and $r(\xi_t)$ in the index \eqref{eqn:pi} are referred to as {\it cost vectors}. The setting is illustrated in Figure \ref{fig:settings}. 

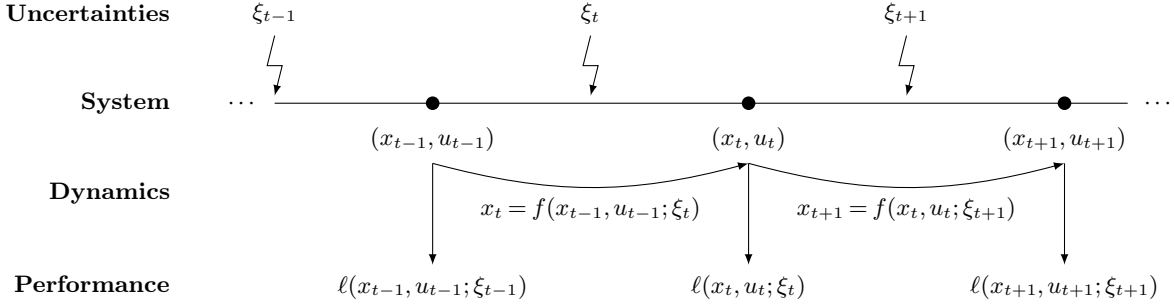
\begin{figure}[t]
  \centering
  \begin{tikzpicture}[font=\footnotesize]
    \def\L{4.2}
    \foreach \x/\l in {-1/t-1,0/t,1/t+1}{
      \node[circle,fill=black,scale=.5,label={[below,yshift=-8pt]:$(x_{\l},u_{\l})$}] at (\L*\x,0) {};\
      \draw[-latex] (\L*\x-\L/2,.9) -- (\L*\x-.1-\L/2,.5) -- (\L*\x+.1-\L/2,.5) -- (\L*\x-\L/2,.1);
      \node at (\L*\x-\L/2,1.2) {$\xi_{\l}$};
      \draw[-latex] (\L*\x,-.8) -- (\L*\x,-2.15) node[below]{$\ell(x_{\l},u_{\l};{\xi_{\l}})$};
    };
    \draw[-latex] (-\L,-.8) [out=-15,in=-165] to node[below] {$x_t=f(x_{t-1},u_{t-1};{ \xi_t})$} (0,-.8)  ;
    \draw[-latex] (0,-.8) [out=-15,in=-165] to node[below] {$x_{t+1}=f(x_{t},u_{t};{ \xi_{t+1}})$} (\L,-.8);
    \draw (-1.5*\L,0) -- (1.2*\L,0);
    \node at (-1.6*\L,0) {$\cdots$};
    \node at (1.3*\L,0) {$\cdots$};
    \node[anchor=east,align=right] at (-1.8*\L,1.2) {\bfseries Uncertainties};
    \node[anchor=east] at (-1.8*\L,-1.2) {\bfseries Dynamics};
    \node[anchor=east] at (-1.8*\L,0) {\bfseries System};
    \node[anchor=east,align=right] at (-1.8*\L,-2.4) {\bfseries Performance};
  \end{tikzpicture} 
  \caption{Illustration of settings}
  \label{fig:settings}
\end{figure}

\red{
\begin{remark}
A key difference between our setting and classical stochastic control is that we~do not impose certain restrictive assumptions about the distributions of uncertainty, which are commonly found in the SMPC literature.
Traditionally, the SMPC literature assumes that $\{\xi_t\}_{t=0}^T$ are mean zero and independent, identically distributed (i.i.d).  
While a nonzero mean can often be integrated into the system dynamics, nonstationary random disturbances cannot be classified as i.i.d random variables. 
The i.i.d assumption is appropriate when the control objective is to track the origin in~the presence of mean-zero disturbances; however, in many applications, the system is affected by exogenous factors that may not be i.i.d.
For instance, in energy management applications, the system must operate in response to time-varying and uncertain exogenous factors such as energy demand and generation costs \cite{chen2022reinforcement}.  
These factors exhibit periodicity but are also significantly affected by~unpredictable~weather conditions, which may only be forecasted up to a certain level of confidence. Our treatment of uncertainties is well-suited for such systems as we do not rely on the i.i.d assumption.

\end{remark}
}

\subsection{Problem Formulation}\label{sec:setting-2}

We assume the following event order:
\begin{equation*}
  \xi_0, x_0, u_0,\xi_1, x_1, u_1,\cdots,\xi_T, x_T, u_T.
\end{equation*}
In each stage, the control decision $u_t$ is made after partially observing the past uncertainty $\bxi_{0:t}$; that is, the control is a {\it recourse} decision \cite{birge2011introduction}. This implies~that $u_t$ can be dependent on $\bxi_{0:t}$, and it is of interest to obtain an optimal decision process $\{u_t(\cdot)\}_{t=0}^T$, where $u_t(\cdot)$ is a function of $\bxi_{0:t}$. Furthermore, we assume that the distribution of $\bxi$ is known. Thus, when the decision is made after observing $\bxi_{0:t}$, the conditional distribution of $\bxi_{t+1:T}$ given $\bxi_{0:t}$ can be taken into account. We denote by $\Xi_t$ and $\bXi_{0:t}$ the support of $\xi_t$ and $\bxi_{0:t}$, respectively. We let $\bXi_{0:t}(\overline{\bxi}_{0:\tau})\coloneqq\{\bxi_{0:t}\in\bXi_{0:t} : \bxi_{0:\tau} = \overline{\bxi}_{0:\tau}\}$ for $\tau\leq t$; $w_t(\bxi_{0:t})\coloneqq[x_t(\bxi_{0:t});u_t(\bxi_{0:t})]$;~$p(\xi_t)\coloneqq[q(\xi_t);r(\xi_t);d(\xi_t)]$; $\cT\coloneqq\mathbb{I}_{[0,T]}$; and $\cT_{a:b}\coloneqq\cT\cap \mathbb{I}_{[a,b]}$.

We now state the problem formulation: 
\begin{subequations}\label{eqn:orig}
  \begin{align} 
    J^\star(\overline\xi_0;\overline{w}_{-1})\coloneqq \min_{\{w_t(\cdot)\}_{t\in\cT}}\;
    & \mathbb{E}_{\bxi}\left[\sum_{t\in\cT}\ell(w_t(\bxi_{0:t});\xi_t)\;\middle|\; \xi_0 = \overline{\xi}_0\right]\\ 
    \st\;\; & x_0(\overline\xi_0) = f(\overline{w}_{-1}; \overline{\xi}_0)\\
    & x_t(\overline\bxi_{0:t}) = f(w_{t-1}(\overline\bxi_{0:t-1}); \overline\xi_{t}),\;\;\forall t\in\cT_{1:T},\;\overline\bxi_{0:t}\in\bXi_{0:t}(\overline{\xi}_0),\label{eqn:orig-con}
  \end{align}
\end{subequations}
where $\overline{w}_{-1}\in\mathbb{R}^{n_x}\times\mathbb{R}^{n_u}$ and $\overline{\xi}_0\in\Xi_{0}$ are given and $w_t:\bXi_{0:t}(\overline{\xi}_0)\rightarrow \mathbb{R}^{n_x}\times \mathbb{R}^{n_u}$. The solution of \eqref{eqn:orig} is denoted by $\{w^\star_t(\cdot;\overline{w}_{-1})\}_{t\in\cT}$ (the existence and uniqueness of the solution will be discussed in Section \ref{sec:perturb}); $w^\star_t(\bxi_{0:t};\overline{w}_{-1})$ is random because it depends on $\bxi_{0:t}$, but it is fixed once $\bxi_{0:t}$ is realized.

Problem \eqref{eqn:orig} seeks to minimize the expected performance index \eqref{eqn:pi} over time horizon $\cT$, while satisfying the stochastic dynamic equation \eqref{eqn:dyn} for all possible realizations of the uncertainty.  Problem \eqref{eqn:orig} does not assume the knowledge of exact uncertainty, but it assumes that {\it the exact distribution} is known. Unlike the deterministic case, in \eqref{eqn:orig} we seek an optimal decision {\it process} $\{w_t(\bxi_{0:t})\}_{t\in\cT}$, rather than an optimal {\it trajectory} $\{w_t\}_{t\in\cT}$. When the support of $\bxi$ is a singleton, the problem in \eqref{eqn:orig} reduces to a deterministic control problem.

In this paper, we analyze the problem in \eqref{eqn:orig} by reformulating the problem into an extensive-form multistage stochastic program under the finite support assumption on $\bxi$ (see Appendix \ref{apx:ext} for details). Unless $\bxi$ has finite support, the finite-dimensional scenario tree cannot be generated, and the extensive form of \eqref{eqn:orig} cannot be obtained. In this case, a sample average approximation strategy can be applied to formulate an approximate problem \cite{kleywegt2002sample,shapiro2006complexity}. In principle, the extensive problem can be solved to optimality, but solving it to optimality is notoriously difficult in most cases because the problem size grows exponentially in $T$ (assuming a fixed number of scenarios per stage). Thus, approximately solving \eqref{eqn:orig} via \red{SMPC} is of interest.


\begin{remark}
  The nature of sequential decision-making under uncertainty prohibits formulating the problem~as a here-and-now or anticipative problem (see \cite{birge2011introduction} for the introduction). The here-and-now formulation enforces $u_t$ are fixed (i.e.,  not dependent on $\bxi_{0:t}$); that is, taking recourse decisions  is not allowed. The here-and-now formulation is often adopted by various \red{SMPC} techniques (see \cite{mesbah2016stochastic} for an overview). While these methods may be able to stabilize the system and achieve respectable performance, their performance is bound to be suboptimal due to the absence of recourse. In contrast, the anticipative formulation allows $u_t$ to be dependent on the full uncertainty $\bxi$. It assumes the availability of the perfect information of $\bxi$ at time $0$. Such a policy cannot be implemented in practice due to nonanticipativity. These settings result in the following expected performances:
  \begin{align*}
    J^{(HN)}(\overline\xi_0;\overline{w}_{-1})
    &\coloneqq
      \begin{aligned}[t]
        \min_{\substack{\{x_t(\cdot)\}_{t\in\cT}\\ \{u_t\}_{t\in\cT}}}\;& \mathbb{E}_{\bxi}\left[\sum_{t\in\cT} \ell(x_t(\bxi_{0:t}),u_t;\xi_{t})\;\middle|\; \xi_0 = \overline{\xi}_0\right]\\
        \st\;& x_0(\overline\xi_{0}) = f(\overline{w}_{-1}; \overline\xi_0)\\
        &x_{t}(\overline\bxi_{0:t})= f(x_{t-1}(\overline\bxi_{0:t-1}),u_{t-1};\overline\xi_{t}),\; \forall t\in\cT_{1:T},\; \overline\bxi_{0:t} \in\bXi_{0:t}(\overline{\xi_0})
      \end{aligned}
    \\
    J^{(AN)}(\overline\xi_0;\overline{w}_{-1})&\coloneqq\mathbb{E}_{\bxi}\left[
                                                \begin{aligned}
                                                  \min_{\{w_t\}_{t\in\cT}}\;& \sum_{t\in\cT} \ell(w_t;\xi_{t}),\\
                                                  \st\;& x_{t}= f(w_{t-1};\xi_{t}),\;  \forall t\in\cT
                                                \end{aligned}
                                                         \;\middle|\; \xi_0 = \overline{\xi}_0,\; w_{-1}=\overline{w}_{-1}\right].
  \end{align*}
  Here, $J^{(HN)}(\overline\xi_0;\overline{w}_{-1})$ and $J^{(AN)}(\overline\xi_0;\overline{w}_{-1})$ denote the expected performance of here-and-now and anticipative policies, respectively. Since the anticipative policy has more flexibility, and here-and-now has less flexibility compared to \eqref{eqn:orig}, we have
  $$J^{(AN)}(\overline\xi_0;\overline{w}_{-1})\leq J^\star(\overline\xi_0;\overline{w}_{-1})\leq J^{(HN)}(\overline\xi_0;\overline{w}_{-1}).$$ Rigorously, the first inequality is due to \cite[(5.22)]{Shapiro2021Lectures}, and the second inequality is due to the fact that the feasible set of the here-and-now problem belongs to \eqref{eqn:orig-con}. 
\end{remark}

\red{
\begin{remark}
The stochastic control problem described in \eqref{eqn:orig} can exhibit a deterministic equivalence when the uncertainties are purely additive; that is, when $A_t$, $B_t$, $Q_t$, $R_t$ are deterministic.~This equivalence stems from the fact that the optimal control policy can be expressed as an affine~function of the realized uncertainties.
Consequently, in the additive uncertainty setting, the affine parameterization approach proposed in \cite{goulartInputtostateStabilityRobust2008} can be employed to compute the control policy in a tractable manner. This approach simplifies both the problem formulation and the computation by eliminating the need for scenario tree formulations. However, it cannot be applied to problems involving multiplicative uncertainties.
\end{remark}

}

\red{
\begin{remark}
One limitation of our analysis is the assumption that the exact distribution of uncertainties is known. In practice, the true distribution of these uncertainties may remain unknown,~highlighting the importance of incorporating distributional robustness into the control policy.
Recent work in a different context has demonstrated stability in expectation within the stochastic MPC~framework when the Wasserstein distance between the true and nominal distributions is bounded \cite{mcallisterInherentDistributionalRobustness2024}.~Nevertheless, the dynamic regret analysis regarding the distributional robustness property remains an open question.
While this topic is beyond the scope of this paper, it presents an intriguing direction for future research.
\end{remark}
}

\subsection{Stochastic Model Predictive Control}

In this section, we introduce the stochastic model predictive control (\red{SMPC}) approach for approximately solving \eqref{eqn:orig}. First, we consider a truncated version of Problem \eqref{eqn:orig}:
\begin{subequations}\label{eqn:spc}
  \begin{align}
    J^{(\tau,W)}(\overline{\bxi}_{0:\tau};\overline{w}_{\tau-1})
    \coloneqq
    \min_{\{w_t(\cdot)\}_{t\in\cT_{\tau:\tau+W}}}\;
    & \mathbb{E}_{\bxi}\left[\sum_{t\in\cT_{\tau:\tau+W}}\ell(w_t(\bxi_{0:t});\xi_t)\;\middle|\; \bxi_{0:\tau}= \overline{\bxi}_{0:\tau}\right]\\
    \st\;
    & x_\tau(\overline{\bxi}_{0:\tau}) = f (\overline{w}_{\tau-1};\overline\xi_{\tau}) \label{eqn:spc-initial}\\
    & x_t({\overline\bxi}_{0:t}) = f(w_{t-1}({\overline\bxi}_{0:t-1}); {\overline\xi}_{t}),\; \forall t\in\cT_{\tau+1:\tau+W},\; {\overline\bxi}_{0:t}\in\bXi_{0:t}(\overline{\bxi}_{0:\tau}). \label{eqn:spc-con}
  \end{align}
\end{subequations}
Here, we are at time $\tau$ and assume that $\overline{w}_{\tau-1}\in\mathbb{R}^{n_x}\times \mathbb{R}^{n_u}$ and $\overline{\bxi}_{0:\tau}\in\bXi_{0:\tau}(\overline\xi_0)$ are given; and $w_t:~\bXi_{0:t}(\overline{\bxi}_{0:\tau})\rightarrow \mathbb{R}^{n_x}\times \mathbb{R}^{n_u}$. Problem \eqref{eqn:spc} aims to find a sequence of optimal decision functions in the next $W$ stages that minimizes the conditional expectation of the performance index over the next $W$ stages, given the uncertainties observed up to that time point and subject to the dynamic constraints. Note that if $\tau> T-W$, the effective horizon length is shorter than $W$ (recall the definition of $\cT_{a:b}$).

Let $w^{(\tau,W,\overline\bxi_{0:\tau})}_t (\cdot;\overline{w}_{\tau-1}):\bXi_{0:t}(\overline\bxi_{0:\tau})\rightarrow  \mathbb{R}^{n_x}\times \mathbb{R}^{n_u}$ for  $t\in\cT_{\tau:\tau+W}$ be the solution of \eqref{eqn:spc} (the existence and uniqueness of the solution will be discussed in Section \ref{sec:perturb}). We observe that the domain $\bXi_{0:t}(\overline \bxi_{0:\tau})$ for different $\overline{\bxi}_{0:\tau}\in\bXi_{0:\tau}(\overline\xi_0)$ is disjoint with each other. Thus, we can accordingly define a {\it composite} solution mapping $w^{(\tau,W)}_t (\cdot;\overline{w}_{\tau-1}) :\bXi_{0:t}(\overline\xi_0)\rightarrow  \mathbb{R}^{n_x}\times \mathbb{R}^{n_u} $ for $t\in\cT_{\tau:\tau+W}$ such that
\begin{align*}
  w^{(\tau,W)}_t (\overline{\bxi}_{0:t};\overline{w}_{\tau-1})  = w^{(\tau,W,\overline{\bxi}_{0:\tau})}_t (\overline{\bxi}_{0:t};\overline{w}_{\tau-1}),\quad \forall \overline{\bxi}_{0:t} \in \bXi_{0:t}(\overline\xi_0).
\end{align*}
This is simply achieved by taking the disjoint unions of the domains and preserving the mapping on each domain. For convenience, we will refer to $\{w^{(\tau,W)}_t (\cdot;\overline{w}_{\tau-1})\}_{t\in\cT_{\tau:\tau+W}}$ as the solution of \eqref{eqn:spc} for varying $\overline{\bxi}_{0:\tau}$.

The \red{closed-loop \red{SMPC} policy} is defined recursively by
\begin{equation}\label{eqn:spc-pol}
  w^{\red{(\textrm{cl},W)}}_t(\bxi_{0:t};\overline{w}_{-1}) \coloneqq w^{(t,W)}_t(\bxi_{0:t};w^{\red{(\textrm{cl},W)}}_{t-1}(\bxi_{0:t-1};\overline{w}_{-1})),\quad t\in\cT,
\end{equation}
where $w^{\red{(\textrm{cl},W)}}_{-1}(\bxi_{0:-1};\overline{w}_{-1})=\overline{w}_{-1}$. The policy in \eqref{eqn:spc-pol} can be explained as follows. At time $t$, with the previous state and control $w_{t-1}^{\red{(\textrm{cl},W)}}(\overline{\bxi}_{0:t-1};\overline{w}_{-1})$ and the newly realized uncertainty $\overline{\xi}_t$ on hand, we look ahead $W$ stages by generating every possible realization of $\bxi_{t:t+W}$ given $\bxi_{0:t}=\overline{\bxi}_{0:t}$, and solve the truncated problem formulated in \eqref{eqn:spc}. Then, we obtain the~first step decision $u_{t}^{\red{(\textrm{cl},W)}}(\overline\bxi_{0:t};\overline{w}_{-1})$, with the state simply given by $x_{t}^{\red{(\textrm{cl},W)}}(\overline\bxi_{0:t};\overline{w}_{-1}) \stackrel{\eqref{eqn:spc-initial}}{=}f(w_{t-1}^{\red{(\textrm{cl},W)}}(\overline\bxi_{0:t-1};\overline{w}_{-1}); \xi_t)$, and we proceed to time $t+1$. Thus, only the first decision obtained from \eqref{eqn:spc} is actuated, and in the next stage, the decisions are reoptimized with the shifted horizon and the newly realized uncertainty. The decision at $t$ is dependent on the full history of past uncertainties ${\bxi}_{0:t}$, as it is used to define the conditional distribution of the future uncertainties over the prediction window. The \red{SMPC} scheme is illustrated in Figure \ref{fig:spc}.

\begin{figure}[t]
  \centering
  \begin{tikzpicture}[font=\footnotesize]
    \def\L{1.2}
    \foreach \x in {-6,...,0}{
      \node[circle,fill=black,scale=.3] at (\L*\x,-1.5) {};
    };
    \foreach \x in {2,...,6}{
      \node[circle,fill=black,scale=.3] at (\L*\x,-1.5) {};
    };
    \draw (\L*-6,-1.5) -- (\L*.5,-1.5);
    \draw (\L*1.5,-1.5) -- (\L*6,-1.5);
    \def\x{0.35}
    \foreach \l/\y/\w/\z in {0/0/0/black,1/1/1/black,2/2/2/black,T-W/8/3/black}{
      \fill[color=\z,fill opacity=.2] (-6*\L+\y*\L,-1.5-\x*\w) rectangle + (4*\L,-\x)  node [midway,black,opacity=1] {Prediction window at $t=\l$};
    }
    \draw[<->] (-4*\L,-1.5-\x*3.5) -- node[below] {$W$} (0,-1.5-\x*3.5);
    \node at (1*\L,-1.5-\x*3.5) {$\cdots$};
    \node at (1*\L,-1.5) {$\cdots$};
    \node at (-\L*6,-1.25) {$t=0$};
    \node at (\L*6,-1.25) {$t=T$};
  \end{tikzpicture}
  \caption{Schematic of stochastic model predictive control}
  \label{fig:spc}
\end{figure}
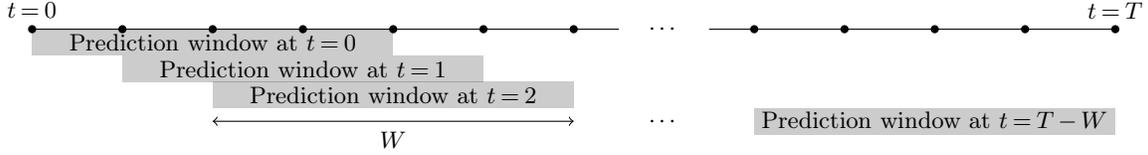

Now we can define the performance index of the \red{SMPC} policy:
\begin{equation}\label{eqn:J-spc}
  J^{(W)}(\overline\xi_0;\overline{w}_{-1})\coloneqq \mathbb{E}_{\bxi}\left[\sum_{t\in\cT} \ell(w^{\red{(\textrm{cl},W)}}_t(\bxi_{0:t};\overline{w}_{-1});\xi_{t}) \;\middle|\;\xi_{0}=\overline{\xi}_0\right].
\end{equation}
Further, the dynamic regret can be defined as follows.
\begin{align*}
  \text{Dynamic regret}\coloneqq J^{(W)}(\overline\xi_0;\overline{w}_{-1})-J^\star(\overline\xi_0;\overline{w}_{-1})
\end{align*}


When formulated as a finite-dimensional multistage stochastic program (see Appendix \ref{apx:ext}), Problem \eqref{eqn:spc} with small $W$ can be significantly smaller than the corresponding formulation of the full problem. The policy in \eqref{eqn:spc-pol} is certainly suboptimal, but we will show in Section \ref{sec:perf} that dynamic regret becomes exponentially small in the prediction horizon length $W$, and thus near-optimal performance can be achieved with moderate $W$.


\section{Main Results}\label{sec:perf}

In this section we establish the performance guarantee of \red{SMPC}. We start by stating the main assumptions: finite support, stabilizability, and detectability. Under these assumptions, we establish the perturbation bound of the open-loop \red{SMPC} policy. This result says that the effect of perturbation decays exponentially in time, which implies that the far-future stages have a small impact on the actuated decisions. The exponential input-to-state stability in expectation (EISSE) of the optimal policy is further obtained~from this result. Moreover, based on the exponential decay result of the open-loop policy, we prove the exponential decay for the closed-loop policy, which finally leads to the EISSE and the performance guarantee of \red{SMPC} policy. Specifically, we show that the dynamic regret of \red{SMPC} policy is exponentially small in the prediction horizon length.

\subsection{Main Assumptions}

We have two main assumptions. The first imposes the condition on the distribution of $\bxi$.

\begin{assumption}\label{ass:fin}
  The distribution of the discrete-time stochastic process $\bxi$ has finite support.
\end{assumption}

This assumption allows for the construction of a finite-dimensional scenario tree that completely describes the evolution of the stochastic process that drives the system dynamics. Thus, we can cast Problem \eqref{eqn:orig} as a finite-dimensional optimization problem (quadratic program, in particular). If the support is infinite (e.g., the uncertainty has a continuous distribution), one needs to apply the sample-average approximation strategy to construct a sampled scenario tree. It is well known that for a sufficiently large number of samples, the solution of the sample-average approximated problem can be arbitrarily close to the true optimal solution with a high probability \cite{shapiro2006complexity}.

We now state the second assumption: stabilizability and detectability.
First, we introduce the notion of stability, stabilizability, and detectability for deterministic settings.
These concepts are commonly used in control theory, but we reintroduce them to write out the associated constants explicitly.

\begin{definition}\label{def:stand}
  Given $L>0$ and $\alpha\in(0,1)$, we define the following.
  \begin{enumerate}[label=(\alph*), leftmargin=1.5em, labelsep = 0.25em]
  \item\label{def:stand-stab} (Stability) A square matrix $\Phi$ is $(L,\alpha)$-stable if $\|\Phi^t\|\leq L \alpha^{t}$ for any $t\in\mathbb{I}_{\geq 0}$.
  \item\label{def:stand-stabil} (Stabilizability) A matrix pair $(A,B)$ is $(L,\alpha)$-stabilizable if $\exists L$-bounded $K$ s.t. $A-BK$ is $(L,\alpha)$-stable.
  \item\label{def:stand-detect} (Detectability) A matrix pair $(A,C)$ is $(L,\alpha)$-detectable if $\exists L$-bounded $K$ s.t. $A-KC$ is $(L,\alpha)$-stable. 
  \end{enumerate}
\end{definition}

We note that the stabilizability and detectability concepts relax the controllability and observability concepts, respectively \cite{anderson2007optimal}. We now adapt these to the stochastic setting.

\begin{definition}\label{def:stoch}
  Given $\xi_0=\overline{\xi}_0$, $L>0$, and $\alpha\in(0,1)$,  we define the following.
  \begin{enumerate}[label=(\alph*), leftmargin=1.5em, labelsep = 0.25em]
  \item\label{def:stoch-stab} (Stability) The square random matrices $\{\Phi_t(\bxi_{0:t})\}_{t\in\cT_{1:T}}$ is $(L,\alpha)$-stable if $\|\prod_{t=t'+1}^{t''}\Phi_t(\bxi_{0:t})\|\leq L\alpha^{t''-t'}$ almost surely (a.s.) for all $t',t''\in\cT$ with $t'< t''$.

  \item\label{def:stoch-stabil} (Stabilizability) The random matrices pair $(\{A(\xi_t)\}_{t\in\cT_{1:T}},\{B(\xi_t)\}_{t\in\cT_{1:T}})$ is $(L,\alpha)$-stabilizable if $\exists L$-bounded (a.s.) $\{K_t(\bxi_{0:t})\}_{t\in\cT_{0:T-1}}$ such that $\{A(\xi_t)-B(\xi_t)K_{t-1}(\bxi_{0:t-1})\}_{t\in\cT_{1:T}}$ is $(L,\alpha)$-stable.

  \item\label{def:stoch-detect} (Detectability) The random matrices pair $(\{A(\xi_t)\}_{t\in\cT_{1:T}},\{C(\xi_t)\}_{t\in\cT_{0:T-1}})$ is $(L,\alpha)$-detectable if $\exists L$-bounded (a.s.) $\{K_t(\bxi_{0:t})\}_{t\in\cT_{1:T}}$ such that $\{A(\xi_t)-K_{t}(\bxi_{0:t})C(\xi_{t-1})\}_{t\in\cT_{1:T}}$ is $(L,\alpha)$-stable.
  \end{enumerate}
\end{definition}

Consider a system~$x_t=\Phi_t(\bxi_{0:t})x_{t-1}$, where $\Phi_t(\bxi_{0:t})$ denotes the state transition mapping. The stability condition in Definition \ref{def:stoch}\ref{def:stoch-stab} states that, for any sequence of realizations, the~product of state transition mappings $\Phi_{t''}(\bxi_{0:t''})\cdots \Phi_{t'+2}(\bxi_{0:t'+2})\Phi_{t'+1}(\bxi_{0:t'+1})$ decays exponentially in $t''-t'$. In~other words, the system $x_t=\Phi_t(\bxi_{0:t})x_{t-1}$ converges to zero a.s. Similarly, the stabilizability assumes that for any possible realization of the system matrices, there exists a sequence of state feedback matrices $K_t(\bxi_{0:t})$ that exponentially stabilizes the system. We emphasize that $K_t(\cdot)$ depends only on $\bxi_{0:t}$; that is, we require the system to be stabilizable (detectable) without using future information. Thus, assuming stabilizability and detectability in Definition \ref{def:stoch} is not contradictory to the nonanticipative nature of our stochastic system setting.

A natural question here is: {\it For a given system $(\{A(\xi_t)\}_{t\in\cT_{1:T}},\{B(\xi_t)\}_{t\in\cT_{1:T}})$, can one verify that a stabilizing state feedback sequence $\{K_t(\bxi_{0:t})\}_{t\in\cT_{0:T-1}}$ exists?} One way to empirically verify the argument is as follows: if there already exists a deterministic predictive or feedback controller that stabilizes the system in the face of uncertainty, it can be deduced that the underlying stochastic system is also stabilizable. This observation is relevant for many practical dynamical systems that are affected by uncertainties, which often arise in the context of conventional process control. Alternatively, one can take a rigorous approach. We show in the next proposition that if the stochastic system is sufficiently close to a {\it deterministically} stable/stabilizable/detectable system, the stochastic system is also stable/stabilizable/detectable.

\begin{proposition}\label{prop:just}

  Under Assumption \ref{ass:fin}, the following hold for any $L\geq 1$, $\alpha\in(0,1)$, and $\Delta \coloneqq (\alpha^{1/2}-\alpha)/L$.
  \begin{enumerate}[label=(\alph*), leftmargin=1.5em, labelsep = 0.25em]
  \item\label{prop:just-stab} $\{\Phi_t(\bxi_{0:t})\}_{t\in\cT_{1:T}}$ is $(L,\alpha^{1/2})$-stable if $\Phi$ is $(L,\alpha)$-stable and $\|\Phi - \Phi_t(\bxi_{0:t})\|\leq \Delta$ a.s. for $t\in\cT_{1:T}$.

  \item\label{prop:just-stabil} $(\{A(\xi_t)\}_{t\in\cT_{1:T}},\{B(\xi_t)\}_{t\in\cT_{1:T}})$ is $(L,\alpha^{1/2})$-stabilizable if $(A,B)$ is $(L,\alpha)$-stabilizable, $\|A-A(\xi_t)\|\leq \Delta/2$, and $\|B-B(\xi_t)\|\leq \Delta/2L$ a.s. for $t\in\cT_{1:T}$.

  \item\label{prop:just-detect} $(\{A(\xi_t)\}_{t\in\cT_{1:T}},\{C(\xi_t)\}_{t\in\cT_{0:T-1}})$ is $(L,\alpha^{1/2})$-detectable if $(A,C)$ is $(L,\alpha)$-detectable, $\|A-A(\xi_t)\|\leq \Delta/2$ a.s. for $t\in\cT_{1:T}$, and $\|C-C(\xi_t)\|\leq \Delta/2L$ a.s.  for $t\in\cT_{0:T-1}$.
  \end{enumerate}

\end{proposition}

The proof is deferred to Appendix \ref{apx:just}. The sketch of the proof is as follows. To show Proposition \ref{prop:just}\ref{prop:just-stab}, we observe that the stability margin $1-sr(\Phi)$ of the deterministic system is uniformly bounded below, where $sr(\cdot)$ denotes the spectral radius. This fact implies that the system can endure a certain degree of deviation while remaining stable. We show that, for any possible sequence of realizations, if the deviation is sufficiently small, the product of the state transition mappings still enjoys exponential decay, which directly leads to stability. Proposition \ref{prop:just}\ref{prop:just-stabil} follows from the fact that the deterministically stabilizing feedback allows for making the closed-loop system sufficiently close to a deterministically stable system. Proposition \ref{prop:just}\ref{prop:just-detect} can be proved in a similar manner.

We are now ready to state the second main assumption:

\begin{assumption}\label{ass:main}
  There exist $L\geq 1$, $\alpha\in(0,1)$, and $\gamma\in(0,1]$ such that
  \begin{enumerate}[label=(\alph*), leftmargin=1.5em, labelsep = 0.25em]
  \item\label{ass:main-bdd} $\{A(\xi_t)\}_{t\in\cT},\{B(\xi_t)\}_{t\in\cT},\{Q(\xi_t)\}_{t\in\cT},\{R(\xi_t)\}_{t\in\cT}$ are $L$-bounded a.s.

  \item\label{ass:main-cvx} $\{Q(\xi_t)\}_{t\in\cT}$ are positive semi-definite a.s., and $\{R(\xi_t)\}_{t\in\cT}$ are $\gamma$-positive definite a.s.

  \item\label{ass:main-stab} $(\{A(\xi_t)\}_{t\in\cT_{1:T}},\{B(\xi_t)\}_{t\in\cT_{1:T}})$ is $(L,\alpha)$-stabilizable.

  \item\label{ass:main-detect} $(\{A(\xi_t)\}_{t\in\cT_{1:T}},\{Q(\xi_t)^{1/2}\}_{t\in\cT_{0:T-1}})$ is $(L,\alpha)$-detectable.
  \end{enumerate}
\end{assumption}

Here, to simplify the notation, we use common constants for different matrices (e.g., $\{A(\xi_t)\}_{t\in\cT}$ and~$\{B(\xi_t)\}_{t\in\cT}$ are $L$-bounded a.s.), rather than introducing constants for each bound (e.g., $\{A(\xi_t)\}_{t\in\cT}$ is $L_A$-bounded a.s. and $\{B(\xi_t)\}_{t\in\cT}$ is $L_B$-bounded a.s.). In particular, we consistently use $L$ for upper bounds, $\gamma$ for strictly positive lower bounds, and $\alpha$ for the upper bounds that are strictly less than $1$. We also emphasize that the requirements of $L\geq 1$ and $\gamma\in(0,1]$ are only for simplifying the presentation. When we only have $L,\gamma>0$, our results still hold by letting $L\leftarrow L\vee1$ and $\gamma\leftarrow\gamma\wedge 1$. We also note that, for deterministic problems, $Q\succeq \bzero$, $R\succ \bzero$, $(A,B)$ stabilizability, $(A,Q^{1/2})$ detectability are standard assumptions imposed in the control literature \cite{anderson2007optimal,rawlings2017model}. Therefore, we do not impose extra assumptions on the system property for studying stochastic problems, but just generalize the standard assumptions from deterministic settings. Our results~will~be~expressed in terms of the constants in Assumption \ref{ass:main}, i.e., $L$, $\alpha$, and $\gamma$.

\red{
\begin{remark}
One might consider a set of assumptions weaker than the conditions outlined in Definition \ref{def:stoch} and Assumption \ref{ass:main}, such as stabilizability in expectation.
However, as noted in other literature, stabilizability-like conditions must be imposed in a robust sense (cf. \cite[Assumption 4]{mcallisterNonlinearStochasticModel2023}) to enable the analysis.
\end{remark}
}

\subsection{Perturbation Analysis}\label{sec:perturb}

We now perform a perturbation analysis for the \red{SMPC} policy. The following theorem establishes the~existence and uniqueness of the solution of \eqref{eqn:spc}, and the perturbation bound of open-loop policy $\{w^{(\tau,W)}_t(\cdot;\overline w_{\tau-1})\}_{t\in\cT_{\tau:\tau+W}}$ with respect to the perturbation in the additive uncertainty $\{p(\xi_{t})\}_{t\in\cT}$. Recall that $p(\xi_t)\coloneqq[q(\xi_t);r(\xi_t);d(\xi_t)]$, and $q(\xi_t)$ and $r(\xi_t)$ are the perturbation in the objective, whereas $d(\xi_t)$ is the perturbation in the constraints.

\begin{theorem}[Perturbation Bound (Open-Loop)]\label{thm:decay}
  Under Assumptions \ref{ass:fin} and \ref{ass:main} and given $\overline{w}_{\tau-1}\in\mathbb{R}^{n_x}\times\mathbb{R}^{n_u}$ and $\overline{\xi}_0\in\Xi_0$, there exists a unique solution $\{w^{(\tau,W)}_t(\cdot;\overline{w}_{\tau-1})\}_{t\in\cT_{\tau:\tau+W}}$ of \eqref{eqn:spc} for all $\tau\in\cT$ and $W\geq 0$. Furthermore, for all $t\in\cT_{\tau:\tau+W}$ and $\overline{\bxi}_{0:\tau}\in\bXi_{0:\tau}(\overline\xi_0)$, we have
  \begin{multline*}
    \left\{\mathbb{E}_{\bxi}\left[\left\|w^{(\tau,W)}_t (\bxi_{0:t};\overline{w}_{\tau-1})\right\|^2\;\middle|\; \bxi_{0:\tau}=\overline{\bxi}_{0:\tau}\right]\right\}^{1/2}\\
    \leq c_1 \left(2L\rho^{t-\tau} \|\overline{w}_{\tau-1}\| + \sum_{t'\in\cT_{\tau:\tau+W}} \rho^{|t-t'|}\left\{\mathbb{E}_{\bxi}\left[\| p(\xi_{t'})\|^2 \;\middle|\; \bxi_{0:\tau}=\overline{\bxi}_{0:\tau}\right]\right\}^{1/2}\right),
  \end{multline*}
  where
  \begin{align}\label{eqn:constants}
    \rho & \coloneqq \left(\dfrac{L_{\tbH}^2-\gamma_{\tbH}^2}{L_{\tbH}^2+\gamma_{\tbH}^2}\right)^{1/2}, \quad c_1\coloneqq
           \frac{L_{\tbH}}{\gamma_{\tbH}^2 \rho},\quad L_{\tbH}  \coloneqq 2L+1, \quad  \gamma_{\tbF}\coloneqq\frac{(1-\alpha)^2}{(1+L)^2L^2},\quad
           \gamma_{\tbG}\coloneqq\frac{\gamma(1-\alpha)^2}{2(1+L)^2L^4}, \\
    \gamma_{\tbH} & \coloneqq
                    \left(\frac{2}{\gamma_{\tbG}} + \left(1+\frac{4L_{\tbH}}{\gamma_{\tbG}}+\frac{4L^2_{\tbH}}{\gamma_{\tbG}^2} \right)\frac{L_{\tbH}(1+\overline{\mu}L_{\tbH})}{\gamma_{\tbF}} + \overline{\mu}\right)^{-1},\quad
                    \overline{\mu}:= \frac{2L^2_{\tbH}/\gamma_{\tbG} + \gamma_{\tbG}+  L_{\tbH}}{\gamma_{\tbF}}.\nonumber
  \end{align}
\end{theorem}

The proof is deferred to Appendix \ref{apx:decay}. The proof involves the reformulation of Problem \eqref{eqn:orig} into~a~finite-dimensional multistage program. This formulation enables performing the perturbation analysis in a convenient linear system form. By applying the state-of-the-art perturbation bound on the graph-structured Karush--Kuhn--Tucker system \cite{shin2022exponential} and establishing a connection between Assumption \ref{ass:main} and the uniform regularity conditions, we obtain the desired result.

Theorem \ref{thm:decay} indicates that the perturbation in the far future stages $p(\xi_{t'})$ (with $t'\gg \tau)$ has an exponentially vanishing effect on the current stage decision $w^{(\tau,W)}_\tau(\bxi_{0:\tau};\overline{w}_{\tau-1})$, and conversely the perturbation in the current stage $p(\xi_{\tau})$ has an exponentially small effect on the later stage decisions $w^{(\tau,W)}_{t'}(\bxi_{0:t'};\overline{w}_{\tau-1})$ (with $t'\gg \tau)$. Since the current stage solution is the only actuated decision in \red{SMPC}, it makes intuitive sense that \red{SMPC} can achieve high performance without taking into account far-future time stages. This observation reveals the fact that the perturbation result will play a crucial role in establishing the near-optimality of \red{SMPC} performance.

We now observe that the optimal policy $\{w^\star_t(\cdot;\overline{w}_{-1})\}_{t\in\cT}$ is the same as the open-loop policy $\{w^{(0,T)}_t(\cdot;\overline{w}_{-1})\}_{t\in\cT}$, where the prediction window length is set to the full horizon length $W=T$, and we obtain the following corollary.

\begin{corollary}[Perturbation Bound (Optimal)]\label{cor:decay}
  Under Assumptions \ref{ass:fin} and \ref{ass:main} and given $\overline{w}_{-1}\in\mathbb{R}^{n_x}\times\mathbb{R}^{n_u}$ and $\overline{\xi}_0\in\Xi_0$, there exists a unique solution $\{w^{\star}_t(\cdot;\overline{w}_{-1})\}_{t\in\cT}$ of \eqref{eqn:orig}, and the following holds for all $t\in\cT$:
  \begin{equation}\label{eqn:optimal}
    \left\{\mathbb{E}_{\bxi}\left[\left\|w^{{\star}}_t(\bxi_{0:t};\overline{w}_{-1})\right\|^2\;\middle|\; \xi_0=\overline{\xi}_0\right]\right\}^{1/2}
    \leq c_1 \left(2L\rho^{t} \|\overline{w}_{-1}\| + \sum_{t'\in\cT} \rho^{|t-t'|}\left\{\mathbb{E}_{\bxi}\left[\| p(\xi_{t'})\|^2\;\middle|\; \xi_0=\overline{\xi}_0\right]\right\}^{1/2}\right),
  \end{equation}
  where $\rho, c_1$ are given by \eqref{eqn:constants}.
\end{corollary}

Corollary \ref{cor:decay} establishes the EISSE of the optimal policy. In particular, if we define $D = \max_{t\in\cT} \{\mathbb{E}_{\bxi}[\| p(\xi_{t})\|^2\mid\xi_{0}=\overline{\xi}_{0}]\}^{1/2}$, then \eqref{eqn:optimal} leads to
\begin{equation*}
  \left\{\mathbb{E}_{\bxi}\left[\|w^{\star}_t(\bxi_{0:t};\overline{w}_{-1})\|^2\mid\xi_{0}=\overline{\xi}_{0}\right]\right\}^{1/2}\leq c_1\rbr{2L\rho^{t} \|\overline{w}_{-1}\| + \frac{2}{1-\rho}D}, \quad\quad \forall t\in \cT . \tag{EISSE}
\end{equation*}
That is, $w^{\star}_t(\bxi_{0:t};\overline{w}_{-1})$ have uniformly bounded second moments if all uncertainties~$\{p(\xi_{t'})\}_{t'\in\cT}$ have uniformly bounded second moments. Furthermore, the second moment of $w^{\star}_t(\bxi_{0:t};\overline{w}_{-1})$ forgets the effect of the initial condition exponentially fast. Input-to-state stability (ISS) is a general notion of stability for perturbed dynamical systems \cite{rawlings2017model}, and EISSE is a generalization of ISS for stochastic systems. Our notion of EISSE is close in spirit to robust asymptotic stability in expectation, defined in \cite{mcallisterNonlinearStochasticModel2023}.
\red{
A key difference is that the EISSE condition guarantees exponential convergence, while~the~robust asymptotic stability in expectation discussed in \cite{mcallisterNonlinearStochasticModel2023} ensures only asymptotic convergence.
}

We should mention that, however, the result in Corollary \ref{cor:decay} only guarantees stability in {\it expectation}. We do not have a guarantee that \eqref{eqn:optimal} holds for every scenario. Thus,  a pathological scenario may exist in which the optimal solution is unbounded. While stability in expectation is standard in the literature \cite{mcallisterNonlinearStochasticModel2023,kouvaritakis2010explicit}, a stronger version of stability  (e.g., stable a.s. \cite{primbs2009stochastic}) may be desired. We leave the investigation of this regard to future work.

Building on the perturbation analysis of the open-loop policy, in the next theorem we analyze the perturbation bound of the closed-loop policy defined by \red{SMPC} \eqref{eqn:spc}.

\begin{theorem}[Perturbation Bound (Closed-Loop)]\label{thm:stab}
  Under Assumptions \ref{ass:fin} and \ref{ass:main}, given $\overline{w}_{-1}\in\mathbb{R}^{n_x}\times\mathbb{R}^{n_u}$, $\overline{\xi}_0\in\Xi_0$, and $W\geq \overline{W}$, the following holds for $\{w^{\red{(\textrm{cl},W)}}_t(\cdot;\overline{w}_{-1})\}_{t\in\cT}$ (defined in \eqref{eqn:spc-pol}) and for all $t\in\cT$:
  \begin{equation*}
    \left\{\mathbb{E}_{\bxi}\left[\left\|w^{\red{(\textrm{cl},W)}}_t(\bxi_{0:t};\overline{w}_{-1})\right\|^2\;\middle|\;\xi_0=\overline\xi_0\right]\right\}^{1/2}\leq c_2 \left(2L\rho^{t/2} \|\overline{w}_{-1}\| + \sum_{t'\in\cT} \rho^{|t-t'|/2}\left\{\mathbb{E}_{\bxi}\left[\| p(\xi_{t'})\|^2\;\middle|\;\xi_0=\overline\xi_0\right]\right\}^{1/2}\right),
  \end{equation*}
  where $\rho$ and $c_1$ are defined in \eqref{eqn:constants} and 
  \begin{equation}\label{eqn:constants-2}
    \overline{W}\coloneqq\frac{\log\left((\alpha^{1/2}-\alpha)/4c_1^2L^3 \right)}{2\log\rho},\quad c_2\coloneqq \frac{2c_1^2L}{\rho(1-\rho^{3/2})}.
  \end{equation}
\end{theorem}

The proof is deferred to Appendix \ref{apx:stab}. The sketch of the proof is as follows. Using the open-loop perturbation result in Theorem \ref{thm:decay}, we show that the open-loop \red{SMPC} policy and the optimal policy become exponentially close as the horizon length $W$ increases. Thus, for $W$ large enough, one can show that the exponential~decay in the perturbation bound holds for the closed-loop \red{SMPC} policy, but the decay rate slightly deteriorates $\rho\rightarrow \rho^{1/2}$. We note that the EISSE of \red{SMPC} can be obtained directly from Theorem \ref{thm:stab}. This result is important not only because it guarantees the EISSE of \red{SMPC}, but also because it serves as an important intermediate step for establishing the performance guarantee.

\red{
We mention that a similar property of \red{SMPC} closed-loop trajectory has been explored in recent MPC literature \cite{mcallisterStochasticRobustnessNominal2023, mcallisterNonlinearStochasticModel2023}. Specifically, the bounds on the state trajectories, akin to those in Theorem \ref{thm:stab}, are presented and referred to as robust asymptotic stability in expectation.
Rather than expressing the state trajectories as an exponential function of time, they are represented using $\mathcal{K}$- and $\mathcal{KL}$-functions, which are commonly employed in classical MPC literature.

}

\subsection{Dynamic Regret Analysis}

We now move on to the \red{dynamic regret analysis}. The next theorem establishes the near-optimal performance guarantee of \red{SMPC} by analyzing its dynamic regret.

\begin{theorem}[Dynamic Regret]\label{thm:perf}
  Under Assumptions \ref{ass:fin} and \ref{ass:main}, given $\overline{w}_{-1}\in\mathbb{R}^{n_x}\times\mathbb{R}^{n_u}$, $\overline{\xi}_0\in\Xi_0$, and $W\geq \overline{W}$, the following holds:
  \begin{equation*}
    J^{(W)}(\overline\xi_0;\overline{w}_{-1}) - J^\star(\overline\xi_0;\overline{w}_{-1}) \leq
    \left[c_5 D^2 T  + c_6 D\|\overline{w}_{-1}\|+ c_7 \|\overline{w}_{-1}\|^2\right]\rho^W
  \end{equation*}
  where $J^\star(\overline\xi_0;\overline{w}_{-1})$, $J^{(W)}(\overline\xi_0;\overline{w}_{-1})$, $(\rho, c_1), (\overline{W}, c_2)$ are defined in \eqref{eqn:orig}, \eqref{eqn:J-spc}, \eqref{eqn:constants}, \eqref{eqn:constants-2}, and 
  \begin{align}\label{eqn:cs}
    c_3 & \coloneqq4c_1^2L   \left(\frac{ 2 c_2L}{1-\rho^{1/2}}+\frac{1}{1-\rho}\right),\\\nonumber
    c_4 & \coloneqq 8c_1^2c_2L^3\\\nonumber
    c_5 & \coloneqq c_3\left(\frac{2c_2L}{1-\rho^{1/2}} +c_3 L/2 + 1\right) + \frac{2c_1^2 c_3 L^2}{1-\rho^2}\left( -1 + 2c_3 L + \frac{4}{1-\rho} + \frac{8c_2L}{1-\rho^{1/2}} \right)\\\nonumber
    c_6 & \coloneqq \frac{1}{1-\rho^{1/2}}\Bigg(\frac{2c_2c_4L}{1-\rho^{1/2}} +c_3c_4 L + c_4 + 2c_2c_3L^2\\\nonumber
        &\qquad+ \frac{2c_1^2 L^2}{1-\rho^2}  \left(-c_4 + 2c_3c_4 L + \frac{4c_4}{1-\rho} + \frac{8c_2c_4L}{1-\rho^{1/2}} + 2c_3L(4c_2L + c_4) \right)\Bigg)\\\nonumber
    c_7 & \coloneqq \frac{1}{1-\rho}\left(c_4(2c_2L^2+c_4L/2) + \frac{4c_1^2 c_4 L^3(4c_2L+c_4)}{1-\rho^2}\right)\\\nonumber
    D & \coloneqq \max_{t\in\cT} \left\{\mathbb{E}_{\bxi}\left[\| p(\xi_{t})\|^2\;\middle|\;\xi_0=\overline{\xi}_0\right]\right\}^{1/2}.
  \end{align}
\end{theorem}

The proof is deferred to Appendix \ref{apx:perf}. The sketch of the proof is as follows. We first analyze the stagewise regret. This quantity estimates how much performance compromise is made in each stage by implementing the \red{SMPC} policy instead of the optimal one. By the exponential decay result, we can show that the stagewise regret is $O(\rho^W)$, and this quantity does not grow in $t$ due to the stability result in Theorem \ref{thm:stab}. Then, by summing up the stagewise dynamic regret over the full horizon $\cT$, we can obtain the result in Theorem \ref{thm:perf}.

This result matches the result for DMPC reported in \cite[Theorem 4.2]{lin2021perturbation}. We note that in addition to dynamic regret, \cite{lin2021perturbation} analyzed the competitive ratio $J^{(W)}(\overline\xi_0;\overline{w}_{-1})/J^\star(\overline\xi_0;\overline{w}_{-1})$, and showed that this ratio is $1+O(\rho^W)$. Unfortunately, this type of analysis does not apply to our setting because the competitive ratio is not well defined for our setting; we allow the perturbation to enter not only as disturbances but also as cost vectors, which makes the optimal performance metric indefinite with respect to the perturbations. 

Theorem \ref{thm:perf} indicates that \red{SMPC} can achieve high performance with a moderate length of prediction horizon: it is sufficient for $W$ to be $O(\log(1/\epsilon))$ to achieve $O(\epsilon T)$ dynamic regret, $J^{(W)}(\overline\xi_0;\overline{w}_{-1}) - J^\star(\overline\xi_0;\overline{w}_{-1}) $, and $O(\log(T))$ to achieve $O(1)$ dynamic regret. Thus, \red{SMPC} can achieve near-optimal performance without dealing with the full horizon. We note that the optimal performance index $J^\star(\overline\xi_0;\overline{w}_{-1})$ often grows linearly with $T$, because, as long as $\mathbb{E}_{\bxi}[\|p(\xi_t)\|^2]\neq 0$ for $t\in\cT$, this nonzero additive noise continually perturbs the system. In this sense, $O(\epsilon T)$ dynamic regret is a reasonable performance guarantee.

\red{
Note that the regret becomes exactly zero when $W \geq T$, but the bound in Theorem \ref{thm:perf} does not account for this. Thus, the bound in Theorem \ref{thm:perf} is particularly useful when $W$ is smaller~than~$T$,~which is a natural setup of the MPC problem. Furthermore, as previously mentioned, Theorem \ref{thm:perf} suggests that it is sufficient for $W$ to be $O(\log(1/\epsilon))$ to achieve $O(\epsilon T)$ dynamic regret. Consequently, the tightness of the bound in the large $W$ regime is of less practical importance.  
}

Theorem \ref{thm:perf} also reveals the trade-off between the computational expense and the performance of the \red{SMPC} scheme. The dynamic regret of \red{SMPC} improves exponentially with $W$, but the improved performance comes at the expense of more complex online computations. Thus, one needs to choose a $W$ that appropriately balances the computational expense and the performance. Furthermore, we observe that $\overline{W}\rightarrow \infty$ and $\rho\rightarrow 1$ as $L\rightarrow \infty$, $\alpha\rightarrow 1$, or $\gamma \rightarrow 0$. Therefore, if Assumption \ref{ass:main} is close to being violated, a longer prediction horizon is necessary to make \red{SMPC} stabilizing and achieve near-optimal performance. Moreover, we note that the results in Theorems \ref{thm:decay}, \ref{thm:stab}, and \ref{thm:perf} are independent of the number of supports of $\bxi$. This fact implies that in the case of sample average approximation, the constants $c_1,\cdots,c_7$, and $\rho$ do not deteriorate as the number of samples increases.

\section{Conclusions and Future Work}\label{sec:conc}

Our \red{dynamic regret analysis} indicates that \red{SMPC} is nearly optimal as a decision policy for sequential decision-making under uncertainty.
\red{Our results suggest that the dynamic regret of \red{SMPC} decreases exponentially with the length of the prediction horizon.}
This, in turn, implies that \red{SMPC} can achieve near-optimal performance with a moderate prediction horizon.
Thus, we conclude that \red{SMPC} is an effective strategy for addressing the intractability of stochastic control problems.

However, our results are limited in some aspects, and  important open questions remain:
\begin{itemize}[leftmargin=1.25em, labelsep = 0.5em]
\item For problems with a large number of scenarios, the current multistage formulation may not adequately reduce the computational complexity. {\it Robust horizon approximation} is one of the widely used methods for reducing the complexity \cite{lucia2013multi}, wherein the extensive scenario tree is considered only up to a point called the {\it robust horizon}, and the tree is sparsified afterward. Accordingly, the problem size grows much more slowly. In the future, we propose to investigate the performance of \red{SMPC} with robust-horizon approximation.

\item The current approach can only deal with stochastic control problems with {\it finite horizons}. However, depending on the application, the performance and stability of \red{SMPC} over an {\it infinite horizon} might be of interest. In the future, we propose to study the average performance of the \red{SMPC} scheme in an appropriate infinite-horizon control setting.

\item \red{
Our \red{dynamic regret analysis} assumes that the exact distribution of uncertainties is known in advance.
Distributional robustness concerning stability properties has been established in recent work \cite{mcallisterInherentDistributionalRobustness2024}.
Analyzing \red{dynamic regret} under inexact knowledge of the distribution presents~an~intriguing avenue for future research.

}
\end{itemize}

\APPENDICES
\section{Perturbation Analysis of Extensive Problem}\label{apx:ext}

In this section, we derive a finite-dimensional equivalent of \eqref{eqn:orig}, which we call an extensive problem, and analyze the perturbation bound of that problem. In particular, we formulate \eqref{eqn:orig} as a multistage stochastic program based on an extensive scenario tree. Then, we apply the perturbation analysis result for graph-structured optimization problems to obtain its perturbation bound.

\paragraph{Notation}{
  A {\it scenario tree} $\cG=(\cV,\cE)$ is a finite, rooted, connected acyclic graph (a typical structure is depicted in Figure \ref{fig:scenario-tree}). For each node, the neighbor on the path toward the root is called the {\it parent} node, and the rest of the neighbors are called {\it children} nodes. The nodes without children are called {\it leaves}. A scenario tree is called {\it stage-$T$ scenario tree} if its leaves have a uniform distance $T$ from the root. We say $j$ is a {\it descendant} of $i$ if it is either a child of $i$ or is (recursively) a descendant of any of the children of $i$. We say $i$ is an {\it ancestor} of $j$ if $j$ is a descendant of $i$. Every node is both an ancestor and a descendant of itself. For $i\in\cV$, we denote its parent by $a(i)$, the set of children by $c(i)$, and $c^t(i)\coloneqq c\circ \cdots \circ c (i)$ (repeated $t$ times). We let $\cV_t$  be the set of nodes at stage $t$, and let $\cV_{t':t''} \coloneqq \bigcup_{t\in\cT_{t':t''}}\cV_t$. Further, we let $\cV^{(k)}_t$ and $\cV^{(k)}_{t':t''}$ be the subsets of $\cV_t$ and $\cV_{t':t''}$ whose elements are descendants of $k$. We also denote by $t(j)$ the stage of $j\in\cV$ (the distance from the root node). Consider $\{i_1 \coloneqq i,\cdots,i_n \coloneqq j\}\subseteq \cV$ such that $i_{t-1}=a(i_t)$ for $t=2,\cdots,n$; we denote such a sequence by $i\rightarrow j$ and denote $\{i_2,\cdots,i_n\}$ by $i\rightharpoonup j$. For $\{v_i\}_{i\in\cV}$, we let $\bv_{i\rightarrow j}\coloneqq [v_{i_1};v_{i_2};\cdots;v_{i_n}]$. Moreover, $\prod\bPhi_{i\rightarrow j}\coloneqq \Phi_{i_n}\cdots \Phi_{i_1}$, and $\prod\Phi_{i\rightharpoonup j}\coloneqq \Phi_{i_n}\cdots \Phi_{i_2}$. For $\{v_i\}_{i\in\cV}$ and $\{M_{ij}\}_{i,j\in\cV}$, we let $\bv_{\cV'}\coloneqq \{v_i\}_{i\in\cV'}$ and $\bM_{\cV',\cV''}\coloneqq \{M_{ij}\}_{i\in\cV',\cV''}$, where $\cV',\cV''\subseteq \cV$.
}

\subsection{Scenario Tree}

We now discuss the construction of a scenario tree from a known distribution of uncertainty $\bxi$. The following lemma proves the existence of the scenario tree that completely captures the distribution of $\bxi$.

\begin{proposition}\label{prop:scen}
  Under Assumption \ref{ass:fin} and given $\overline{\xi}_0$, there exist stage-$T$ scenario tree $\cG\coloneqq(\cV,\cE)$ whose root is $0\in\cV$, nodal realizations $\underline{\bxi}\coloneqq\{\underline{\xi}_i\in\Xi\}_{i\in\cV}$, and nodal probabilities $\bpi\coloneqq\{\pi_{i}\}_{i\in\cV}$ such that
  \begin{enumerate}[label=(\alph*), leftmargin=1.5em, labelsep = 0.25em]
  \item\label{prop:scen-a} $\overline\bxi_{0:t}\in\bXi_{0:t}(\overline{\bxi}_{0:\tau}) \iff \exists k\in\cV_{\tau}$and its descendant $j\in\cV^{(k)}_{t}$ such that $ \overline\bxi_{0:\tau} =\underline{\bxi}_{0\rightarrow k}$ and $ \overline\bxi_{0:t} =\underline{\bxi}_{0\rightarrow j}$.

  \item\label{prop:scen-b} $\pi_0=1$, $\bpi>0$, and $\sum_{j\in c(i)}\pi_j = \pi_i$ for all $i\in \cV_{0:T-1}$.

  \item\label{prop:scen-c} $\mathbb{P}[\bxi_{0:t(j)}=\underline{\bxi}_{0\rightarrow j} \mid \bxi_{0:t(k)}=\underline{\bxi}_{0\rightarrow k}] = \pi_{j\mid k} \coloneqq \pi_j / \pi_k$ for $j\in\cV$ and its ancestor $k\in\cV$.
  \end{enumerate}

\end{proposition}
\begin{proof}{Proof.}
  We prove the claim by direct construction. Let
  \begin{subequations}\label{eqn:G}
    \begin{align}
      \label{eqn:G-1}&\cG  \coloneqq( \cV, \cE),\\
      \label{eqn:G-2}&\cV \coloneqq\bigcup_{t\in\cT} \bXi_{0:t}(\overline{\xi}_0), \\
      \label{eqn:G-3}&\cE  \coloneqq\bigcup_{t\in\cT_{1:T}} \{(\overline\bxi_{0:t-1},\overline\bxi_{0:t}) : \overline\bxi_{0:t}\in \bXi_{0:t}(\overline\xi_0)\}, \\
      \label{eqn:G-4}&\forall j (= \overline\bxi_{0:t})\in \cV: \quad\underline{\xi}_j \coloneqq\overline\xi_t,\quad \pi_j \coloneqq \mathbb{P}[\bxi_{0:t} = \overline\bxi_{0:t} \mid \xi_0 = \overline{\xi}_0].
    \end{align}
  \end{subequations}
  We first show that $\cG$ is a stage-$T$ scenario tree. We let $\overline{\xi}_0$ be the root (note that one can later relabel~this~as~$0$). In fact, for any $\overline\bxi_{0:t}\in\cV$, one can see that there exists a unique path $\{\overline\xi_0, \overline\bxi_{0:1}\}, \cdots\{\overline\bxi_{0:t-1},\overline\bxi_{0:t}\}$ from the~root. Thus, ${\cG}$ is acyclic and connected. Furthermore, each leave $\overline\bxi_{0:T}\in\cV$ has a uniform distance $T$ to the root. Moreover, $\cV$ is finite since $\bXi_{0:t}(\overline\xi_0)$ for $t\in\cT$ are finite (cf. Assumption \ref{ass:fin}). Thus, ${\cG}$ is a stage-$T$ scenario~tree.

  We now prove Proposition \ref{prop:scen}\ref{prop:scen-a}. If $\overline\bxi_{0:t}\in\bXi_{0:t}(\overline{\bxi}_{0:\tau})$, there exist $k\coloneqq\overline\bxi_{0:\tau}\in\cV$ and $j\coloneqq\overline\bxi_{0:t}\in\cV$ by \eqref{eqn:G-2}. From \eqref{eqn:G-3}, we can observe that $j$ and $k$ are $t$ and $\tau$ hops apart from the root node, and one can see that $j$ is a descendant of $k$. This implies that $j\in \cV^{(k)}_t$ and $k\in\cV_\tau$. Further, from \eqref{eqn:G-3}-\eqref{eqn:G-4}, one can see that $\overline\bxi_{0:t} = \underline\bxi_{0\rightarrow j}$. Conversely, if $\exists k\in\cV_\tau $ and $j\in\cV^{(k)}_t$ such that $\overline\bxi_{0:t} =\underline{\bxi}_{0\rightarrow j}$, we have from \eqref{eqn:G-3}-\eqref{eqn:G-4} that $j=\overline\bxi_{0:t}$ and $k=\overline\bxi_{0:\tau}$. By \eqref{eqn:G-2}, $\overline\bxi_{0:\tau}\in \bXi_{0:\tau}(\overline\xi_0)$ and $\overline\bxi_{0:t}\in \bXi_{0:t}(\overline\xi_0)$, and therefore, we have $\overline\bxi_{0:t}\in\bXi_{0:t}(\overline{\bxi}_{0:\tau})$.

  Next, we prove Proposition \ref{prop:scen}\ref{prop:scen-b}. One can see that $\pi_0=1$ from \eqref{eqn:G-4}. Since the support of $\bxi$ is finite~(cf. Assumption \ref{ass:fin}), we have $\pi_j>0$ for all $j\in\cV$. By \eqref{eqn:G-2}, we have that $\bxi_{0:t}=\overline\bxi_{0:t}$ for all $\overline\bxi_{0:t}\in\cV_t$ are~disjoint nonempty events, and $\bxi_{0:t}=\overline\bxi_{0:t}$ for all $\overline\bxi_{0:t}\notin\cV_t$ are empty events. Accordingly, 
  \begin{align*}
  \mathbb{P}[\bxi_{0:t}=i \mid \xi_0 = \overline{\xi}_0] = \sum_{j\in c(i)}\mathbb{P}[\bxi_{0:t+1}=j \mid \xi_0 = \overline{\xi}_0],\;\forall i\in\cV_{0:T-1}.
  \end{align*}
From this and the definition of $\{\pi_j\}_{j\in\cV}$ in \eqref{eqn:G-4}, we obtain $\sum_{j\in c(i)}\pi_j = \pi_i$.

  Finally, we prove Proposition \ref{prop:scen}\ref{prop:scen-c}. Since the event $\bxi_{0:t(j)} = \underline\bxi_{0\rightarrow j}$ is a subset of the event $\bxi_{0:t(k)} = \underline\bxi_{0\rightarrow k}$ when $j$ is a descendant of $k$, and since both events are a subset of the event $\xi_{0} = \overline{\xi}_{0}$, we have
  \begin{equation*}
    \mathbb{P}[\bxi_{0:t(j)}=\underline{\bxi}_{0\rightarrow j} \mid \bxi_{0:t(k)}=\underline{\bxi}_{0\rightarrow k}] = \frac{\mathbb{P}[\bxi_{0:t(j)}=\underline{\bxi}_{0\rightarrow j} \mid \xi_0 = \overline{\xi}_0]}{\mathbb{P}[ \bxi_{0:t(k)}=\underline{\bxi}_{0\rightarrow k} \mid \xi_0 = \overline{\xi}_0]} \stackrel{\eqref{eqn:G-4}}{=} \frac{\pi_j}{\pi_k}.
  \end{equation*}
  This completes the proof.
  \qedhere
\end{proof}

Proposition \ref{prop:scen}\ref{prop:scen-a} suggests that the scenario tree fully captures the support of $\bxi$, and Proposition \ref{prop:scen}\ref{prop:scen-c} suggests that the conditional distribution of $\bxi$ can be completely characterized by the scenario tree. In the following, we will assume that the scenario tree $\cG=(\cV,\cE)$, nodal realizations $\underline\bxi$, and nodal probabilities $\bpi$  are automatically given whenever Assumption \ref{ass:fin} is invoked and $\overline{\xi}_0\in\Xi_0$ is given.

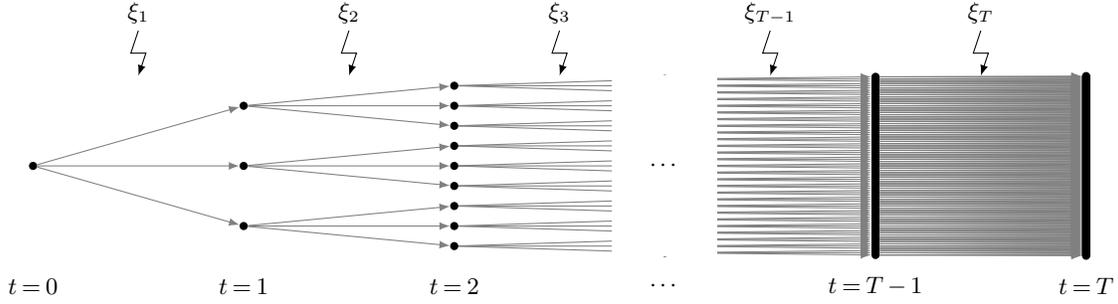
\begin{figure}[t]
  \centering
  \begin{tikzpicture}[font=\footnotesize]
    \def\W{2.8};
    \def\H{.8};
    \foreach \x/\l in {1/1,2/2,3/3,4/T-1,5/T}{
      \draw[-latex] (\W*\x-\W/2,1.5*\H+.6) node[label={[yshift=-5pt]:$\xi_{\l}$}] {} -- (\W*\x-.1-\W/2,1.5*\H+.3) -- (\W*\x+.1-\W/2,1.5*\H+.3) -- (\W*\x-\W/2,1.5*\H);
    };
    \node[circle,fill=black,scale=.3] (A0)  at (0,0) {};
    \foreach \x/\l in {-1/3,0/2,1/1}
    \node[circle,fill=black,scale=.3] (B\x) at (\W,\H*\x) {};
    \foreach \x in {-1,0,1}{
      \foreach \y in {-1,0,1}
      \node[circle,fill=black,scale=.3] (C\x\y) at (2*\W,\H*\x + \H*\y/3) {};}
    \foreach \x in {-1,0,1}{
      \foreach \y in {-1,0,1}{
        \foreach \z in {-1,0,1}
        \node[circle,fill=black,scale=.3] (D\x\y\z) at (3*\W,\H*\x + \H*\y/3 +\H*\z/9){};
      }}
    \foreach \x in {-1,0,1}{
      \foreach \y in {-1,0,1}{
        \foreach \z in {-1,0,1}{
          \foreach \w in {-1,0,1}
          \node[circle,fill=black,scale=.3] (E\x\y\z\w) at (4*\W,\H*\x + \H*\y/3 +\H*\z/9 +\H*\w/27){};
        }}}
    \foreach \x in {-1,0,1}{
      \foreach \y in {-1,0,1}{
        \foreach \z in {-1,0,1}{
          \foreach \w in {-1,0,1}{
            \foreach \v in {-1,0,1}
            \node[circle,fill=black,scale=.3] (F\x\y\z\w\v) at (5*\W,\H*\x + \H*\y/3 +\H*\z/9 +\H*\w/27 +\H*\v/81){};
          }}}}

    \foreach \x in {-1,0,1}
    \draw[gray,-latex] (A0)--(B\x);
    \foreach \x in {-1,0,1}{
      \foreach \y in {-1,0,1}
      \draw[gray,-latex] (B\x)-- (C\x\y);}
    \foreach \x in {-1,0,1}{
      \foreach \y in {-1,0,1}{
        \foreach \z in {-1,0,1}
        \draw[gray,-latex] (C\x\y)--(D\x\y\z);}}
    \foreach \x in {-1,0,1}{
      \foreach \y in {-1,0,1}{
        \foreach \z in {-1,0,1}{
          \foreach \w in {-1,0,1}
          \draw[gray,-latex] (D\x\y\z)--(E\x\y\z\w);}}}
    \foreach \x in {-1,0,1}{
      \foreach \y in {-1,0,1}{
        \foreach \z in {-1,0,1}{
          \foreach \w in {-1,0,1}{
            \foreach \v in {-1,0,1}
            \draw[gray,-latex] (E\x\y\z\w)--(F\x\y\z\w\v);}}}}
    \foreach \x/\l in {0/t=0,1/t=1,2/t=2,3/\cdots,4/t=T-1,5/t=T}{
      \node at (\W*\x,-2*\H) {$\l$};
    };
    \fill[white] (2.75*\W,-1.5*\H) rectangle + (.5*\W,3*\H);
    \node at (\W*3,0) {$\cdots$};
  \end{tikzpicture}
  \caption{Structure of a typical scenario tree.}
  \label{fig:scenario-tree}
\end{figure}

\subsection{Extensive Problem Formulation}

In this section, we derive the extensive problem formulations for \eqref{eqn:orig} and \eqref{eqn:spc}, under Assumption \ref{ass:fin}, given $\xi_0=\overline{\xi}_0\in\Xi_0$ and the scenario tree $\cG=(\cV,\cE)$ with nodal realizations $\underline\bxi$, and nodal probabilities $\bpi$. We define
\begin{subequations}\label{eqn:modeld}
  \begin{align}
    \underline{f}_i(\underline{w}_{a(i)}) &\coloneqq \underline{A}_{i}\underline{x}_{a(i)} +  \underline{B}_{i} \underline{u}_{a(i)} + \underline{d}_{i}, \quad\forall i\in\cV\\
    \underline{\ell}_i(\underline{w}_i) &\coloneqq \begin{bmatrix}
      \underline{x}_i\\
      \underline{u}_i
    \end{bmatrix}^\top\begin{bmatrix}
      \underline{Q}_i\\
      & \underline{R}_i
    \end{bmatrix}\begin{bmatrix}
      \underline{x}_i\\
      \underline{u}_i
    \end{bmatrix} - \begin{bmatrix}
      \underline{q}_i\\
      \underline{r}_i
    \end{bmatrix}^\top\begin{bmatrix}
      \underline{x}_i\\
      \underline{u}_i
    \end{bmatrix}, \quad\forall i\in\cV ,
  \end{align}
\end{subequations}
where $\underline w_i \coloneqq [\underline{x}_i;\underline{u}_i] $ are the state and control variables at node $i$; $ \underline{A}_i\coloneqq A(\underline{\xi}_i)$, $ \underline{B}_i\coloneqq B(\underline{\xi}_i)$, $ \underline{Q}_i\coloneqq Q(\underline{\xi}_i)$, $ \underline{R}_i\coloneqq R(\underline{\xi}_i)$, $ \underline{d}_i\coloneqq d(\underline{\xi}_i)$, $ \underline{q}_i\coloneqq q(\underline{\xi}_i)$, and $ \underline{r}_i\coloneqq r(\underline{\xi}_i)$, for all $i\in\cV$. We recall the definitions of $A(\cdot)$, $B(\cdot)$, $Q(\cdot)$, $R(\cdot)$, etc. from Section \ref{sec:setting-1}. Here, we use the underline notation $\underline\cdot$ to denote the  variables associated with the scenario tree.

We now consider an extensive form multistage stochastic program that corresponds to \eqref{eqn:orig}:
\begin{subequations}\label{eqn:origd}
  \begin{align}
    \underline{J}^\star(\overline{w}_{-1})\coloneqq\min_{\{\underline{w}_i\}_{i\in\cV}}\;
    & \sum_{i\in \cV} \pi_{i}\cdot \underline{\ell}_i(\underline{w}_i)\\
    \st\;\; & \underline{x}_0 = \underline{f}_0(\overline{w}_{-1}) \;\mid\; \underline y_0 \\
    & \pi_{i}\cdot\underline{x}_{i} = \pi_{i}\cdot\underline{f}_i(\underline{w}_{a(i)}) \;\mid\; \underline y_i, \quad i\in\cV_{1:T},\label{eqn:origd-con}
  \end{align}
\end{subequations}
where $\overline{w}_{-1}\in\mathbb{R}^{n_x}\times \mathbb{R}^{n_u}$ is given and $\underline y_{i}$ is the Lagrange multiplier. Problem \eqref{eqn:origd} explicitly considers every possible realization sequence of the uncertainty in a single optimization problem. Thus, solving \eqref{eqn:origd} yields a tree of decisions $\underline\bw^\star\coloneqq\{\underline w^\star_i\}_{i\in\cV}$. We note that the scaling factor $\pi_i$ is introduced in constraint \eqref{eqn:origd-con}, but this scaling does not change the solution in the primal space. We introduce the following short-hand notation: $\underline{z}_i\coloneqq[\underline{w}_i;\underline{ y}_i]$ and $\underline{p}_i\coloneqq [\underline{q}_i;\underline{r}_i;\underline{d}_i]$.

We now consider the extensive formulation that corresponds to the problem in \eqref{eqn:spc}:
\begin{subequations}\label{eqn:spcd}
  \begin{align}
    \underline{J}^{(k,W)}(\overline{w}_{\tau-1})\coloneqq\min_{\{\underline w_i\}_{i\in\cV^{(k)}_{\tau:\tau+W}}}\;
    & \sum_{i\in \cV^{(k)}_{\tau:\tau+W}} \pi_{i|k}\cdot  \underline{\ell}_i(\underline{w}_i) \\
    \st\;\; & x_k = \underline{f}_k(\overline{w}_{\tau-1}) \;\mid\; y_k \label{eqn:spcd-con-1}\\
    & \pi_{i|k}\cdot \underline x_{i} = \pi_{i|k}\cdot \underline{f}_i(\underline w_{a(i)})\;\mid\; \underline y_i, \quad i\in\cV^{(k)}_{\tau+1:\tau+W},
  \end{align}
\end{subequations}
where $\overline{w}_{\tau-1}\in\mathbb{R}^{n_x}\times \mathbb{R}^{n_u}$ and $\overline{\bxi}_{0:\tau}\in\bXi_{0:\tau}(\overline \xi_0)$ are given, and $k\in\cV$ satisfies $\underline{\bxi}_{0\rightarrow k} = \overline\bxi_{0:\tau}$ (accordingly, $\tau=t(k)$; such $k$ exists due to Proposition \ref{prop:scen}\ref{prop:scen-a}). The solution of \eqref{eqn:spcd} is denoted by $\underline\bw^{(k,W)}(\overline{w}_{\tau-1})\coloneqq\{\underline w^{(k,W)}_i (\overline{w}_{\tau-1}) \}_{i\in\cV^{(k)}_{\tau:\tau+W}}$ (the existence and uniqueness of the solution will be discussed in Section \ref{sec:perturbd}). The problem is formulated over a subtree that roots from $k$ and spans the subsequent $W$ layers of the descendant. In Lemma \ref{lem:equiv} we will establish the equivalence~between~\eqref{eqn:spc}~and~\eqref{eqn:spcd}.

The \red{SMPC} policy in \eqref{eqn:spc-pol} can be recursively expressed by:
\begin{equation}\label{eqn:spcd-pol}
  \underline w^{\red{(\textrm{cl},W)}}_k(\overline{w}_{-1})  \coloneqq \underline w^{(k,W)}_k(\underline w^{\red{(\textrm{cl},W)}}_{a(k)}(\overline{w}_{-1}) ),\quad k\in\cV ,
\end{equation}
where $\underline w^{\red{(\textrm{cl},W)}}_{a(0)}= \overline w_{-1}$ and $\overline w_{-1}\in\mathbb{R}^{n_x}\times \mathbb{R}^{n_u}$ is given. Here, note that the dependence on the past uncertainty is implicitly expressed by the node index $k$ (the decision is dependent on $\underline\bxi_{0\rightarrow k}$). Similarly to \eqref{eqn:spc}, the policy is dependent on the full history of past uncertainties, since it is used to define the conditional distribution of the future uncertainties over the prediction window.

\subsection{Probability-Scaled Norms}

As the probability of the realization of each node decays exponentially as it moves toward leaves, the problem in \eqref{eqn:origd} is inherently ill-conditioned, and this poses a challenge in the analysis of perturbation bound. We introduce a custom norm for scenario-tree-indexed vectors and matrices to overcome this challenge. In particular, we scale the nodal quantities by the probability of the associated node being realized. The probability-scaled norm (more precisely, probability-weighted inner product space) was first introduced for the analysis of multistage stochastic programs in \cite{rockafellar1991scenarios}. The definition of the {\it probability-scaled norm} is as follows. 

\begin{definition}\label{def:norm}
  Given a stage-$T$ scenario tree $\cG=(\cV,\cE)$ with nodal realizations $\underline\bxi$ and nodal probabilities $\bpi$, and given the vectors $\bv\coloneqq\{v_i\}_{i\in\cV}$, $\bu\coloneqq\{u_i\}_{i\in\cV}$, and the matrix $\bM\coloneqq\{M_{ij}\}_{i,j\in\cV}$, where $\cV',\cV''\subseteq \cV$, we define the probability-scaled norms as
  \begin{enumerate}[label=(\alph*), leftmargin=1.5em, labelsep = 0.25em]
  \item $\|\bv_{\cV'}\|_{\bpi}\coloneqq\|\{\pi^{1/2}_i v_i\}_{i\in\cV'}\|$;
  \item $\|\bM_{\cV',\cV''}\|_{\bpi}\coloneqq\max\{\|\bM_{\cV',\cV''} \bv_{\cV''}\|_{\bpi}: \|\bv_{\cV''}\|_{\bpi}\leq 1\}$;
  \item $\osigma_{\bpi}(\bM_{\cV',\cV''})\coloneqq\max\{\bu^\top_{\cV'}\bM_{\cV',\cV''} \bv_{\cV''}: \|\bu_{\cV'}\|_{\bpi}\leq 1,\|\bv_{\cV''}\|_{\bpi}\leq 1\}$.
  \end{enumerate}
\end{definition}

Note that $\|\cdot\|_{\bpi}$ with vector arguments is a vector norm, and $\|\cdot\|_{\bpi}$ and $\osigma_{\bpi}(\cdot)$ with matrix arguments are matrix norms (satisfying subadditivity, absolute homogeneity, and positive definiteness). Also, note that the matrix norm $\|\cdot\|_{\bpi}$ is an operator norm induced by the vector norm $\|\cdot\|_{\bpi}$.

In the next proposition, we establish the basic properties of the probability-scaled matrix norms.

\begin{proposition}\label{prop:norm}
  The following holds for $\bM\coloneqq\{M_{ij}\}_{i,j\in\cV}$, $\bM'\coloneqq\{M'_{ij}\}_{i,j\in\cV}$, $\widetilde{\bM}\coloneqq\{(\pi_{i\mid j})^{1/2}M_{ij}\}_{i,j\in\cV}$, $\widehat{\bM}\coloneqq\{(\pi_i\pi_j)^{-1/2}M_{ij}\}_{i,j\in\cV}$, and $\cV',\cV'',\cV'''\subseteq \cV$.
  \begin{enumerate}[label=(\alph*), leftmargin=1.5em, labelsep = 0.25em]
  \item\label{prop:norm-a} $\|\bM_{\cV',\cV''}\|_{\bpi} = \|\widetilde{\bM}_{\cV',\cV''}\|$.
  \item\label{prop:norm-b} $\osigma_{\bpi}(\bM_{\cV',\cV''}) = \|\widehat{\bM}_{\cV',\cV''}\|$.
  \item\label{prop:norm-c} $\|\bM_{\cV',\cV''} \bM'_{\cV'',\cV'''}\|_{\bpi}\leq \|\bM_{\cV',\cV''}\|_{\bpi}\cdot \|\bM'_{\cV'',\cV'''} \|_{\bpi}$.
  \item\label{prop:norm-d} $\sigma_{\bpi}(\bM_{\cV',\cV''}\bM'_{\cV'',\cV'''})\leq \osigma_{\bpi} (\bM_{\cV',\cV''})\cdot \|\bM'_{\cV'',\cV'''}\|_{\bpi} \wedge \osigma_{\bpi} (\bM'_{\cV'',\cV'''})\cdot \|\bM^\top_{\cV',\cV''}\|_{\bpi}  $.
  \end{enumerate}
\end{proposition}

\begin{proof}{Proof.}
  To prove Proposition \ref{prop:norm}\ref{prop:norm-a}, we note that
  \begin{equation*}
    \max_{\|\bv_{\cV''}\|_{\bpi}\leq 1} \|\bM_{\cV',\cV''}\bv_{\cV''}\|_{\bpi} = \max_{\|\{\pi^{1/2}_i v_i\}_{i\in\cV''}\|\leq 1} \left\|\left\{\sum_{j\in\cV''}(\pi_{i\mid j})^{1/2}M_{ij}(\pi^{1/2}_jv_j)\right\}_{i\in\cV'}\right\| = \|\widetilde{\bM}_{\cV',\cV''}\|. 
  \end{equation*}	
  To prove Proposition \ref{prop:norm}\ref{prop:norm-b}, we note that
  \begin{equation*}
    \max_{\substack{\|\bu_{\cV'}\|_{\bpi}\leq 1\\\|\bv_{\cV''}\|_{\bpi}\leq 1}}\bu_{\cV'}^\top \bM_{\cV',\cV''} \bv_{\cV''} = \max_{\substack{\|\{\pi^{1/2}_iu_i\}_{i\in\cV'}\|\leq 1\\\|\{\pi^{1/2}_jv_j\}_{j\in\cV''}\|\leq 1}}  \sum_{i\in\cV',j\in\cV''}(\pi^{1/2}_iu_i )^\top\cbr{(\pi_i\pi_j)^{-1/2} M_{ij}} \pi^{1/2}_j v_j = \|\widehat{\bM}_{\cV',\cV''}\|.
  \end{equation*}
  To prove Proposition \ref{prop:norm}\ref{prop:norm-c}, we note that
  \begin{align*}
    \|\bM_{\cV',\cV''} \bM'_{\cV'',\cV'''}\|_{\bpi}
    &=\left\|\left\{\pi^{1/2}_{i|j}M_{ij}\right\}_{i\in\cV',j\in\cV''}\left\{\pi^{1/2}_{j|k}M'_{jk}\right\}_{j\in\cV'',k\in\cV'''}\right\| \\
    & = \|\widetilde{\bM}_{\cV',\cV''}\widetilde{\bM}'_{\cV'',\cV'''} \| \leq \|\widetilde{\bM}_{\cV',\cV''}\|\|\widetilde{\bM}'_{\cV'',\cV'''} \|
      = \|\bM_{\cV',\cV''}\|_{\bpi} \| \bM'_{\cV'',\cV'''}\|_{\bpi} ,
  \end{align*}
  where $\widetilde{\bM}'\coloneqq\{(\pi_{i\mid j})^{1/2}M'_{ij}\}_{i,j\in\cV}$ and the first and the last equalities follow from Proposition \ref{prop:norm}\ref{prop:norm-a}. Finally, to prove Proposition \ref{prop:norm}\ref{prop:norm-d}, we note that
  \begin{align*}
    \osigma_{\bpi}(\bM_{\cV',\cV''}\bM'_{\cV'',\cV'''} ) &=\left\|\left\{(\pi_{i}\pi_j)^{-1/2}M_{ij}\right\}_{i\in\cV',j\in\cV''}\left\{\pi^{1/2}_{j|k}M'_{jk}\right\}_{j\in\cV'',k\in\cV'''}\right\|\\
                                                         &= \|\widehat{\bM}_{\cV',\cV''}\widetilde{\bM}'_{\cV'',\cV'''} \| \leq \|\widehat{\bM}_{\cV',\cV''}\|\|\widetilde{\bM}'_{\cV'',\cV'''} \|=    \osigma_{\bpi}(\bM_{\cV',\cV''}) \|\bM'_{\cV'',\cV'''}\|,
  \end{align*}
  where the first equality follows from Proposition \ref{prop:norm}\ref{prop:norm-b}. We note that the second part of the result can be obtained by using the transpose-invariant property of $\osigma_{\bpi}(\cdot)$. This completes the proof.
  \qedhere
\end{proof}

The basic statistical property of the probability-scaled norm is established in the following proposition.

\begin{proposition}\label{prop:ev}
  Under Assumption \ref{ass:fin} and given $\overline{\xi}_0\in\Xi_0$, $k\in\cV$, $t_1,t_2\in \cT$ with $t(k)\leq t_1\leq t_2$, $\{v_t(\bxi_{0:t})\}_{t\in\cT_{t_1:t_2}}$, and $\{\underline v_i\}_{i\in\cV^{(k)}_{t_1:t_2}}$, we suppose $v_{t(j)}(\underline\bxi_{0\rightarrow j}) = \underline v_j$ for all $j\in\cV^{(k)}_{t_1:t_2}$. Then, we have
  \begin{equation*}
    \|\underline\bv_{\cV^{(k)}_{t}}\|_{\bpi} = \pi_{k}^{1/2}\left\{\mathbb{E}_{\bxi}\left[\|v_{t}(\bxi_{0:t})\|^2\;\middle|\; \bxi_{0:t(k)}=\underline{\bxi}_{0\rightarrow k}\right]\right\}^{1/2},\quad\forall t\in\cT_{t_1:t_2}.
  \end{equation*}

\end{proposition}

\begin{proof}{Proof.}
  The result can be obtained from
  \begin{align*}
    \|\underline\bv_{\cV^{(k)}_{t}}\|_{\bpi}
    &=\pi^{1/2}_k\left(\sum_{j\in \cV^{(k)}_{t}} \pi_{j|k} \|\underline v_j\|^2\right)^{1/2} = \pi^{1/2}_k\left(\sum_{j\in\cV^{(k)}_{t}} \mathbb{P}[\bxi_{0:t} = \underline\bxi_{0\rightarrow j} \mid \bxi_{0:t(k)} = \underline\bxi_{0\rightarrow k}]\|v_{t}(\underline\bxi_{0\rightarrow j})\|^2 \right)^{1/2}\\
    &=  \pi^{1/2}_k\left\{\mathbb{E}_{\bxi}\left[\|v_{t}(\bxi_{0:t})\|^2\;\middle|\; \bxi_{0:t(k)}=\underline{\bxi}_{0\rightarrow k}\right]\right\}^{1/2},
  \end{align*}
  where the first equality follows from Definition \ref{def:norm}; the second equality follows from $v_t(\bxi_{0:t})=\underline{v}_j$ a.s. given $\bxi_{0:t}=\underline\bxi_{0\rightarrow j}$ (from the assumption in Proposition \ref{prop:ev}), and $\mathbb{P}[\bxi_{0:t} = \underline\bxi_{0\rightarrow j} \mid \bxi_{0:t(k)} = \underline\bxi_{0\rightarrow k}]=\pi_{j\mid k}$ (by Proposition \ref{prop:scen}\ref{prop:scen-c}); and the last equality follows from the fact that the event $\bxi_{0:t(k)}=\underline\bxi_{0\rightarrow k}$ is a disjoint union of the events $\bxi_{0:t}=\underline\bxi_{0\rightarrow j}$ for $j\in \cV^{(k)}_{t}$ (by Proposition \ref{prop:scen}\ref{prop:scen-a}).
  \qedhere
\end{proof}

Proposition \ref{prop:ev} says that the probability-scaled vector norm is fundamentally related to the expected value of the squared norm of the associated random variable. This relationship allows for imposing statistical meaning to the perturbation bound for the extensive problems.

\subsection{Perturbation Analysis of Extensive Problem}\label{sec:perturbd}

In this section, we  study the sensitivity of the {\it primal-dual} solution $\underline\bz^{(k,W)}(\overline{w}_{-1})\coloneqq\{\underline z^{(k,W)}_i(\overline{w}_{-1})\}_{i\in\cV^{(k)}_{t(k):t(k)+W}}$ of \eqref{eqn:spcd} against the perturbation in $\underline\bp \coloneqq \{\underline p_i\}_{i\in\cV}$ under Assumptions \ref{ass:fin} and \ref{ass:main}. Due to the notational complexity~of~\eqref{eqn:spcd}, we simplify the presentation by establishing the perturbation bound for \eqref{eqn:origd} instead. Since \eqref{eqn:origd} and \eqref{eqn:spcd} have the same problem structure (\eqref{eqn:spcd} corresponds to a subtree of \eqref{eqn:origd}), the perturbation results for \eqref{eqn:origd} (Theorem \ref{thm:decay-ext}) can be directly extended to the results for \eqref{eqn:spcd} (Theorem \ref{thm:decay-ext-sub}).

We first discuss the extensive counterpart of Definition \ref{def:stoch}.

\begin{definition}\label{def:stochd}
  Given a stage-$T$ scenario tree $\cG=(\cV,\cE)$ with nodal realizations $\underline\bxi$ and nodal probabilities $\bpi$, we define the following for $L>0$ and $\alpha\in(0,1)$.
  \begin{enumerate}[label=(\alph*), leftmargin=1.5em, labelsep = 0.25em]
  \item\label{def:stochd-stab} (Stability) $\{\underline\Phi_i\}_{i\in\cV_{1:T}}$ is $(L,\alpha)$-stable if $\|\prod\underline{\bPhi}_{i\rightharpoonup j}\|\leq L \alpha^{t(j)-t(i)}$ for all $i\in\cV$ and its descendant $j\in\cV$.
    
  \item\label{def:stochd-stabil} (Stabilizability) $(\{\underline{A}_i\}_{i\in\cV_{1:T}},\{\underline{B}_i\}_{i\in\cV_{1:T}})$ is $(L,\alpha)$-stabilizable if $\exists L$-bounded $\{\underline{K}_i\}_{i\in\cV_{0:T-1}}$ such that $\{\underline{A}_i-\underline{B}_i\underline{K}_{a(i)}\}_{i\in\cV_{1:T}}$ is $(L,\alpha)$-stable.

  \item\label{def:stochd-detect} (Detectability) $(\{\underline{A}_i\}_{i\in\cV_{1:T}},\{\underline C_i\}_{i\in\cV_{0:T-1}})$ is $(L,\alpha)$-detectable if $\exists L$-bounded $\{\underline{K}_i\}_{i\in\cV_{1:T}}$ such that $\{\underline{A}_i-\underline{K}_i\underline C_{a(i)}\}_{i\in\cV_{1:T}}$ is $(L,\alpha)$-stable.

  \end{enumerate}
\end{definition}

Definition \ref{def:stochd} allows for interpreting Assumption \ref{ass:main} in the context of the extensive formulation.

\begin{lemma}\label{lem:stochd}
  Under Assumption \ref{ass:fin} and given $\overline{\xi}_0\in\Xi_0$, the following holds.
  \begin{enumerate}[label=(\alph*), leftmargin=1.5em, labelsep = 0.25em]
  \item\label{lem:stochd:stab} If $\{\Phi_t(\bxi_{0:t}))\}_{t\in\cT_{1:T}}$ is $(L,\alpha)$-stable, then $\{\underline\Phi_i\coloneqq\Phi_t(\underline\bxi_{0\rightarrow i})\}_{i\in\cV_{1:T}}$ is $(L,\alpha)$-stable.

  \item\label{lem:stochd:stabil} If $(\{A(\xi_t)\}_{t\in\cT_{1:T}},\{B(\xi_t)\}_{t\in\cT_{1:T}})$ is $(L,\alpha)$-stabilizable, then $(\{\underline{A}_i\coloneqq A(\underline\xi_{i})\}_{i\in\cV_{1:T}},\{\underline{B}_i\coloneqq B(\underline\xi_{i})\}_{i\in\cV_{1:T}})$ is $(L,\alpha)$-stabilizable.

  \item\label{lem:stochd:detect} If $(\{A(\xi_t)\}_{t\in\cT_{1:T}},\{C(\xi_t)\}_{t\in\cT_{0:T-1}})$ is $(L,\alpha)$-detectable, $(\{\underline{A}_i\coloneqq A(\underline\xi_{i})\}_{i\in\cV_{1:T}},\{\underline C_i\coloneqq C(\underline\xi_{i})\}_{i\in\cV_{0:T-1}})$ is $(L,\alpha)$-detectable.
  \end{enumerate}
\end{lemma}

\begin{proof}{Proof.}
  For Lemma \ref{lem:stochd}\ref{lem:stochd:stab}, if $\{\Phi_t(\bxi_{0:t})\}_{t\in\cT_{1:T}}$ is $(L,\alpha)$-stable, we have $\|\prod_{t=t'+1}^{t''}\Phi_{t}(\overline\bxi_{0:t})\|\leq L\alpha^{t''-t'}$ for any $\overline\bxi_{0:T}\in{\bXi}_{0:T}(\xi_0)$ (cf. Definition \ref{def:stoch}\ref{def:stoch-stab}). This implies that for any $i\in\cV_{t'}$ and its descendant $j\in\cV_{t''}$, we have $\|\prod\underline\bPhi_{i\rightharpoonup j}\|\leq L\alpha^{t''-t'}$ (cf. Proposition \ref{prop:scen}\ref{prop:scen-a}). Thus, $\{\underline\Phi_i\}_{i\in\cV_{1:T}}$ is $(L,\alpha)$-stable.

  For Lemma \ref{lem:stochd}\ref{lem:stochd:stabil}, if $(\{A(\xi_t)\}_{t\in\cT_{1:T}},\{B(\xi_t)\}_{t\in\cT_{1:T}})$ is $(L,\alpha)$-stabilizable, there exists $\{K_t(\bxi_{0:t})\}_{t\in\cT_{0:T-1}}$ such that $\{A(\xi_t)-B(\xi_t)K_{t-1}(\bxi_{0:t-1})\}_{t\in\cT_{1:T}}$ is $(L,\alpha)$-stable (cf. Definition \ref{def:stoch}\ref{def:stoch-stabil}). Letting $\{\underline K_{i}\coloneqq K_{t(i)}(\underline\bxi_{0\rightarrow i}) \}_{i\in\cV_{0:T-1}}$, we immediately have that $\{\underline A_i-\underline B_i\underline K_{a(i)}\}_{i\in\cV_{1:T}}$ is $(L,\alpha)$-stable. Thus, $(\{\underline{A}_i\}_{i\in\cV_{1:T}},\{\underline{B}_i\}_{i\in\cV_{1:T}})$ is $(L,\alpha)$-stabilizable.

  For Lemma \ref{lem:stochd}\ref{lem:stochd:detect}, we can prove it in the same way as Lemma \ref{lem:stochd}\ref{lem:stochd:stabil}; thus, we omit the proof.
  \qedhere
\end{proof}

We are now ready to establish the main perturbation result for \eqref{eqn:origd}. Here, we assume $\overline{w}_{-1}=\bzero$, since the effect of $\overline{w}_{-1}$ can be cast as the perturbation on the initial stage data $\underline d_0 \leftarrow \underline d_0 + \underline A_0 \overline{x}_{-1}+ \underline B_0 \overline{u}_{-1}$.

\begin{theorem}\label{thm:decay-ext}
  Under Assumptions \ref{ass:fin} and \ref{ass:main} and given $\overline{\xi}_0\in\Xi_0$ and $\overline{w}_{-1}=\bzero$, there exists a unique primal-dual solution $\underline\bz^\star\coloneqq\{\underline z^\star_i\}_{i\in\cV}$ of Problem \eqref{eqn:origd}. Further, there exist $\underline\bOmega^\star\coloneqq\{\underline\Omega^\star_{ij}\}_{i,j\in\cV}$ and $\underline\bPsi^\star\coloneqq\{\underline\Psi^\star_{ij}\}_{i,j\in\cV}$  such that $\underline\bz^\star=\underline\bOmega^\star \underline\bp$, $\underline\bw^\star=\underline\bPsi^\star \underline\bp$. Moreover,
  \begin{equation*}
    \|\underline \bOmega^\star_{\cV_t,\cV_{t'}}\|_{\bpi} \vee \|\underline \bPsi^\star_{\cV_t,\cV_{t'}}\|_{\bpi} \leq c_1 \rho^{|t-t'|},\quad \forall t,t'\in\cT,
  \end{equation*}
  where $c_1$ and $\rho$ are defined in \eqref{eqn:constants}.
\end{theorem}

The proof is given later in this section. The sketch of the proof is as follows. We start from the observation that Problem \eqref{eqn:origd} is a graph-structured optimization problem whose structure is given by the scenario~tree. We aim to apply the exponential decay of the perturbation bound established in \cite{shin2022exponential}, but Problem~\eqref{eqn:origd} does not satisfy the {\it uniform regularity} conditions that are required by \cite{shin2022exponential} (cf. Theorem \ref{thm:inv}). To address this issue, we consider an equivalent {\it scaled} version of Problem \eqref{eqn:origd}, where the scaling factor is given by the probability. We show that the scaled problem satisfies the uniform regularity, from which the exponential decay is obtained for the scaled solutions. By undoing the scaling, we effectively replace the standard 2-norms with the probability-scaled norms, and the exponential decay in Theorem \ref{thm:decay-ext} is finally obtained.

Let us first state the exponential decay result for graph-structured optimization problems established in \cite{shin2022exponential}. For our problem, the graph structure is a line, induced by the time horizon. Thus, we state a special case of the original theorem by assuming that the graph is a line.

\begin{theorem}\label{thm:inv}
  Given a matrix $\tbH\coloneqq\begin{bmatrix}\tbG & \tbF^\top\\ \tbF\end{bmatrix}\in\mathbb{S}_{n_{\tbH}}$, we suppose the following holds for some constants $L_{\tbH},\gamma_{\tbG},\gamma_{\tbF}> 0$:
  \begin{enumerate}[label=(\alph*), leftmargin=1.5em, labelsep = 0.25em]
  \item $\|\tbH\|\leq L_{\tbH}$,
  \item $\tbF \tbF^\top \succeq \gamma_{\tbF} I$,
  \item $ReH(\tbG,\tbF)\succeq \gamma_{\tbG} \bI$,
  \end{enumerate}
  where $ReH(\tbG,\tbF)\coloneqq\bZ^\top\tbG\bZ$ with $\bZ$ being the null-space matrix of $\tbF$, i.e., $\bZ$ has orthonormal columns and satisfies $\tbF\bZ = \bzero$. We consider a partition $\{I_i\}_{i=0}^T$ of $\mathbb{I}_{[1,n_{\tbH}]}$. If $\tbH[I_i,I_j]=\bzero$ for any $|i-j| > 1$, then
  \begin{equation*}
    \left\Vert (\tbH^{-1})[I_{i},I_{j}]\right\Vert \leq c_1 \rho^{|i-j|},\quad \forall i,j\in \mathbb{I}_{[0,T]},
  \end{equation*} 
  where $c_1$ and $\rho$ are defined in \eqref{eqn:constants}.
\end{theorem}
\begin{proof}{Proof.}
  See \cite[Theorem 3.6]{shin2022exponential} and \cite[Theorem A.3]{shin2022near}.
  \qedhere
\end{proof}

The three conditions in Theorem \ref{thm:inv} are called the {\it uniform regularity} conditions. Theorem \ref{thm:inv} says that~the uniform regularity conditions are sufficient conditions for exponential decay of the inverse of a graph-induced sparse matrix. Furthermore, the decay bounds are expressed in terms of $L_{\tbH},\gamma_{\tbG},\gamma_{\tbF}$. Thus, in order to establish the exponential decay from Assumption \ref{ass:main}, it suffices to validate the uniform regularity conditions from Assumption \ref{ass:main}. However, the original problem in \eqref{eqn:origd} does not satisfy the uniform regularity conditions, since the probability of each scenario vanishes to zero as $T$ increases; one can easily see that $ReH(\tbG,\tbF)\succeq \gamma_{\tbG} \bI$ condition is violated. We address this issue by considering a {\it scaled} version of \eqref{eqn:origd}, which will be proven to be uniformly regular. The scaled problem is formulated as follows:
\begin{subequations}\label{eqn:spcd-scaled}
  \begin{align}
    \min_{\tbx,\tbu}\;\; & \frac{1}{2}\begin{bmatrix}
      \tbx\\
      \tbu\\
    \end{bmatrix}^\top\begin{bmatrix}
      \tbQ\\
      & \tbR\\
    \end{bmatrix}\begin{bmatrix}
      \tbx\\
      \tbu\\
    \end{bmatrix} - \begin{bmatrix}
      \tbq\\
      \tbr\\
    \end{bmatrix}^\top\begin{bmatrix}
      \tbx\\
      \tbu\\
    \end{bmatrix},\\
    \st\;\;& (\bI-\tbA)\tbx - \tbB\tbu  = \tbd\;\; \mid\; \tby,
  \end{align}
\end{subequations}
where $\tbx\coloneqq\{\tx_i\}_{i\in\cV}$ (similar for $\tbu$, $\tby$, $\tbd$, $\tbq$, $\tbr$), $\tbQ \coloneqq\{\tQ_{ij}\}_{i,j\in\cV}$ (similar for $\tbR$), and $\tbA\coloneqq\{\tA_{ij}\}_{i,j\in\cV}$ (similar for $\tbB$).
Here, $\tx_i$, $\tQ_{ij}$, and $\tA_{ij}$ are (scaled) variables defined as (similar for $\tu_i$, $\ty_i$, $\td_i$, $\tq_i$, $\tr_i$, $\tR_{ij}$, $\tB_{ij}$)
\begin{equation*}
  \tx_i\coloneqq (\pi_{i})^{1/2}\underline x_i,\quad \quad \tQ_{ij} \coloneqq
  \begin{cases}
    \underline Q_{i} & \text{if } i=j, \\
    0&\text{otherwise},
  \end{cases}\quad \quad
  \tA_{ij} \coloneqq
  \begin{cases}
    (\pi_{i\mid j})^{1/2}\underline A_{i}&\text{if } i\in c(j), \\
    0&\text{otherwise}.
  \end{cases}
\end{equation*}
We let $\tbG\coloneqq\diag(\tbQ,\tbR)$ and  $\tbF\coloneqq\begin{bmatrix}\bI-\tbA& -\tbB\end{bmatrix}$. Then, the first-order optimality condition for \eqref{eqn:spcd-scaled} is
\begin{equation}\label{eqn:lin}
  \tbH\tbz = \tbp ,
\end{equation}
where $\tbz\coloneqq\{\tz_i\}_{i\in\cV}$, $\tbp\coloneqq\{\tp_i\}_{i\in\cV}$, $\tbH\coloneqq\{\tH_{ij}\}_{i,j\in\cV}$ with $\tz_i \coloneqq [\tx_i;\tu_i;\ty_i]$, $\tp_i \coloneqq [\tq_i;\tr_i;\td_i]$, and
\begin{equation}\label{eqn:Hij}
  \tH_{ij}\coloneqq\begin{bmatrix} \tQ_{ij} & &\delta_{ij}I - \tA^\top_{ji}\\ & \tR_{ij} & -\tB^\top_{ji}\\ \delta_{ij}I-\tA_{ij}& -\tB_{ij}\end{bmatrix}.
\end{equation}
Here, $\delta_{ij}=1$ if $i=j$ and $0$ otherwise. We note that $\tbH$ can be permuted to $\begin{bmatrix}
  \tbG &\tbF^\top\\
  \tbF
\end{bmatrix}$.

We now validate the uniform regularity conditions for Problem \eqref{eqn:spcd-scaled} based on Assumption \ref{ass:main}.

\begin{lemma}\label{lem:ureg}
  Under Assumptions \ref{ass:fin} and \ref{ass:main} and given $\overline{\xi}_0\in\Xi_0$, the following statements hold with constants $L_{\tbH}$, $\gamma_{\tbF}$, $\gamma_{\tbG}$ defined in \eqref{eqn:constants}:
  \begin{enumerate}[label=(\alph*), leftmargin=1.5em, labelsep = 0.25em]
  \item\label{lem:ureg-ublh} $\|\tbH\|\leq L_{\tbH}$;
  \item\label{lem:ureg-licq} $\tbF\tbF^\top\succeq \gamma_{\tbF} \bI$;
  \item\label{lem:ureg-sosc} $ReH(\tbG,\tbF) \succeq \gamma_{\tbG} \bI$.
  \end{enumerate}
\end{lemma}

To prove Lemma \ref{lem:ureg}, we need a helper lemma.

\begin{lemma}\label{lem:basic}
  Suppose $\tbPhi\coloneqq\{\tPhi_{ij}\}_{i,j\in\cV}$ and $\tPhi_{ij}\coloneqq
  \begin{cases}
    \pi_{i\mid j}^{1/2}\underline \Phi_{i}&\text{if }i\in c(j) \\
    0&\text{otherwise}
  \end{cases}$. The following statements hold.
  \begin{enumerate}[label=(\alph*), leftmargin=1.5em, labelsep = 0.25em]
  \item\label{lem:basic-bdd} If $\{\underline \Phi_i\}_{i\in\cV_{1:T}}$ is $L$-bounded, then $\tbPhi$ is $L$-bounded.
  \item\label{lem:basic-stab} If $\{\underline \Phi_i\}_{i\in\cV_{1:T}}$ is $(L,\alpha)$-stable, then $(\bI-\tbPhi)^{-1}$ is $L/(1-\alpha)$-bounded.
  \end{enumerate}
\end{lemma}

\begin{proof}{Proof.}

  For Lemma \ref{lem:basic}\ref{lem:basic-bdd}, we note that the nonzero blocks of $\tbPhi$ consist of $\tbPhi_{c(j),j}$ for $j\in\cV_{0:T-1}$. For a fixed $j\in\cV_{0:T-1}$, we have
  \begin{equation*}
    \|\tbPhi_{c(j),j}\|\leq \left(\sum_{i\in c(j)} \|\tPhi_{ij}\|^2\right)^{1/2} = \left(\sum_{i\in c(j)} \pi_{i \mid j}\cdot \|\underline\Phi_{i}\|^2\right)^{1/2}\leq L,
  \end{equation*}
  where the first inequality follows from the property of the induced 2-norm, the equality follows from the definition of $\tbPhi$, and the last inequality follows from $L$-boundedness of $\underline\Phi_i$ and $\sum_{i\in c(j)}\pi_i = \pi_j$ (cf. Proposition \ref{prop:scen}\ref{prop:scen-b}). Since the nonzero blocks form a block-subdiagonal structure, we have
  \begin{equation*}
    \|\tbPhi\|\leq \max_{j\in\cV_{0:T-1}}\|\tbPhi_{c(j),j}\|\leq L.
  \end{equation*}
  Thus, we know $\tbPhi$ is $L$-bounded.

  For Lemma \ref{lem:basic}\ref{lem:basic-stab}, we let $\tbPhi^t=\{(\tbPhi^t)_{ij}\}_{i,j\in\cV}$ where
  \begin{equation}\label{eqn:tbPhit}
    (\tbPhi^t)_{ij} =
    \begin{cases}
      (\pi_{i\mid j})^{1/2}\prod\underline\bPhi_{j\rightharpoonup i} & \text{if } i\in c^t(j), \\
      0 & \text{otherwise}.
    \end{cases}
  \end{equation}
  That is, $\tbPhi^t$ has a block-subdiagonal structure whose nonzero blocks consist of $(\tbPhi^t)_{c^t(j),j}$ over $j\in\cV_{0:T-t}$. Furthermore,
  \begin{equation*}
    \|(\tbPhi^t)_{c^t(j),j}\|\leq \left( \sum_{i\in c^t(j)}\|(\tbPhi^t)_{ij}\|^2\right)^{1/2} \stackrel{\eqref{eqn:tbPhit}}{=} \left(\sum_{i\in c^t(j)} \pi_{i\mid j} \cdot \left\|\prod\underline\bPhi_{j\rightharpoonup i}\right\|^2\right)^{1/2} \leq L\alpha^t,
  \end{equation*}
  where the first inequality follows from the property of induced 2-norm, and the last inequality follows from the $(L,\alpha)$-stability of $\{\underline\Phi_{i}\}_{i\in\cV_{1:T}}$ and $\sum_{i\in c^t(j)}\pi_i = \pi_j$ (cf. Proposition \ref{prop:scen}\ref{prop:scen-b}). This implies $\|\tbPhi^t\|\leq L\alpha^t$. Using $(\bI-\tbPhi)^{-1}=\bI + \tbPhi + \cdots \tbPhi^{T}$, we complete the proof by observing
  \begin{align*}
    \left\|(\bI-\tbPhi)^{-1}\right\|\leq \left\|\bI\right\| + \left\|\tbPhi\right\| + \cdots \left\|\tbPhi^{T}\right\|\leq  1  + L\alpha + \cdots L\alpha^T\leq \frac{L}{1-\alpha}.
  \end{align*}
  Here, the second inequality follows from $L\geq 1$. 
  \qedhere
\end{proof}

Now we are ready to prove Lemma \ref{lem:ureg}.

\begin{proof}{Proof of Lemma \ref{lem:ureg}.}

  For Lemma \ref{lem:ureg}\ref{lem:ureg-ublh}, we first observe that $\tbH$ can be permuted to the form of
  \begin{equation*}
    \begin{bmatrix}
      \tbQ&&\bI-\tbA^\top\\
      &\tbR&-\tbB^\top\\
      \bI-\tbA&-\tbB
    \end{bmatrix}.
  \end{equation*}
  By the block diagonality and Assumption \ref{ass:main}\ref{ass:main-bdd}, $\tbQ$ and $\tbR$ are $L$-bounded. By Assumption \ref{ass:main}\ref{ass:main-bdd} and Lemma \ref{lem:basic}\ref{lem:basic-bdd}, $\tbA$ and $\tbB$ are $L$-bounded. The result in Lemma \ref{lem:ureg}\ref{lem:ureg-ublh} follows from
  \begin{equation*}
    \|\tbH\|\leq
    \left\|\begin{bmatrix} \tbQ\\&&-\tbB^\top\\&-\tbB \end{bmatrix}\right\| +
    \left\|\begin{bmatrix} &&\bI\\&\\\bI \end{bmatrix}\right\|+ 
    \left\|\begin{bmatrix} &&-\tbA^\top\\&\tbR\\-\tbA \end{bmatrix}\right\|
    \leq 2L +1.
  \end{equation*}
  For Lemma \ref{lem:ureg}\ref{lem:ureg-licq}, we define $\tbK\coloneqq\{\tK_{ij}\}_{i,j\in\cV}$ where 
  $\tK_{ij} =
  \begin{cases}
    \underline K_{i}&\text{if } i=j\in \cV_{0:T-1} \\
    0&\text{otherwise}\\
  \end{cases}$. Here, $\{\underline K_i\}_{i\in\cV_{0:T-1}}$ is the $(L,\alpha)$-stabilizing feedback for $\{\underline A_i\}_{i\in\cV_{1:T}}$ and $\{\underline B_i\}_{i\in\cV_{1:T}}$, whose existence follows from Assumption \ref{ass:main}\ref{ass:main-stab} and Lemma \ref{lem:stochd}\ref{lem:stochd:stabil}. Now, we observe that $\tbF=\begin{bmatrix}\bI-\tbPhi&-\tbB\\
  \end{bmatrix}\begin{bmatrix}\bI\\\tbK&\bI\\\end{bmatrix}$, where $\tbPhi\coloneqq\{\tPhi_{ij}\}_{i,j\in\cV}$ and
  \begin{equation*}
    \tPhi_{ij} \coloneqq
    \begin{cases}
      (\pi_{i\mid j})^{1/2} (\underline  A_i-\underline  B_i \underline  K_{j}) &\text{if }i\in c(j), \\
      \bzero & \text{otherwise}.
    \end{cases}
  \end{equation*}
  This implies that ($\ulambda(\cdot)$ denotes the smallest eigenvalue of the argument)
  \begin{align}\label{eqn:ureg-licq-1}
    \tbF\tbF^\top & \succeq \ulambda\left( \begin{bmatrix}
        \bI-\tbPhi&-\tbB\\
      \end{bmatrix}
    \begin{bmatrix}
      \bI-\tbPhi&-\tbB\\
    \end{bmatrix}^\top
    \right)	\ulambda\left( \begin{bmatrix}
        \bI\\
        \tbK&\bI\\
      \end{bmatrix}\begin{bmatrix}
        \bI\\
        \tbK&\bI\\
      \end{bmatrix}^\top
    \right)\bI \nonumber\\
                  &\succeq
                    \ulambda\left(
                    (\bI-\tbPhi)(\bI-\tbPhi)^\top + \tbB\tbB^\top
                    \right)
                    \left\|
                    \begin{bmatrix}
                      \bI\\
                      \tbK&\bI\\
                    \end{bmatrix}^{-\top}
    \begin{bmatrix}
      \bI\\
      \tbK&\bI\\
    \end{bmatrix}^{-1}
    \right\|^{-1}\bI \nonumber\\
                  &\succeq
                    \ulambda\left(
                    (\bI-\tbPhi)(\bI-\tbPhi)^\top 
                    \right)
                    \left\|
                    \begin{bmatrix}
                      \bI\\
                      -\tbK&\bI\\
                    \end{bmatrix}
    \right\|^{-2}\bI \nonumber\\
                  &\succeq \|(\bI-\tbPhi)^{-1}\|^{-2} (1 + \|\tbK\|)^{-2}\bI.
  \end{align}
  By the block diagonality and $L$-boundedness of $\{\underline K_i\}_{i\in\cV_{0:T-1}}$, $\tbK$ is $L$-bounded. Since $\{\underline A_i-\underline B_i \underline K_{a(i)}\}_{i\in\cV_{1:T}}$ is $(L,\alpha)$-stable, Lemma \ref{lem:basic}\ref{lem:basic-stab} shows that $(\bI-\tbPhi)^{-1}$ is $L/(1-\alpha)$-bounded. Combining with \eqref{eqn:ureg-licq-1}, we complete the proof for Lemma  \ref{lem:ureg}\ref{lem:ureg-licq}.

  For Lemma \ref{lem:ureg}\ref{lem:ureg-sosc}, we consider any vectors $(\bx,\bu)\neq \bzero$ such that
  \begin{equation}\label{eqn:ureg-ass}
    (\bI-\tbA)\bx+\tbB\bu=\bzero.
  \end{equation}
  Such vectors must exist due to Lemma \ref{lem:ureg}\ref{lem:ureg-licq}. We note that 
  \begin{align*}
    \bx^\top \tbQ\bx + \bu^\top\tbR\bu
    &\geq \bx^\top \tbQ\bx +  \gamma \|\bu\|^2 \\
    &\geq \bx^\top \tbQ\bx +  (\gamma/2L^2)\|\tbB\bu\|^2 + (\gamma/2) \|\bu\|^2\\
    &\stackrel{\mathclap{\eqref{eqn:ureg-ass}}}{\geq} \bx^\top \tbQ\bx +  (\gamma/2L^2)\|(\bI-\tbA)\bx\|^2 + (\gamma/2) \|\bu\|^2\\
    &\geq (\gamma/2L^2) \ulambda(\tbQ+(\bI-\tbA)^\top(\bI-\tbA))\|\bx\|^2+ (\gamma/2)\|\bu\|^2,
  \end{align*}
  where the first inequality follows from Assumption \ref{ass:main}\ref{ass:main-cvx}, the second inequality follows from the fact that $\tbB$ is $L$-bounded (cf Lemma \ref{lem:basic}\ref{lem:basic-bdd}), and the last inequality follows from the property of the smallest eigenvalue and the fact that $L\geq 1$, $\alpha\in(0,1)$, and $\gamma\in(0,1]$.~Furthermore, using $L\geq 1$, $\alpha\in(0,1)$, and $\gamma\in(0,1]$, we can see that, to prove Lemma \ref{lem:ureg}\ref{lem:ureg-sosc}, it suffices to show
  \begin{equation}\label{eqn:ureg-sosc-1}
    \bD^\top\bD\succeq (1-\alpha^2)/(1+L)^2L^2\bI,
  \end{equation}
  where $\bD\coloneqq\begin{bmatrix} \tbQ^{1/2}\\ \bI-\tbA \end{bmatrix}$. Similar to the proof of Lemma \ref{lem:ureg}\ref{lem:ureg-licq}, we let $\tbK\coloneqq\{\tK_{ij}\}_{i,j\in\cV}$ with
  \begin{equation*}
    \tK_{ij} = \begin{cases}
      (\pi_{i\mid j})^{1/2}\underline K_{i}&\text{if } i\in c(j), \\
      0&\text{otherwise}, \\
    \end{cases}
  \end{equation*}
  where $\{\underline K_i\}_{i\in\cV_{1:T}}$ is the $(L,\alpha)$-detectable observer for $(\{\underline A_i\}_{i\in\cV_{1:T}},\{\underline Q_i^{1/2})\}_{i\in\cV_{0:T-1}})$. The existence of $\{\underline K_i\}_{i\in\cV_{1:T}}$ follows from Assumption \ref{ass:main}\ref{ass:main-detect} and Lemma \ref{lem:stochd}\ref{lem:stochd:detect}. We can see that $\bD=\begin{bmatrix}\bI\\-\tbK&\bI\end{bmatrix}\begin{bmatrix}\tbQ^{1/2}\\\bI-\tbPhi\end{bmatrix}$, where $\tbPhi\coloneqq\{\tPhi_{ij}\}_{i,j\in\cV}$ and
  \begin{equation*}
    \tPhi_{ij}=
    \begin{cases}
      (\pi_{i\mid j})^{1/2}(\underline A_i- \underline K_{i}\underline Q^{1/2}_{j})&\text{if }i\in c(j),  \\
      \bzero & \text{otherwise}.
    \end{cases}
  \end{equation*}
  By Lemma \ref{lem:basic}\ref{lem:basic-bdd}, $\|\tbK\|\leq L$. Since $\{\underline A_i-\underline K_{i}\underline Q^{1/2}_{a(i)} \}_{i\in\cV_{1:T}}$ is $(L,\alpha)$-stable, Lemma \ref{lem:basic}\ref{lem:basic-stab} shows that $(\bI-\tbPhi)^{-1}$ is $L/(1-\alpha)$-bounded. Similar to \eqref{eqn:ureg-licq-1}, we can show \eqref{eqn:ureg-sosc-1} holds, and complete the proof.
  \qedhere
\end{proof}

Combining Lemma \ref{lem:ureg} and Theorem \ref{thm:inv}, we can prove Theorem \ref{thm:decay-ext}.

\begin{proof}{Proof of Theorem \ref{thm:decay-ext}.}
  It follows from Lemma \ref{lem:ureg} and \cite[Lemma 16.1]{nocedal1999numerical} that the scaled problem \eqref{eqn:spcd-scaled} has a unique global primal-dual solution $\tbz^\star$, and there exists $\tbOmega^\star =(\tbH)^{-1}$ such that $\tbz^\star = \tbOmega^\star\tbp$. Furthermore, from Theorem \ref{thm:inv} and Lemma \ref{lem:ureg}, we have that $\|\tbOmega^\star_{\cV_t,\cV_{t'}}\|\leq c_1\rho^{|t-t'|}$. Thus, by unscaling the problem, we know there exists a unique global primal-dual solution $\underline\bz^\star$ of \eqref{eqn:origd}; further, from the one-to-one correspondence between the scaled variables ($\tbz$ and $\tbp$) and non-scaled variables ($\underline\bz$ and $\underline\bp$), one can see that the unscaled solution satisfies $\underline\bz^\star = \underline\bOmega^\star\underline\bp$ for $\underline\bOmega^\star\coloneqq\{\pi^{-1/2}_{i|j} \tOmega_{ij}^\star\}_{i,j\in\cV}$. From Proposition \ref{prop:norm}, one can see that $\|\underline\bOmega^\star_{\cV_t,\cV_{t'}}\|_{\bpi}=\|\tbOmega^\star_{\cV_t,\cV_{t'}}\|\leq c_1\rho^{|t-t'|}$.  The result for the primal solution directly follows from the observation that $\underline\bPsi^\star$ is a submatrix of $\underline\bOmega^\star$. 
  \qedhere
\end{proof}

We now adapt Theorem \ref{thm:decay-ext} to Problem \eqref{eqn:spcd}. We note that Problems \eqref{eqn:origd} and \eqref{eqn:spcd} have the same structure, although they are formulated over a different subset of nodes on the scenario tree. Thus, the result in~Theorem \ref{thm:decay-ext} can be directly generalized to Problem \eqref{eqn:spcd}.

\begin{theorem}\label{thm:decay-ext-sub}
  
  Under Assumptions \ref{ass:fin} and \ref{ass:main} and given $\overline{\xi}_0\in\Xi_0$, $\overline{w}_{a(k)}=\bzero$, $k\in\cV$, and $W\geq 0$, there exist a unique primal-dual solution $\underline\bz^{(k,W)}\coloneqq\{\underline z^{(k,W)}_i\}_{i\in\cV^{(k)}_{t(k):t(k)+W}}$ of \eqref{eqn:spcd}, $\underline\bOmega^{(k,W)}\coloneqq\{\underline\Omega^{(k,W)}_{ij}\}_{i,j\in\cV^{(k)}_{t(k):t(k)+W}}$, and $\underline\bPsi^{(k,W)}\coloneqq\{\underline\Psi^{(k,W)}_{ij}\}_{i,j\in\cV^{(k)}_{t(k):t(k)+W}}$ such that
  \begin{equation*}
    \begin{aligned}
      & \underline\bz^{(k,W)}=\underline\bOmega^{(k,W)} \underline\bp_{\cV^{(k)}_{t(k):t(k)+W}}, \\
      & \underline\bw^{(k,W)}=\underline\bPsi^{(k,W)} \underline\bp_{\cV^{(k)}_{t(k):t(k)+W}},\\
      &\|\underline \bOmega^{(k,W)}_{\cV^{(k)}_{t},\cV^{(k)}_{t'}}\|_{\bpi} \vee \|\underline \bPsi^{(k,W)}_{\cV^{(k)}_{t},\cV^{(k)}_{t'}}\|_{\bpi}\leq c_1 \rho^{|t-t'|},\quad \forall t,t'\in\cT_{t(k):t(k)+W},
    \end{aligned}
  \end{equation*}
  where $c_1$ and $\rho$ are defined in \eqref{eqn:constants}. 
\end{theorem}

\section{Proofs}\label{apx:proofs}

\subsection{Proof of Proposition \ref{prop:just}}\label{apx:just}

We first state a helper lemma.

\begin{lemma}\label{lem:help-3}

  For any submultiplicative matrix norm $\|\cdot\|$, we suppose that deterministic sequences $\{\Phi_t\}_{t\in\mathbb{I}_{[t'+1,t'']}}$ and $\{\Phi'_t\}_{t\in\mathbb{I}_{[t'+1,t'']}}$ satisfy $\|\prod_{t=t'+1}^{t''}\Phi_{t}\|\leq L\alpha^{t''-t'}$ and $\|\Phi_t-\Phi'_t\|\leq \Delta$ for $t\in\mathbb{I}_{[t'+1,t'']}$ and
  $\Delta\coloneqq(\alpha^{1/2}-\alpha)/L$. Then, we have $\left\|\prod_{t=t'+1}^{t''} \Phi'_{t}\right\|\leq L\alpha^{(t''-t')/2}$.

\end{lemma}


\begin{proof}{Proof.}
	
  We note that
  \begin{align*}
    \left\|\prod_{t={t'+1}}^{t''}\Phi'_{t}\right\|
    & \leq  \left\|\prod_{t={t'+1}}^{t''} \cbr{(\Phi'_{t} - \Phi_{t}) + \Phi_{t} } \right\|\\
    &\leq \sum_{\tau=0}^{t''-t'} \sum_{\{t_1,\cdots,t_\tau\}\subseteq \mathbb{I}_{[t'+1,t'']}}   
      \|\Phi_{t''}\cdots \Phi_{t_{\tau}+1} \|\cdots
      \| \|\Phi_{t_2-1}\cdots \Phi_{t_1+1}\| \|\Phi'_{t_1}-\Phi_{t_1}\|  \|\Phi_{t_1-1}\cdots \Phi_{t'+1}\|\\
    &\leq \sum_{\tau=0}^{t''-t'} {t''-t' \choose \tau}  \Delta^\tau L^{\tau+1} \alpha^{t''-t'-\tau}\\
    &\leq L\alpha^{t''-t'} \sum_{\tau=0}^{t''-t'}  {t''-t' \choose \tau} (L\Delta/\alpha)^{\tau} \\
    &\leq L\alpha^{t''-t'} (1+L\Delta/\alpha)^{t''-t'}\\
    &\leq L\alpha^{(t''-t')/2},
  \end{align*}
  where the second inequality follows from the binomial expansion and the submultiplicativity of $\|\cdot\|$; the third inequality follows from the assumption that $\|\prod_{t=t'+1}^{t''}\Phi_{t}\|\leq L\alpha^{t''-t'}$ and $\|\Phi_{t}-\Phi'_{t}\|\leq \Delta$; the fourth inequality is obtained by rearrangement; the fifth inequality follows from the binomial theorem; and the last inequality follows from the definition of $\Delta$.
  \qedhere
\end{proof}

\begin{proof}{Proof of Proposition \ref{prop:just}}

  For Proposition \ref{prop:just}\ref{prop:just-stab}, we apply Lemma \ref{lem:help-3} and have $\|\prod_{t=t'+1}^{t''} \Phi_{t}(\bxi_{0:t})\|\leq L\alpha^{(t''-t')/2}$ a.s. Thus, we know $\{\Phi_t(\bxi_{0:t})\}_{t\in\cT_{1:T}}$ is $(L,\alpha^{1/2})$-stable. For Proposition \ref{prop:just}\ref{prop:just-stabil}, we let $K$ be the $(L,\alpha)$-stabilizing feedback for $(A,B)$. Using the facts that (i) $A(\xi_t)-B(\xi_t)K = (A-BK) + ((A(\xi_t)-A) + (B(\xi_t)-B)K)$, (ii) $(L,\alpha)$-stability of $A-BK$, and (iii) $\|(A(\xi_t)-A) + (B(\xi_t)-B)K\|\leq \Delta$ a.s., we~apply~Proposition~\ref{prop:just}\ref{prop:just-stab} and obtain the result. Proposition \ref{prop:just}\ref{prop:just-detect} can be proved similarly.
  \qedhere
\end{proof}

\subsection{Proof of Theorem \ref{thm:decay}}\label{apx:decay}

We  prove Theorem \ref{thm:decay} by using the equivalence between \eqref{eqn:spc} and \eqref{eqn:spcd}. The following lemma formally establishes such equivalence.

\begin{lemma}\label{lem:equiv}
  Under Assumptions \ref{ass:fin} and \ref{ass:main} and given $\overline{w}_{\tau-1}\in\mathbb{R}^{n_x}\times\mathbb{R}^{n_u}$, $\overline{\xi}_0\in\Xi_0$, $\tau\in\cT$, and $W\geq 0$, there exists a unique solution of \eqref{eqn:spc} for any $\overline{\bxi}_{0:\tau}\in\bXi_{0:\tau}(\xi_0)$. Furthermore, for $k\in\cV$ such that $\underline{\bxi}_{0\rightarrow k} = \overline{\bxi}_{0:\tau}$ (such $k$ exists due to Proposition \ref{prop:scen}\ref{prop:scen-a}), we have
  \begin{equation}\label{eqn:equiv}
    w^{(\tau,W)}_{t(j)}(\underline\bxi_{0\rightarrow j};\overline{w}_{\tau-1})=\underline w^{(k,W)}_j(\overline{w}_{\tau-1}), \quad\quad\forall j\in \cV^{(k)}_{\tau:\tau+W}.
  \end{equation}
\end{lemma}

\begin{proof}{Proof.}

  We only have to show that $\{w_t(\cdot)\}_{t\in\cT_{\tau:\tau+W}}$ with $w_t:\bXi_{0:t}(\overline{\bxi}_{0:\tau})\rightarrow \mathbb{R}^{n_x}\times \mathbb{R}^{n_u}$ satisfying
  \begin{equation}\label{eqn:construct}
    w_{t(j)}(\underline{\bxi}_{0\rightarrow j}) = w^{(k,W)}_{j}(\overline{w}_{\tau-1}),\quad \forall j\in\cV^{(k)}_{\tau:\tau+W},
  \end{equation}
  is a unique solution of \eqref{eqn:spc}. We note that the definition \eqref{eqn:construct} covers the entire domain of $\{w_t(\cdot)\}_{t\in\cT_{\tau:\tau+W}}$ (cf. Proposition \ref{prop:scen}\ref{prop:scen-a}). Suppose there exists $\{w'_t(\cdot)\}_{t\in\cT_{\tau:\tau+W}}\neq \{w_t(\cdot)\}_{t\in\cT_{\tau:\tau+W}}$ that satisfies the constraints of \eqref{eqn:spc} and does not have a worse objective value than $\{w_t(\cdot)\}_{t\in\cT_{\tau:\tau+W}}$. By Proposition \ref{prop:scen}, we can express the expectation in \eqref{eqn:spc} as an explicit summation:
  \begin{align*}
    \mathbb{E}_{\bxi}\left[\sum_{t\in\cT_{\tau:\tau+W}}\ell(w'_t(\bxi_{0:t});\xi_t)\;\middle|\; \bxi_{0:\tau}= \overline{\bxi}_{0:\tau}\right]
    &= \sum_{j\in\cV^{(k)}_{t(k):t(k)+W}}\pi_{j|k}{\ell}(w'_{t(j)}(\underline\bxi_{0\rightarrow j}); \underline{\xi}_j)\\
    & \stackrel{\mathclap{\eqref{eqn:modeld}}}{=} \sum_{j\in\cV^{(k)}_{t(k):t(k)+W}}\pi_{j|k}\underline{\ell}_j(w'_{t(j)}(\underline\bxi_{0\rightarrow j})).
  \end{align*}
  Applying Proposition \ref{prop:scen}\ref{prop:scen-a}, the constraints of \eqref{eqn:spc} can be rewritten as
  \begin{align*}
    & x'_\tau({\underline\bxi}_{0\rightarrow k}) = f (\overline{w}_{\tau-1};\overline\xi_{\tau})\\
    & x'_t({\underline\bxi}_{0\rightarrow j}) = f(w'_{t-1}({\underline\bxi}_{0\rightarrow j}); {\underline\xi}_{j}),\; \forall j\in\cV^{(k)}_{t(k)+1:t(k)+W}. 
  \end{align*}
  Applying \eqref{eqn:modeld} and multiplying $\pi_{j|k}$ on both sides, we further have
  \begin{align*}
    & x'_\tau({\underline\bxi}_{0\rightarrow k}) = \underline f_k (\overline{w}_{\tau-1})\\
    & \pi_{j|k} x'_t({\underline\bxi}_{0\rightarrow j}) = \pi_{j|k}\underline f_j (w'_{t-1}({\underline\bxi}_{0\rightarrow j})),\; \forall j\in\cV^{(k)}_{t(k)+1:t(k)+W}.
  \end{align*}
  Thus, we know that $\{w'(\bxi_{0\rightarrow j})\}_{j\in\cV^{(k)}_{\tau:\tau+W}}$ is also feasible for Problem \eqref{eqn:spcd} and does not have a worse objective than $\{\underline w^{(k,W)}_i(\overline{w}_{\tau-1})\}_{i\in\cV^{(k)}_{\tau:\tau+W}}$. This contradicts Theorem \ref{thm:decay-ext-sub} that $\{\underline w^{(k,W)}_i(\overline{w}_{\tau-1})\}_{i\in\cV^{(k)}_{\tau:\tau+W}}$ is a unique solution of \eqref{eqn:spcd}. Thus, we prove the existence of the unique solution of \eqref{eqn:spc}.
  \qedhere
\end{proof}

We are now ready to prove Theorem \ref{thm:decay}

\begin{proof}{Proof of Theorem \ref{thm:decay}.}

  For now, we consider a special case of $\overline{w}_{\tau-1}=\bzero$. We choose $k$ so that $\overline{\bxi}_{0:\tau}=\underline\bxi_{0\rightarrow k}$ (such $k$ exists due to Proposition \ref{prop:scen}\ref{prop:scen-a}). By Lemma \ref{lem:equiv} and Theorem \ref{thm:decay-ext-sub}, there exist unique solutions of \eqref{eqn:spc} and \eqref{eqn:spcd}, namely, $\{w^{(\tau,W)}_t (\cdot;\overline w_{\tau-1})\}_{t\in\cT_{\tau:\tau+W}}$ and $\{\underline w^{(k,W)}_i (\overline{w}_{\tau-1}) \}_{i\in\cV^{(k)}_{\tau:\tau+W}}$. By Lemma \ref{lem:equiv} and Proposition \ref{prop:ev}, we have
  \begin{equation}\label{eqn:decay-equiv:1}
    \|\underline\bw^{(k,W)}_{\cV^{(k)}_{t}}(\overline{w}_{\tau-1})\|_{\bpi}  = \pi_k^{1/2}\left\{\mathbb{E}_{\bxi}\left[\|w^{(\tau,W)}_{t}(\bxi_{0:t};\overline{w}_{\tau-1})\|^2\;\middle|\; \bxi_{0:\tau}=\underline\bxi_{0\rightarrow k} \right]\right\}^{1/2}.
  \end{equation}
  Further, by the definition of $\underline\bp$ and Proposition \ref{prop:ev}, we have
  \begin{equation}\label{eqn:decay-equiv:2}
    \| \underline\bp_{\cV^{(k)}_{t'}} \|_{\bpi} = \pi_k^{1/2}\left\{\mathbb{E}_{\bxi}\left[\|p(\xi_{t'})\|^2 \;\middle|\; \bxi_{0:\tau}=\underline\bxi_{0\rightarrow k}\right]\right\}^{1/2}.
  \end{equation}
  From Theorem \ref{thm:decay-ext-sub}, we have
  \begin{equation}\label{eqn:decay-0}
    \underline\bw^{(k,W)}_{\cV^{(k)}_{t}}  =  \sum_{t'\in\cT_{\tau:\tau+W}} \underline\bPsi^{(k,W)}_{\cV^{(k)}_{t}, \cV^{(k)}_{t'}}  \underline\bp_{\cV^{(k)}_{t'}},\quad \forall t\in\cT_{\tau:\tau+W}.
  \end{equation}
  Combining \eqref{eqn:decay-equiv:1}, \eqref{eqn:decay-equiv:2} and \eqref{eqn:decay-0}, dividing both sides by $\pi_k^{1/2}$ (nonzero due to Proposition \ref{prop:scen}\ref{prop:scen-b}), and applying Theorem \ref{thm:decay-ext-sub}, we have for all $t\in\cT_{\tau:\tau+W}$ that
  \begin{align*}
    \left\{\mathbb{E}_{\bxi}\left[\|w^{(\tau,W)}_{t}(\bxi_{0:t};\overline{w}_{\tau-1})\|^2  \mid \bxi_{0:\tau}=\overline{\bxi}_{0:\tau}\right]\right\}^{1/2}
    &\leq   \sum_{t'\in\cT_{\tau:\tau+W}} \left\|\underline\bPsi^{(k,W)}_{\cV^{(k)}_{t}, \cV^{(k)}_{t'}}\right\|_{\bpi} \left\{\mathbb{E}_{\bxi}\left[\|p(\xi_{t'})\|^2 \mid \bxi_{0:\tau}=\overline{\bxi}_{0:\tau}\right]\right\}^{1/2}\\
    &\leq   \sum_{t'\in\cT_{\tau:\tau+W}} c_1 \rho^{|t-t'|} \left\{\mathbb{E}_{\bxi}\left[\|p(\xi_{t'})\|^2 \mid \bxi_{0:\tau}=\overline{\bxi}_{0:\tau}\right]\right\}^{1/2}.
  \end{align*}
  Setting $d(\overline{\xi}_{\tau}) \leftarrow d(\overline{\xi}_{\tau}) + A(\overline{\xi}_{\tau}) \overline{x}_{\tau-1}+ B(\overline{\xi}_{\tau}) \overline{u}_{\tau-1}$ and using Assumption \ref{ass:main}\ref{ass:main-bdd}, we obtain the result for $\overline{w}_{\tau-1}\neq\bzero$. This completes the proof.
  \qedhere
\end{proof}

\subsection{Proof of Theorem \ref{thm:stab}}\label{apx:stab}

Recall the definition of $\{\underline w^{\red{(\textrm{cl},W)}}_i(\overline{w}_{-1})\}_{i\in\cV}$ from \eqref{eqn:spcd-pol} and $\underline{\bPsi}^{(k,W)}$ from Theorem \ref{thm:decay-ext-sub}. Further, we have the following formula from Theorem \ref{thm:decay-ext-sub}:
\begin{equation}\label{eqn:rec}
  \underline w^{\red{(\textrm{cl},W)}}_k(\overline{w}_{-1})
  = \underline S^{(W)}_{k,a(k)} \underline w^{\red{(\textrm{cl},W)}}_{a(k)} (\overline{w}_{-1})+ \sum_{t'\in\cT_{t(k):t(k)+W}}\underline \bPsi^{(k,W)}_{k,\cV^{(k)}_{t'}} \underline\bp_{\cV^{(k)}_{t'}},\quad \forall k\in\cV,
\end{equation}
where $\underline\bS^{(W)}\coloneqq\{\underline S^{(W)}_{ij}\}_{i,j\in\cV}$ and  $\underline\bLambda\coloneqq\{\underline\Lambda_{ij}\}_{i,j\in\cV}$ have the form
\begin{equation}\label{eqn:Lambda}
  \underline \Lambda_{ij}\coloneqq
  \begin{cases}
    \begin{bmatrix}
      \bzero & \bzero \\
      \bzero &\bzero  \\
      \underline A_{i} &\underline B_{i}  \\
    \end{bmatrix}&\text{if }i\in c(j), \\
    \bzero&\text{otherwise}, 
  \end{cases}\quad \quad\quad 
  \underline S^{(W)}_{ij}\coloneqq
  \begin{cases}
    \underline \Psi^{(i,W)}_{ii} \underline \Lambda_{ij}&\text{if }i\in c(j), \\
    \bzero& \text{otherwise},
  \end{cases}
\end{equation}
and we also have $\underline S_{0,a(0)}\coloneqq\underline \Psi^{(0,W)}_{0,0} \underline \Lambda_{0,a(0)}$, and $\underline \Lambda_{0,a(0)} = \begin{bmatrix}\bzero&\bzero\\\bzero&\bzero\\\underline A_{0}&\underline B_{0}\\\end{bmatrix}$.

The recursion in \eqref{eqn:rec} suggests that $\underline w^{\red{(\textrm{cl},W)}}_k(\overline{w}_{-1})$ can be expressed in terms of $\underline \bp_{\cV^{(k)}_{t(k):t(k)+W}}$ and the previous augmented state $\underline w^{\red{(\textrm{cl},W)}}_{a(k)}(\overline{w}_{-1})$. This means that for each $t$, $\underline{\bw}^{\red{(\textrm{cl},W)}}_{\cV_t}(\overline{w}_{-1})$ can be expressed in terms of $\underline \bp_{\cV_{t:t+W}}$ and $\underline{\bw}^{\red{(\textrm{cl},W)}}_{\cV_{t-1}}(\overline{w}_{-1})$. By concatenating \eqref{eqn:rec}, we obtain
\begin{equation}\label{eqn:rec-full}
  \underline\bw^{\red{(\textrm{cl},W)}}_{\cV_t} (\overline{w}_{-1}) = \underline\bS^{(W)}_{\cV_t,\cV_{t-1}} \underline\bw^{\red{(\textrm{cl},W)}}_{\cV_{t-1}} (\overline{w}_{-1})+ \sum_{t'\in\cT_{t:t+W}}\underline\bPsi^{(\cV_t,W)}_{\cV_t,\cV_{t'}} \underline\bp_{\cV_{t'}},\quad t\in\cT,
\end{equation}
where $\underline\bPsi^{(\cV_{t},W)}\coloneqq\{\underline\Psi^{(\cV_{t},W)}_{ij}\}_{i,j\in\cV_{t:t+W}}$ has the form
\begin{equation}\label{eqn:Psi}
  \underline \Psi^{(\cV_{t},W)}_{ij} \coloneqq
  \begin{cases}
    \underline \Psi^{(k,W)}_{ij}&\text{if }\exists k\in\cV_{t} \text{ s.t. }i,j\in\cV^{(k)}_{t:t+W}, \\
    \bzero&\text{otherwise}, \\
  \end{cases}
\end{equation}
$\underline \bS^{(W)}_{\cV_0,\cV_{-1}}\coloneqq \underline S^{(W)}_{0,a(0)}$, and $\underline\bw^{\red{(\textrm{cl},W)}}_{\cV_{-1}}(\overline{w}_{-1})\coloneqq \overline{w}_{-1}$.

Based on \eqref{eqn:rec-full}, we derive an explicit expression of $\underline\bw^{\red{(\textrm{cl},W)}}_{\cV_t}(\overline{w}_{-1})$ in terms of $\underline\bp_{\cV_0},\cdots,\underline\bp_{\cV_T}$.

\begin{lemma}\label{lem:help-1}
  Under Assumptions \ref{ass:fin} and \ref{ass:main} and given $\overline{w}_{-1}=\bzero$, $\overline{\xi}_0\in\Xi_0$, and $W\geq 0$, we have
  \begin{equation}\label{eqn:help-1-0} 
    \underline\bw^{\red{(\textrm{cl},W)}}_{\cV_t}(\overline{w}_{-1}) =  \sum_{t'\in\cT_{0:t+W}} \sum_{t''\in\cT_{(t'-W):(t\wedge t')}} \left(\prod_{t'''=t''+1}^t \underline\bS^{(W)}_{\cV_{t'''},\cV_{t'''-1}} \right)\underline\bPsi^{(\cV_{t''},W)}_{\cV_{t''},\cV_{t'}} \underline\bp_{\cV_{t'}},\quad t\in\cT.
  \end{equation}
\end{lemma}

Since Lemma \ref{lem:help-1} is complex in notation, we briefly discuss the intuition behind Lemma \ref{lem:help-1} to facilitate~the reading. The formula in \eqref{eqn:rec-full} allows for recursively eliminating the effect of the previous augmented state. Whenever the previous augmented state is eliminated, $\underline\bS^{(W)}_{\cV_{t'''},\cV_{t'''-1}}$ is multiplied,  and thus we see the production of $\underline\bS^{(W)}_{\cV_{t'''},\cV_{t'''-1}}$ over $t'''=t''+1,\cdots,t$ in \eqref{eqn:help-1-0}. Furthermore, the summation over $\cT_{(t'-W):(t\wedge t')}$ appears because the new effect of $\underline\bp$ is introduced whenever the previous augmented state is eliminated.

\begin{proof}{Proof.}
  We prove \eqref{eqn:help-1-0} by induction. First, one can see that \eqref{eqn:help-1-0} for $t=0$ holds directly from \eqref{eqn:rec-full}. Assuming that the claim holds for $0,\cdots,t$,
  we aim to prove the claim for $t+1$. From \eqref{eqn:rec-full} and \eqref{eqn:help-1-0} for $t$, we have
  \begin{align*}
    \underline\bw^{\red{(\textrm{cl},W)}}_{\cV_{t+1}}(\overline{w}_{-1})
    &=  \underline\bS^{(W)}_{\cV_{t+1},\cV_t}\sum_{t'\in\cT_{0:t+W}} \sum_{t''\in\cT_{(t'-W):(t\wedge t')}} \left(\prod_{t'''=t''+1}^t \underline\bS^{(W)}_{\cV_{t'''},\cV_{t'''-1}} \right)\underline\bPsi^{(\cV_{t''},W)}_{\cV_{t''},\cV_{t'}} \underline\bp_{\cV_{t'}} \\
    &\quad + \sum_{t'\in\cT_{t+1:t+W+1}}\underline\bPsi^{(\cV_{t+1},W)}_{\cV_{t+1},\cV_{t'}} \underline\bp_{\cV_{t'}}\\
    &=  \sum_{t'\in\cT_{0:t}} \sum_{t''\in\cT_{(t'-W):(t\wedge t')}} \left(\prod_{t'''=t''+1}^{t+1} \underline\bS^{(W)}_{\cV_{t'''},\cV_{t'''-1}} \right)\underline\bPsi^{(\cV_{t''},W)}_{\cV_{t''},\cV_{t'}} \underline\bp_{\cV_{t'}} + \underline\bPsi^{(\cV_{t+1},W)}_{\cV_{t+1},\cV_{t+W+1}} \underline\bp_{\cV_{t+W+1}}\\
    &\quad+  \sum_{t'\in\cT_{t+1:t+W}} \sum_{t''\in\cT_{(t'-W):(t\wedge t')}} \left(\prod_{t'''=t''+1}^{t+1} \underline\bS^{(W)}_{\cV_{t'''},\cV_{t'''-1}} \right)\underline\bPsi^{(\cV_{t''},W)}_{\cV_{t''},\cV_{t'}} \underline\bp_{\cV_{t'}} +\\
    &\quad+ \sum_{t'\in\cT_{t+1:t+W}}\underline\bPsi^{(\cV_{t+1},W)}_{\cV_{t+1},\cV_{t'}} \underline\bp_{\cV_{t'}},\\
    & =  \sum_{t'\in\cT_{0:t}} \sum_{t''\in\cT_{(t'-W):(t+1\wedge t')}} \left(\prod_{t'''=t''+1}^{t+1} \underline\bS^{(W)}_{\cV_{t'''},\cV_{t'''-1}} \right)\underline\bPsi^{(\cV_{t''},W)}_{\cV_{t''},\cV_{t'}} \underline\bp_{\cV_{t'}} + \underline\bPsi^{(\cV_{t+1},W)}_{\cV_{t+1},\cV_{t+W+1}} \underline\bp_{\cV_{t+W+1}}\\
    &\quad+ \sum_{t'\in\cT_{t+1:t+W}}\sum_{t''\in\cT_{(t'-W):(t+1\wedge t')}} \left(\prod_{t'''=t''+1}^{t+1} \underline\bS^{(W)}_{\cV_{t'''},\cV_{t'''-1}} \right)\underline\bPsi^{(\cV_{t''},W)}_{\cV_{t''},\cV_{t'}} \underline\bp_{\cV_{t'}},\\
    &=\sum_{t'\in\cT_{0:t+W+1}} \sum_{t''\in\cT_{(t'-W):(t+1\wedge t')}} \left(\prod_{t'''=t''+1}^{t+1} \underline\bS^{(W)}_{\cV_{t'''},\cV_{t'''-1}} \right)\underline\bPsi^{(\cV_{t''},W)}_{\cV_{t''},\cV_{t'}} \underline\bp_{\cV_{t'}}.
  \end{align*}
  Here, the second equality is obtained by splitting the summations; note that $\underline\bPsi^{(\cV_{t+1},W)}_{\cV_{t+1},\cV_{t+W+1}} \underline\bp_{\cV_{t+W+1}}$ term for $t+W+1 > T$ can be treated as zeros; the third equality is obtained by observing that $t+1\wedge t'=t+1$ for $t'\geq t+1$ and $t+1\wedge t'=t'$ for $t'\leq t$, $\prod_{t'''=t+2}^{t+1} \underline\bS^{(W)}_{\cV_{t'''},\cV_{t'''-1}}=\bI$, and by merging the third and fourth term; the last equality can be obtained by merging the summations. Thus, by induction, \eqref{eqn:help-1-0} is proved.
  \qedhere
\end{proof}

From Lemma \ref{lem:help-1}, we see that the boundedness of $\underline\bw^{(k)}_{\cV_t}$ can be obtained by showing $\|\prod_{t'''=t''+1}^t \underline\bS^{(W)}_{\cV_{t'''},\cV_{t'''-1}}\|_{\bpi} $ decays exponentially in $t-t''$. We do so by showing two results: (i)  $\|\prod_{t'''=t''+1}^t \underline\bS^{(\red{T})}_{\cV_{t'''},\cV_{t'''-1}}\|_{\bpi} $ exponentially decays, and (ii)
  $\underline\bS^{(W)}_{\cV_{t'''},\cV_{t'''-1}}-\underline\bS^{(\red{T})}_{\cV_{t'''},\cV_{t'''-1}}$ is exponentially small in $W$.
\red{
Here, we use $W=\infty$ to denote the case where the horizon fully covers the rest of the horizon. It is actually not an infinite horizon because, based on our definition, $\cT_{t:\infty} = \cT_{t:T}$~for~any~$t\in\cT$.
}
  
  Next, by applying Lemma \ref{lem:help-3}, we obtain the desired result. Here, we note that Lemma \ref{lem:help-3} holds even if $\|\cdot\|$ is replaced by $\|\cdot\|_{\bpi}$ because $\|\cdot\|_{\bpi}$ is submultiplicative (cf. Proposition \ref{prop:norm}\ref{prop:norm-c}).
We prove the first step in the following two lemmas. 

\begin{lemma}\label{lem:bellman}
  Under Assumptions \ref{ass:fin} and \ref{ass:main} and given $\overline{\xi}_0\in\Xi_0$, $\overline{w}_{a(i)}\in\mathbb{R}^{n_x}\times\mathbb{R}^{n_u}$, $i\in\cV$, and its strict descendant $j\in\cV$, we have
  \begin{equation}\label{eqn:bel}
    \underline\bw^{(i,\red{T})}_{\cV^{(j)}}(\overline{w}_{a(i)}) = \underline\bw^{(j,\red{T})}(\underline{w}_{a(j)}^{(i,\red{T})}(\overline{w}_{a(i)})).
  \end{equation}
\end{lemma}
Recall that $\underline\bw^{(i,\red{T})}(\overline{w}_{a(i)})$ solves Problem \eqref{eqn:spcd} that roots from $i$, and the left hand side of \eqref{eqn:bel} denotes the part of $\underline\bw^{(i,\red{T})}(\overline{w}_{a(i)})$ associated with $\cV^{(j)}$, which is the subtree rooting from $j$. The right-hand side denotes the solution that solves Problem \eqref{eqn:spcd} rooting from $j$.

\begin{proof}{Proof.}
  We prove this by contradiction. Suppose the result does not hold; that is, $\underline\bw^{(i,\red{T})}_{\cV^{(j)}}(\overline{w}_{a(i)})$ is not a solution of Problem \eqref{eqn:spcd} with $k=j$ and $\overline{w}_{a(k)}=\underline{w}_{a(j)}^{(i,\red{T})}(\overline{w}_{a(i)})$. By Theorem~\ref{thm:decay-ext-sub}, we~know that there exists a feasible point $\underline\bw'_{\cV^{(j)}}\coloneqq\{\underline w'_\ell\}_{\ell\in\cV^{(j)}}\neq \underline\bw^{(i,\red{T})}_{\cV^{(j)}}(\overline{w}_{a(i)})$~that has a smaller objective value for Problem \eqref{eqn:spcd} with $k=j$ and $\overline{w}_{a(k)}=\underline{w}_{a(j)}^{(i,\red{T})}(\overline{w}_{a(i)})$ than $\underline\bw^{(i,\red{T})}_{\cV^{(j)}}(\overline{w}_{a(i)})$. Then, we can easily see~that $\underline\bw''_{\cV^{(j)}}\coloneqq\{\underline w''_\ell\}_{\ell\in\cV^{(i)}}$ with $\underline w''_\ell\coloneqq
  \begin{cases}
    \underline w'_\ell&\text{if }\ell\in\cV^{(j)}\\
    \underline w^{(i,\red{T})}_\ell(\overline{w}_{a(i)})&\text{otherwise}
  \end{cases}$ is feasible and has a smaller objective value for Problem \eqref{eqn:spcd} with $k=i$ than $\underline\bw^{(i,\red{T})}(\overline{w}_{a(i)})$. This contradicts the fact that $\underline\bw^{(i,\red{T})}(\overline{w}_{a(i)})$ is the unique solution (cf. Theorem \ref{thm:decay-ext-sub}). Thus, we complete the proof.
  \qedhere
\end{proof}

\begin{lemma}\label{lem:help-2}

  Under Assumptions \ref{ass:fin} and \ref{ass:main} and given $\overline{\xi}_0\in\Xi_0$ and $t''\in\cT$, we have
  \begin{equation*}
    \left\|\prod_{t'''=t''+1}^t \underline\bS^{(\red{T})}_{\cV_{t'''},\cV_{t'''-1}}\right\|_{\bpi}\leq \dfrac{2c_1L}{\rho} \rho^{t-t''},\quad \forall t\in\cT_{t'':T}.
  \end{equation*}
\end{lemma}

\begin{proof}{Proof.}
  Let $\underline\bp_{\cV_{t''+1:T}}=\bzero$ and consider $\underline{\bw}_{\cV_{t}}$ obtained by recursively applying \eqref{eqn:rec-full} starting from $t=t''+1$ with $W=\infty$ and given $\overline{\bw}_{\cV_{t''}}\coloneqq\{\overline{w}_i\}_{i\in\cV_{t''}}$. Then, we have
  \begin{equation}\label{eqn:help-2-2}
    \underline{\bw}_{\cV_t} =  \prod_{t'''=t''+1}^t \underline\bS^{(\red{T})}_{\cV_{t'''},\cV_{t'''-1}} \overline{\bw}_{\cV_t''},\quad\forall  t\in\cT_{t'':T}.
  \end{equation}
  By Lemma \ref{lem:bellman}, we know \eqref{eqn:rec-full} with $W=\infty$ follows the exact open-loop policy, and thus we have
  \begin{equation}\label{eqn:help-2-3}
    \underline{\bw}_{\cV_{t}} = \bPsi^{(\cV_{t''+1},\red{T})}_{\cV_{t},\cV_{t''+1}}\underline\bLambda_{\cV_{t''+1},\cV_{t''}}\overline{\bw}_{\cV_{t''}},\quad\forall t\in\cT_{t'':T}.
  \end{equation}
  By the equivalence between \eqref{eqn:help-2-2} and \eqref{eqn:help-2-3} for all $\overline\bw_{\cV''_t}$ and the injectivity of the mappings, we obtain
  \begin{equation*}
    \prod_{t'''=t''+1}^t \underline\bS^{(\red{T})}_{\cV_{t'''},\cV_{t'''-1}}= \underline\bPsi^{(\cV_{t''+1},\red{T})}_{\cV_{t},\cV_{t''+1}} \underline\bLambda_{\cV_{t''+1},\cV_{t''}},\quad\forall  t\in\cT_{t'':T}.
  \end{equation*}
  By Theorem \ref{thm:decay-ext-sub}, we have $\|\underline\bPsi^{(k,\red{T})}_{\cV^{(k)}_{t},k}\|_{\bpi}\leq c_1\rho^{t-t''-1}$ for any $k\in\cV_{t''+1}$. Noting that $\underline\bPsi^{(\cV_{t''+1},\red{T})}_{\cV_{t},\cV_{t''+1}}$ has a block diagonal structure (cf. \eqref{eqn:Psi}), we have
  \begin{equation*}
    \|\underline\bPsi^{(\cV_{t''+1},\red{T})}_{\cV_{t},\cV_{t''+1}}\|_{\bpi} \leq \max_{k\in\cV_{t''+1}}\|\underline\bPsi^{(k,\red{T})}_{\cV^{(k)}_{t},k}\|_{\bpi}\leq c_1\rho^{t-t''-1}.
  \end{equation*}
  Furthermore, using the block diagonal structure of $\underline\bLambda_{\cV_{t''+1},\cV_{t''}}$, we have the following for $t''\in\cT_{0:T-1}$:
  \begin{equation}\label{eqn:lambound}
    \|\underline\bLambda_{\cV_{t''+1},\cV_{t''}}\|_{\bpi} \leq \left(\max_{j\in\cV_{t''}}\sum_{i\in c(j)}\pi_{i|j}\|\underline\Lambda_{ij}\|^2\right)^{1/2} \stackrel{\eqref{eqn:Lambda}}{=} \left(\max_{j\in\cV_{t''}}\sum_{i\in c(j)} \pi_{i|j} \left\|\begin{bmatrix} & \\&\\\underline A_i&\underline B_i  \end{bmatrix} \right\|^2\right)^{1/2}\leq 2L,
  \end{equation}
  where the first inequality follows from the property of induced 2-norm, and the last inequality follows from Assumption \ref{ass:main}\ref{ass:main-bdd} and Proposition \ref{prop:scen}\ref{prop:scen-b}. Finally, combining the above three displays completes the proof.
  \qedhere
\end{proof}

The second step is proved by the following lemma.

\begin{lemma}\label{lem:trunc}
  Under Assumptions \ref{ass:fin} and \ref{ass:main} and given $\overline{\xi}_0\in\Xi_0$ and $W\geq 0$, we have
  \begin{equation*}
    \|\underline\bPsi^{(\cV_t,W)}_{\cV_t,\cV_{t'}} - \underline\bPsi^{(\cV_t,\red{T})}_{\cV_t,\cV_{t'}} \|_{\bpi} \leq  2c_1^2L  \rho^{2W-t'+t},\quad \|\underline\bS^{(\cV_t,W)}_{\cV_t,\cV_{t-1}} - \underline\bS^{(\cV_t,\red{T})}_{\cV_t,\cV_{t-1}} \|_{\bpi} \leq  4c_1^2L^2  \rho^{2W},\quad \forall t\in\cT,\; t'\in\cT_{t:T}.
  \end{equation*}
\end{lemma}

\begin{proof}{Proof.}
  Recall from Theorem \ref{thm:decay-ext-sub} that $\underline\tbOmega^{(k,W)}= \tbH_{\cV^{(k)}_{t(k):t(k)+W},\cV^{(k)}_{t(k):t(k)+W}}^{-1}$ and $\tOmega^{(k,W)}_{ij}=(\pi_{i\mid j})^{1/2}\underline\Omega^{(k,W)}_{ij}$. By the definition, $\tbH_{\cV^{(k)},\cV^{(k)}} \tbOmega^{(k,\red{T})} = \bI$. Extracting the rows and columns of $\cV^{(k)}_{t(k):t(k)+W}$, we can see
  \begin{equation*}
    \tbH_{\cV^{(k)}_{t(k):t(k)+W},\cV^{(k)}_{t(k):t(k)+W}} \tbOmega_{\cV^{(k)}_{t(k):t(k)+W},\cV^{(k)}_{t(k):t(k)+W}}^{(k,\red{T})} +  \tbH_{\cV^{(k)}_{t(k):t(k)+W},\cV^{(k)}_{t(k)+W+1:T}}( \tbOmega_{\cV^{(k)}_{t(k):t(k)+W},\cV^{(k)}_{t(k)+W+1:T}}^{(k,\red{T})})^\top = \bI.
  \end{equation*}
  We multiply $\tbOmega^{(k,W)}$  from the left on both sides, rearrange terms, and obtain
  \begin{equation*}
    \tbOmega^{(k,W)} - \tbOmega^{(k,\red{T})}_{\cV^{(k)}_{t(k):t(k)+W},\cV^{(k)}_{t(k):t(k)+W}} = \tbOmega^{(k,W)}\tbH_{\cV^{(k)}_{t(k):t(k)+W},\cV^{(k)}_{t(k)+W+1:T}} (\tbOmega^{(k,\red{T})}_{\cV^{(k)}_{t(k):t(k)+W},\cV^{(k)}_{t(k)+W+1:T}})^\top.
  \end{equation*}
  Extracting the rows for $k$ and the columns for $\cV^{(k)}_{t'}$, we further obtain
  \begin{equation}\label{eqn:trunc-1}
    \tbOmega^{(k,W)}_{k,\cV^{(k)}_{t'}} - \tbOmega^{(k,\red{T})}_{k,\cV^{(k)}_{t'}} = \tbOmega^{(k,W)}_{k,\cV^{(k)}_{t(k):t(k)+W}}\tbH_{\cV^{(k)}_{t(k):t(k)+W},\cV^{(k)}_{t(k)+W+1:T}} (\tbOmega^{(k,\red{T})}_{\cV^{(k)}_{\cV_{t'}},\cV^{(k)}_{t(k)+W+1:T}})^\top.
  \end{equation}
  We note that the blocks of $\tbH_{\cV^{(k)}_{t(k):t(k)+W},\cV^{(k)}_{t(k)+W+1:T}} $ are zero except for $\tbH_{\cV^{(k)}_{t(k)+W},\cV^{(k)}_{t(k)+W+1}}$. From \eqref{eqn:trunc-1},~we~have
  \begin{equation*}
    \left\|\tbOmega^{(k,W)}_{k,\cV^{(k)}_{t'}} - \tbOmega^{(k,\red{T})}_{k,\cV^{(k)}_{t'}} \right \|\leq \left\|\tbOmega^{(k,W)}_{k,\cV^{(k)}_{t(k)+W}} \right\| \left\| \tbH_{\cV^{(k)}_{t(k)+W},\cV^{(k)}_{t(k)+W+1}} \right\| \left\|\tbOmega^{(k,\red{T})}_{\cV^{(k)}_{\cV_{t'}},\cV^{(k)}_{t(k)+W+1}} \right\|,\quad \forall t'\in\cT_{t(k):t(k)+W}.
  \end{equation*}
  Applying Theorem \ref{thm:decay-ext-sub}, we have
  \begin{equation*}
    \left\| \tbOmega^{(k,W)}_{k,\cV^{(k)}_{t(k)+W}} \right\| \leq c_1 \rho^{W},\quad\quad\quad  \left\|\tbOmega^{(k,\red{T})}_{\cV^{(k)}_{t'},\cV^{(k)}_{t(k)+W+1}}\right\|\leq c_1 \rho^{W-t'+t(k)}.
  \end{equation*}
  Furthermore,
  \begin{equation*}
    \|\tbH_{\cV^{(k)}_{t(k)+W},\cV^{(k)}_{t(k)+W+1}}\|\leq \left(\max_{i\in\cV^{(k)}_{t(k)+W}}\sum_{j\in c(i)}\|\tH_{ij}\|^2\right)^{1/2} \stackrel{\eqref{eqn:Hij}}{=} \left(\max_{i\in\cV^{(k)}_{t(k)+W}}\sum_{j\in c(i)} \pi_{j|i} \left\|\begin{bmatrix} &&A_j^\top \\&&B_j^\top\\\;  \end{bmatrix} \right\|^2\right)^{1/2}\leq 2L,
  \end{equation*}
  where the first inequality follows from the property of induced 2-norm, and the last inequality follows from Assumption \ref{ass:main}\ref{ass:main-bdd}. Combining the above three displays, we obtain
  \begin{equation*}
    \|\tbOmega^{(k,W)}_{k,\cV^{(k)}_{t'}} - \tbOmega^{(k,\red{T})}_{k,\cV^{(k)}_{t'}}\| \leq 2c_1^2L  \rho^{2W-t'+t(k)}.
  \end{equation*}
  Since $\tbPsi^{(k,W)}$ is a submatrix of $\tbOmega^{(k,W)}$, we also have
  \begin{equation*}
    \|\tbPsi^{(k,W)}_{k,\cV^{(k)}_{t'}} - \tbPsi^{(k,\red{T})}_{k,\cV^{(k)}_{t'}}\| \leq 2c_1^2L  \rho^{2W-t'+t(k)}.
  \end{equation*}
  By the block diagonal structure of $\underline\bPsi^{(\cV_{t},W)}_{\cV_{t},\cV_{t'}}$ and $\underline\bPsi^{(\cV_{t},\red{T})}_{\cV_{t},\cV_{t'}}$, we obtain
  \begin{equation}\label{eqn:trunc-4}
    \|\underline\bPsi^{(\cV_t,W)}_{\cV_t,\cV_{t'}} - \underline\bPsi^{(\cV_t,\red{T})}_{\cV_t,\cV_{t'}} \|_{\bpi}
    \leq \max_{k\in\cV_{t}} \|\underline\bPsi^{(k,W)}_{k,\cV^{(k)}_{t'}} - \underline\bPsi^{(k,\red{T})}_{k,\cV^{(k)}_{t'}} \|_{\bpi} \leq \max_{k\in\cV_{t}} \|\tbPsi^{(k,W)}_{k,\cV^{(k)}_{t'}} - \tbPsi^{(k,\red{T})}_{k,\cV^{(k)}_{t'}} \|\
    \leq 2c_1^2L  \rho^{2W-t'+t},
  \end{equation}
  where the second inequality follows from $\widetilde{\bPsi}^{(k,W)}=\{\pi_{i|j}\underline\Psi^{(k,W)}_{ij}\}_{i,j\in\cV^{(k)}_{t(k):t(k)+W}}$ (recall the definition of the scaled problem in \eqref{eqn:spcd-scaled}) and Proposition \ref{prop:norm}\ref{prop:norm-a}.
  Finally, noting the fact that
  \begin{equation*}
    \underline\bS^{(W)}_{\cV_{t},\cV_{t-1}} - \underline\bS^{(\red{T})}_{\cV_{t},\cV_{t-1}}= (\underline\bPsi^{(\cV_{t},W)}_{\cV_{t},\cV_t} - \underline\bPsi^{(\cV_{t},\red{T})}_{\cV_{t},\cV_t} )\underline\bLambda_{\cV_t,\cV_{t-1}},
  \end{equation*}
  we obtain
  \begin{equation*}
    \|\underline\bS^{(W)}_{\cV_{t},\cV_{t-1}} - \underline\bS^{(\red{T})}_{\cV_{t},\cV_{t-1}}\|_{\bpi} \leq \left\|\underline\bPsi^{(\cV_{t},W)}_{\cV_{t},\cV_t} - \underline\bPsi^{(\cV_{t},\red{T})}_{\cV_{t},\cV_t} \right\|_{\bpi} \|\underline\bLambda_{\cV_t,\cV_{t-1}}\|_{\bpi} \leq 4c_1^2L^2  \rho^{2W},
  \end{equation*}
  where the inequalities follow from Proposition \ref{prop:norm}\ref{prop:norm-c}, \eqref{eqn:lambound}, and \eqref{eqn:trunc-4}.
  \qedhere
\end{proof}

By Lemmas \ref{lem:help-2} and \ref{lem:trunc}, we have shown that (i) $\|\prod_{t'''=t''+1}^t \underline\bS^{(\red{T})}_{\cV_{t'''},\cV_{t'''-1}}\|_{\bpi} $ decays exponentially in $t-t''$, and (ii) $\|\underline\bS^{(W)}_{\cV_{t'''},\cV_{t'''-1}}-\underline\bS^{(\red{T})}_{\cV_{t'''},\cV_{t'''-1}}\|_{\bpi}$ can be made arbitrarily small. Thus, we can show $\|\prod_{t'''=t''+1}^t \underline\bS^{(W)}_{\cV_{t'''},\cV_{t'''-1}}\|_{\bpi} $ decays exponentially in $t-t''$. Based on this result, we can derive the bound for $\underline\bw^{\red{(\textrm{cl},W)}}_{\cV_t}(\overline{w}_{-1})$. 

\begin{proof}{Proof of Theorem \ref{thm:stab}.}

  Since $W\geq \overline{W}$, we apply Lemma \ref{lem:trunc} and can verify that 
  \begin{equation*}
    \|\underline\bS^{(W)}_{\cV_t,\cV_{t-1}}-\underline\bS^{(\red{T})}_{\cV_t,\cV_{t-1}}\|_{\bpi}\leq (\alpha^{1/2}-\alpha)/L.
  \end{equation*}
  From Lemmas \ref{lem:help-3} and \ref{lem:help-2}, we have that
  \begin{equation}\label{eqn:stab-0}
    \left\|\prod_{t'''=t''+1}^t \underline\bS^{(W)}_{\cV_{t'''},\cV_{t'''-1}}\right\|_{\bpi} \leq (2c_1L/\rho) \rho^{(t-t'')/2}.
  \end{equation}
  Assuming $\overline w_{-1}=\bzero$ for now, by Lemma \ref{lem:help-1}, there exists $\underline\bPsi^{(W)}\coloneqq\{\underline\Psi^{(W)}_{ij}\}_{i,j\in\cV}$ such that
  \begin{align*}
    \underline\bw^{\red{(\textrm{cl},W)}}(\overline{w}_{-1})=\underline\bPsi^{(W)}\underline\bp,
  \end{align*}
  and again by Lemma \ref{lem:help-1},
  \begin{align}\label{eqn:stab}
    \|\underline\bPsi^{(W)}_{\cV_t,\cV_{t'}}\|_{\bpi} &\leq  \sum_{t''\in\cT_{(t'-W):(t\wedge t')}} \left\|\prod_{t'''=t''+1}^t \underline\bS^{(W)}_{\cV_{t'''},\cV_{t'''-1}} \right\|_{\bpi}\left\|\underline\bPsi^{(\cV_{t''},W)}_{\cV_{t''},\cV_{t'}} \right\|_{\bpi} \nonumber\\
                                                      &\leq \sum_{t''=0}^{t\wedge t'}  (2c_1L/\rho)\rho^{(t-t'')/2} c_1\rho^{t'-t''} \nonumber\\
                                                      &\leq   (2c_1^2L/\rho) \rho^{|t-t'|/2}\sum_{t''=0}^{t\wedge t'}\rho^{t/2-3t''/2 + t' - |t-t'|/2} \nonumber\\
                                                      &\leq \frac{2c_1^2L\rho^{|t-t'|/2} }{\rho(1-\rho^{3/2})}.
  \end{align}
  Here, the first inequality follows from Proposition \ref{prop:norm}\ref{prop:norm-c}; the second inequality follows from \eqref{eqn:stab-0} and Theorem \ref{thm:decay-ext-sub}; the third inequality follows from rearranging the terms; and the last inequality follows from that $t/2-3(t\wedge t')/2+t'-|t-t'|/2 \geq 0$.
  By the definitions in \eqref{eqn:spc-pol} and \eqref{eqn:spcd-pol}, and by Lemma \ref{lem:equiv} and Proposition \ref{prop:ev}, we have
  \begin{subequations}\label{eqn:stab-equiv}
    \begin{align}
      \|\underline\bw^{\red{(\textrm{cl},W)}}_{\cV_{t}}(\overline{w}_{-1})\|_{\bpi} &= \pi_k^{1/2}\left\{\mathbb{E}_{\bxi}\left[\|w^{\red{(\textrm{cl},W)}}_{t}(\bxi_{0:t};\overline{w}_{-1})\|^2\;\middle|\; \xi_{0}=\overline\xi_{0} \right]\right\}^{1/2}, \\
      \| \underline\bp_{\cV_{t'}} \|_{\bpi}&= \pi_k^{1/2}\left\{\mathbb{E}_{\bxi}\left[\|p(\xi_{t'})\|^2 \;\middle|\; \xi_{0}=\overline\xi_{0}\right]\right\}^{1/2}.
    \end{align}
  \end{subequations}
  Applying \eqref{eqn:stab} and \eqref{eqn:stab-equiv} to $\underline\bw^{\red{(\textrm{cl},W)}}_{\cV_{t}}(\overline{w}_{-1})=\sum_{t'=0}^T\underline\bPsi^{(W)}_{\cV_{t},\cV_{t'}}\underline\bp_{\cV_{t'}}$, we obtain
  \begin{equation*}
    \left\{\mathbb{E}_{\bxi}\left[\left\|w^{\red{(\textrm{cl},W)}}_t(\bxi_{0:t};\overline{w}_{-1})\right\|^2\;\middle|\;\xi_0=\overline\xi_0\right]\right\}^{1/2} \leq \sum_{t'\in\cT} c_2 \rho^{|t-t'|/2}\left\{\mathbb{E}_{\bxi}\left[\| p(\xi_{t'})\|^2\;\middle|\;\xi_0=\overline\xi_0\right]\right\}^{1/2},\quad \forall t\in\cT.
  \end{equation*}
  By setting $d(\overline{\xi}_0)\leftarrow d(\overline{\xi}_0) + A(\overline{\xi}_0) \overline{x}_{-1}+ B(\overline{\xi}_0) \overline{u}_{-1}$ and applying Assumption \ref{ass:main}\ref{ass:main-bdd}, we obtain the result for $\overline w_{-1}\neq \bzero$. This completes the proof.
  \qedhere
\end{proof}

\subsection{Theorem \ref{thm:perf}}\label{apx:perf}

Let us define
\begin{equation}\label{eqn:newJ}
  \underline{J}^{(W)}_k(\overline{w}_{-1}) \coloneqq\sum_{i\in\cV^{(k)}} \pi_{i|k} \underline{\ell}_i(\underline{w}^{\red{(\textrm{cl},W)}}_i(\overline{w}_{-1})), \hskip1cm 
  \underline{J}^{(W)}_{\cV_t}(\overline{w}_{-1})\coloneqq\sum_{k\in\cV_t}\pi_{k}\underline{J}^{(W)}_k(\overline{w}_{-1}).
\end{equation}
We can observe from Proposition \ref{prop:scen}\ref{prop:scen-c} ($\pi_{i|k}$ is the conditional probability), Theorem \ref{thm:decay-ext-sub} (a unique solution of \eqref{eqn:spcd} exists), and Lemma \ref{lem:equiv} (the solution of \eqref{eqn:spcd} is the solution of \eqref{eqn:spc}) that $J^{(W)}_{\cV_0}(\overline{w}_{-1})$ is the expected performance of \red{SMPC} with prediction horizon length $W$, starting from $\overline{\xi}_0$ and $\overline{w}_{-1}$. Furthermore, we recall from \eqref{eqn:spcd} that
\begin{subequations}\label{eqn:newnewJ}
  \begin{equation}
    \underline{J}^{(k,W)}(\underline{w}_{a(k)})=\sum_{i\in\cV^{(k)}_{t(k):t(k)+W}} \pi_{i|k} \underline{\ell}_i(\underline{w}^{(k,W)}_i(\underline{w}_{a(k)})).
  \end{equation}
  Similarly to \eqref{eqn:newJ}, for $\underline{\bw}_{\cV_{t-1}}=\{\underline{w}_i\}_{i\in\cV_{t-1}}$, we define
  \begin{equation} \underline{J}^{(\cV_t,W)}(\underline{\bw}_{\cV_{t-1}}) \coloneqq\sum_{k\in\cV_t}\pi_{k}\underline{J}^{(k,W)}(\underline{w}_{a(k)}).
  \end{equation}
\end{subequations}
In the next lemma, we establish the exponential decay in the Hessian of $J^{(\cV_t,\red{T})}(\underline{\bw}_{\cV_{t-1}})$ (with respect to $\underline\bp$).

\begin{lemma}\label{lem:sig}
  Under Assumptions \ref{ass:fin} and \ref{ass:main} and given $\overline{\xi}_0\in\Xi_0$, there exists symmetric $\underline\bSigma^{(\cV_t,\red{T})}\coloneqq\{\underline\Sigma^{(\cV_t,\red{T})}_{ij}\}_{i,j\in\cV_{t:T}} $ such that $\underline J^{(\cV_t,\red{T})}(\bzero) = (1/2)(\underline\bp_{\cV_{t:T}})^\top \underline\bSigma^{(\cV_t,\red{T})}\underline\bp_{\cV_{t:T}}$, and the following holds
  \begin{equation*}
    \osigma_{\bpi}(\underline\bSigma^{(\cV_t,\red{T})}_{\cV_{t'},\cV_{t''}})\leq c_1^2L \rho^{|t'-t''|}(|t'-t''| + \dfrac{2}{1-\rho^2} ),\quad\quad  \forall t',t''\in\cT_{t:T}.
  \end{equation*}
\end{lemma}

\begin{proof}{Proof.}

  First, we recall from Theorem \ref{thm:decay-ext-sub} that $\underline\bPsi^{(\cV_t,\red{T})}$ is the solution mapping of \eqref{eqn:spcd}. We have
  \begin{equation*}
    \underline J^{(\cV_t,\red{T})} (\bzero)=  (1/2)\cdot (\underline\bPsi^{(\cV_t,\red{T})}\underline\bp_{\cV_{t:T}})^\top \underline\bP_{\cV_{t:T},\cV_{t:T}} \underline\bPsi^{(\cV_t,\red{T})}\underline\bp_{\cV_{t:T}} + (\underline\bE_{\cV_{t:T},\cV_{t:T}}\underline\bp_{\cV_{t:T}})^\top \underline\bPsi^{(\cV_t,\red{T})}\underline\bp_{\cV_{t:T}},
  \end{equation*}
  where $\underline\bP\coloneqq\{\underline P_{ij}\}_{i,j\in\cV}$, $\underline\bE\coloneqq\{\underline E_{ij}\}_{i,j\in\cV}$, and
  \begin{align*}
    \underline P_{ij}\coloneqq
    \begin{cases}
      \pi_i
      \begin{bmatrix}
        \underline Q_i&\bzero\\
        \bzero&\underline R_i\\
      \end{bmatrix}&\text{if }i=j, \\
      \bzero&\text{otherwise}, 
    \end{cases},\quad
              \underline E_{ij}\coloneqq
              \begin{cases}
                \pi_i
                \begin{bmatrix}
                  \bI&\bzero&\bzero\\
                  \bzero&\bI&\bzero\\
                \end{bmatrix}&\text{if }i=j, \\
                \bzero&\text{otherwise}. 
              \end{cases}
  \end{align*}
  Thus, we can define $\underline\bSigma^{(\cV_t,\red{T})}$ as
  \begin{equation*}
    \underline\bSigma^{(\cV_t,\red{T})}_{}=(\underline\bPsi^{(\cV_t,\red{T})})^\top \underline\bP_{\cV_{t:T},\cV_{t:T}} \underline\bPsi^{(\cV_t,\red{T})} +(1/2)\left( \underline\bE_{\cV_{t:T},\cV_{t:T}}^\top \underline\bPsi^{(\cV_t,\red{T})}  +  (\underline\bPsi^{(\cV_t,\red{T})})^\top \underline\bE_{\cV_{t:T},\cV_{t:T}}\right).
  \end{equation*}
  By the block diagonal structure of $\underline\bP$ and $\underline\bE$,
  \begin{equation*}
    \underline\bSigma^{(\cV_t,\red{T})}_{\cV_{t'},\cV_{t''}}  = \sum_{t'''\in\cT_{t:T}}\left[(\underline\bPsi^{(\cV_{t},\red{T})}_{\cV_{t'''},\cV_{t'}})^\top \underline\bP_{\cV_{t'''},\cV_{t'''}} (\underline\bPsi^{(\cV_t,\red{T})}_{\cV_{t'''},\cV_{t''}})\right] + (1/2)\left( \underline\bE_{\cV_{t'},\cV_{t'}}^\top\underline\bPsi^{(\cV_t,\red{T})}_{\cV_{t'},\cV_{t''}} +  (\underline\bPsi^{(\cV_t,\red{T})}_{\cV_{t'},\cV_{t''}})^\top\underline\bE_{\cV_{t''},\cV_{t''}}\right).
  \end{equation*}
  By Proposition \ref{prop:norm}\ref{prop:norm-d}, Theorem \ref{thm:decay-ext-sub}, and $\osigma_{\bpi}(\underline\bP_{\cV_t,\cV_t})\leq L$, we obtain
  \begin{align*}
    \osigma_{\bpi}\left((\underline\bPsi^{(\cV_t,\red{T})}_{\cV_{t'''},\cV_{t'}})^\top \underline\bP_{\cV_{t'''},\cV_{t'''}} (\underline\bPsi^{(\cV_t,\red{T})}_{\cV_{t'''},\cV_{t''}})\right)
    &\leq \|\underline\bPsi^{(\cV_t,\red{T})}_{\cV_{t'''},\cV_{t'}}\|_{\bpi}\|\underline\bPsi^{(\cV_t,\red{T})}_{\cV_{t'''},\cV_{t''}}\|_{\bpi} \osigma_{\bpi}(\underline\bP_{\cV_{t'''},\cV_{t'''}}) \\
    &\leq c_1^2 L \rho^{|t'-t'''|+ |t''-t'''|}.
  \end{align*}  
  By Proposition \ref{prop:norm}\ref{prop:norm-d}, Theorem \ref{thm:decay-ext-sub}, subadditivity of $\osigma_{\bpi}(\cdot)$ and $\osigma_{\bpi}(\underline\bE_{\cV_{t'},\cV_{t'}})\leq 1$, we obtain
  \begin{align*}
    \osigma_{\bpi}\left(\left( \underline\bE^\top_{\cV_{t'},\cV_{t'}}\underline\bPsi^{(\cV_t,\red{T})}_{\cV_{t'},\cV_{t''}} +  (\underline\bPsi^{(\cV_t,\red{T})}_{\cV_{t'},\cV_{t''}})^\top\underline\bE_{\cV_{t''},\cV_{t''}}\right)\right)
    &\leq \osigma_{\bpi}(\underline\bE^\top_{\cV_{t'},\cV_{t'}}\underline\bPsi^{(\cV_t,\red{T})}_{\cV_{t'},\cV_{t''}}) + \osigma_{\bpi}((\underline\bPsi^{(\cV_t,\red{T})}_{\cV_{t''},\cV_{t'}})^\top\underline\bE_{\cV_{t''},\cV_{t''}})\\
    &\leq \osigma_{\bpi}(\underline\bE^\top_{\cV_{t'},\cV_{t'}})\|\underline\bPsi^{(\cV_t,\red{T})}_{\cV_{t'},\cV_{t''}}\|_{\bpi} + \osigma_{\bpi}(\underline\bE_{\cV_{t''},\cV_{t''}})\|\underline\bPsi^{(\cV_t,\red{T})}_{\cV_{t''},\cV_{t'}}\|_{\bpi}\\
    &\leq 2c_1\rho^{|t'-t''|}.
  \end{align*}
  Combining the above three displays, and applying the subadditivity of $\osigma_{\bpi}(\cdot)$, we have
  \begin{equation*}
    \osigma_{\bpi}(\underline\bSigma^{(\cV_t,\red{T})}_{\cV_{t'},\cV_{t''}}) \leq c_1\rho^{|t'-t''|} + \sum_{t'''\in\cT_{t:T}} c_1^2 L \rho^{|t'-t'''|+ |t''-t'''|}\leq c_1^2 L (\rho^{|t'-t''|}+\sum_{t'''\in\cT_{t:T}}\rho^{|t'-t'''|+ |t''-t'''|}),
  \end{equation*}  
  where the second inequality follows from $L,c_1\geq 1$ (cf. Assumption \ref{ass:main} and \eqref{eqn:constants}). For the right-hand side term, we note from $\rho\in(0,1)$ that 
  \begin{align*}
    \rho^{|t'-t''|} & + \sum_{t'''\in\cT_{t:T}}\rho^{|t'-t'''|+ |t''-t'''|}\\
                    &\leq  \rho^{|t'-t''|} + \sum_{t'''=t}^{t'\wedge t''}\rho^{|t'-t'''|+ |t''-t'''|}  + \sum_{t'''=t'\wedge t''+1}^{t'\vee t''-1}\rho^{|t'-t'''|+ |t''-t'''|} +  \sum_{t'''=t'\vee t''F}^{T}\rho^{|t'-t'''|+ |t''-t'''|} \\
                    &\leq \frac{\rho^{|t''-t'|}}{1-\rho^2}+|t''-t'|\rho^{|t''-t'|} + \frac{\rho^{|t''-t'|}}{1-\rho^2} = \rho^{|t'-t''|}(|t'-t''| + \dfrac{2}{1-\rho^2}).
  \end{align*}
  Combining the above two displays, we complete the proof.
  \qedhere
\end{proof}

Next, we prepare to analyze the stagewise dynamic regret. Let $\widehat{\bw}^{\red{(\textrm{cl},W)}}(\overline{w}_{-1})\coloneqq\{\widehat{w}^{\red{(\textrm{cl},W)}}_k(\overline{w}_{-1})\}_{k\in\cV}$ with 
\begin{equation}\label{equ:def:what}
  \widehat{w}^{\red{(\textrm{cl},W)}}_{k}(\overline{w}_{-1})\coloneqq \underline w^{(k,\red{T})}_k(\underline{w}^{\red{(\textrm{cl},W)}}_{a(k)}(\overline{w}_{-1}))
\end{equation}
This is a {\it hypothetical} augmented state-control variable at node $k$, which is obtained by implementing the optimal full-horizon policy from the previous augmented state-control variable $\underline w^{\red{(\textrm{cl},W)}}_{a(k)}(\overline{w}_{-1})$ (defined in \eqref{eqn:spcd-pol}). It can be equivalently expressed by
\begin{subequations}\label{eqn:hypo}
  \begin{align}
    \widehat{w}^{\red{(\textrm{cl},W)}}_k(\overline{w}_{-1})=\underline S^{(\red{T})}_{k,a(k)} \underline w^{\red{(\textrm{cl},W)}}_{a(k)} (\overline{w}_{-1})+ \sum_{t'\in\cT_{t(k):T}}\underline \bPsi^{(k,\red{T})}_{k,\cV^{(k)}_{t'}} \underline\bp_{\cV^{(k)}_{t'}}.
  \end{align}
  Also,
  \begin{align}
    \widehat{\bw}^{\red{(\textrm{cl},W)}}_{\cV_t}(\overline{w}_{-1})=\underline \bS^{(\red{T})}_{\cV_{t},\cV_{t-1}} \underline {\bw}^{\red{(\textrm{cl},W)}}_{\cV_{t-1}} (\overline{w}_{-1})+ \sum_{t'\in\cT_{t:T}}\underline \bPsi^{(\cV_t,\red{T})}_{\cV_t,\cV_{t'}} \underline\bp_{\cV_{t'}}.
  \end{align}
\end{subequations}
In the next lemma, we prove that this hypothetical augmented state-control variable $\widehat{\bw}^{\red{(\textrm{cl},W)}}_{\cV_t}(\overline{w}_{-1})$ is exponentially close to the actual augmented state-control variable $\underline\bw^{\red{(\textrm{cl},W)}}_{\cV_t}(\overline{w}_{-1})$ in $W$.

\begin{lemma}\label{lem:perf}
  Under Assumptions \ref{ass:fin} and \ref{ass:main} and given $\overline{w}_{-1}\in\mathbb{R}^{n_x}\times\mathbb{R}^{n_u}$, $\overline{\xi}_0\in\Xi_0$, and $W\geq \overline{W}$, we have
  \begin{equation*}
    \|\underline\bw^{\red{(\textrm{cl},W)}}_{\cV_t} (\overline{w}_{-1}) -\widehat{\bw}^{\red{(\textrm{cl},W)}}_{\cV_t} (\overline{w}_{-1})\|_{\bpi} \leq  \left(c_3  D + c_4\rho^{t/2} \|\overline{w}_{-1}\|\right)\rho^W,\quad \forall t\in\cT,
  \end{equation*}
  where $c_3, c_4, D$ are defined in \eqref{eqn:cs}.
\end{lemma}

\begin{proof}{Proof.}
  From the definition of $\underline\bw^{\red{(\textrm{cl},W)}}_{\cV_t}(\overline{w}_{-1})$ and $\widehat{\bw}^{\red{(\textrm{cl},W)}}_{\cV_t}(\overline{w}_{-1})$, we have
  \begin{align*}
    \underline\bw^{\red{(\textrm{cl},W)}}_{\cV_t}-\widehat{\bw}^{\red{(\textrm{cl},W)}}_{\cV_t}
    &=\left(\underline\bS^{(\cV_t,W)}_{\cV_t,\cV_{t-1}}\underline\bw^{\red{(\textrm{cl},W)}}_{\cV_{t-1}} +  \underline\bPsi^{(\cV_t,W)}_{\cV_t,\cV_{t:t+W}} \underline\bp_{\cV_{t:t+W}}\right) -
      \left(\underline\bS^{(\cV_t,\red{T})}_{\cV_t,\cV_{t-1}}\underline\bw^{\red{(\textrm{cl},W)}}_{\cV_{t-1}} +  \underline\bPsi^{(\cV_t,\red{T})}_{\cV_t,\cV_{t:T}} \underline\bp_{\cV_{t:T}}\right)
    \\
    &=(\underline\bS^{(\cV_t,W)}_{\cV_t,\cV_{t-1}}-\underline\bS^{(\cV_t,\red{T})}_{\cV_t,\cV_{t-1}})\underline\bw^{\red{(\textrm{cl},W)}}_{\cV_{t-1}}+\sum_{t'\in\cT_{t:t+W}}(\underline\bPsi^{(\cV_t,W)}_{\cV_t,\cV_{t'}} - \underline\bPsi^{(\cV_t,\red{T})}_{\cV_t,\cV_{t'}}){\underline{\bp}}_{\cV_{t'}}- \sum_{t'\in\cT_{t+W+1:T}}\underline\bPsi^{(\cV_t,\red{T})}_{\cV_t,\cV_{t'}}\underline{\bp}_{\cV_{t'}}.
  \end{align*}
  Here, we suppress the dependency of $\widehat{\bw}^{\red{(\textrm{cl},W)}}$ on $\overline{w}_{-1}$ to ease the notation. Applying Lemma \ref{lem:trunc} and Theorem \ref{thm:decay-ext-sub},
  \begin{align*}
    \|\underline\bw^{\red{(\textrm{cl},W)}}_{\cV_t}-\widehat{\bw}^{\red{(\textrm{cl},W)}}_{\cV_t}\|_{\bpi}
    &\leq 4c_1^2L^2 \rho^{2W} \| \underline\bw^{\red{(\textrm{cl},W)}}_{\cV_{t-1}}\|_{\bpi} + \sum_{t'\in\cT_{t:t+W}} 2c_1^2L \rho^{2W-t'+t} \| \underline\bp_{\cV_{t'}}\|_{\bpi} + \sum_{t'\in\cT_{t+W+1:T}}c_1\rho^{t'-t}\|\underline{\bp}_{\cV_{t'}}\|_{\bpi}\\
    &\leq 2c_1^2L \rho^{W} \left( 2L\rho^{W}\| \underline\bw^{\red{(\textrm{cl},W)}}_{\cV_{t-1}}\|_{\bpi} + \sum_{t'\in\cT_{t:t+W}}  \rho^{W-t'+t} \| \underline\bp_{\cV_{t'}}\|_{\bpi} + \sum_{t'\in\cT_{t+W+1:T}}\rho^{t'-t-W}\|\underline{\bp}_{\cV_{t'}}\|_{\bpi}\right)\\
    &\leq 2c_1^2L \rho^{W} \left(2L\|\underline\bw^{\red{(\textrm{cl},W)}}_{\cV_{t-1}}\|_{\bpi} + \frac{2}{1-\rho}D\right)\\
    &\leq 2c_1^2L \rho^{W} \left(2c_2L\left(2L\rho^{t/2}\|\overline{w}_{-1}\| + \sum_{t'\in\cT}\rho^{|t-t'|/2}D\right) + \frac{2}{1-\rho}D\right)\\
    &\leq 4c_1^2L \rho^{W} \cbr{  \left(\frac{ 2 c_2 L}{1-\rho^{1/2}}+\frac{1}{1-\rho}\right) D + 2c_2L^2 \rho^{t/2} \|\overline{w}_{-1}\| }\\
    &\leq  \left(c_3  D + c_4\rho^{t/2} \|\overline{w}_{-1}\|\right)\rho^W,
  \end{align*}
  where the second inequality follows from $L,c_1\geq 1$; the third inequality follows from the summation of geometric series, $\rho\in(0,1)$, and $\| \underline\bp_{\cV_{t'}}\|_{\bpi}\leq D$ (from \eqref{eqn:cs} and Proposition \ref{prop:ev}); the fourth inequality follows from Theorem \ref{thm:stab}; the fifth inequality can be obtained by using the summation of geometric series; the last inequality can be obtained from the definitions of $c_3,c_4$ in \eqref{eqn:cs}.
  \qedhere
\end{proof}

Now we are ready to prove  Theorem \ref{thm:perf}.

\begin{proof}{Proof of Theorem \ref{thm:perf}.}
  From the definitions in \eqref{eqn:newJ}, \eqref{eqn:newnewJ}, for any $t\in\cT$,
  \begin{align}\label{eqn:perf-a}
    \underline{J}&^{(W)}_{\cV_t}(\overline{w}_{-1}) = \underline{J}^{(W)}_{\cV_{t+1}}(\overline{w}_{-1})
                   +(\underline\bx^{\red{(\textrm{cl},W)}}_{\cV_{t}})^\top (\frac{1}{2}\underline\bQ_{\cV_{t},\cV_{t}} \underline\bx^{\red{(\textrm{cl},W)}}_{\cV_{t}} -\underline\bq_{\cV_{t}})
                   +(\underline\bu^{\red{(\textrm{cl},W)}}_{\cV_{t}})^\top (\frac{1}{2}\underline\bR_{\cV_{t},\cV_{t}} \underline\bu^{\red{(\textrm{cl},W)}}_{\cV_{t}} -\underline\br_{\cV_{t}}),
  \end{align}

  where $\underline\bQ \coloneqq\{\underline Q_{ij}\}_{i,j\in\cV}$, where $\underline Q_{ij} \coloneqq
  \begin{cases}
    \pi_i\underline Q_{i} & \text{if } i=j, \\
    0&\text{otherwise},
  \end{cases}$ (similar for $\underline\bR$). Further,
  \begin{align*}
    \underline{J}^{(\cV_t,\red{T})}(\underline{\bw}^{\red{(\textrm{cl},W)}}_{\cV_{t-1}})
    & \stackrel{\mathclap{\eqref{eqn:newnewJ}}}{=} \sum_{k\in\cV_t} \pi_k\underline{\ell}_k(\underline{w}_k^{(k, \red{T})}(\underline{w}^{\red{(\textrm{cl},W)}}_{a(k)}) ) + \sum_{k\in\cV_{t}}\sum_{i\in \cV^{(k)}_{t+1}}\sum_{j\in \cV^{(i)}_{t+1:T}} \pi_j \underline{\ell}_j(\underline{w}_{j}^{(k, \red{T})}(\underline{w}_{a(k)}^{\red{(\textrm{cl},W)}}) ) \\
    & \stackrel{\mathclap{\substack{\text{Lem. \ref{lem:bellman}}}}}{=}\;\; \sum_{k\in\cV_t} \pi_k\underline{\ell}_k(\underline{w}_k^{(k, \red{T})}(\underline{w}^{\red{(\textrm{cl},W)}}_{a(k)}) ) + \sum_{k\in\cV_{t}}\sum_{i\in \cV^{(k)}_{t+1}}\sum_{j\in \cV^{(i)}_{t+1:T}}  \pi_j \underline{\ell}_j(\underline{w}_{j}^{(i, \red{T})}(\underline{w}_{k}^{(k, \red{T})}(\underline{w}_{a(k)}^{\red{(\textrm{cl},W)}})) ).\\
    &\stackrel{\mathclap{\eqref{equ:def:what}}}{=} \sum_{k\in\cV_t} \pi_k\underline{\ell}_k(\widehat{w}_k^{\red{(\textrm{cl},W)}}) + \sum_{k\in\cV_{t}}\sum_{i\in \cV^{(k)}_{t+1}}\sum_{j\in \cV^{(i)}_{t+1:T}} \pi_j \underline{\ell}_j(\underline{w}_j^{(i, \red{T})}(\widehat{w}_{a(i)}^{\red{(\textrm{cl},W)}}) )  \\
    & \stackrel{\mathclap{\eqref{eqn:newnewJ}}}{=} \sum_{k\in\cV_t} \pi_k\underline{\ell}_k(\widehat{w}_k^{\red{(\textrm{cl},W)}}) + \sum_{i\in \cV_{t+1}} \pi_i\underline{J}^{(i, \red{T})}(\widehat{w}_{a(i)}^{\red{(\textrm{cl},W)}})\\
    & \stackrel{\mathclap{\eqref{eqn:newnewJ}}}{=} \sum_{k\in\cV_t} \pi_k\underline{\ell}_k(\widehat{w}_k^{\red{(\textrm{cl},W)}}) + \underline{J}^{(\cV_{t+1}, \red{T})}(\widehat{\bw}_{\cV_t}^{(W)}).
  \end{align*}
  Here, we suppress the dependency of $\widehat{\bw}^{\red{(\textrm{cl},W)}}$ on $\overline{w}_{-1}$ for concise notation. We also let $\underline{J}^{(W)}_{\cV_{T+1}}(\overline{w}_{-1})=0$ and $\underline{J}^{(\cV_{T+1},\red{T})}(\underline{\bw}^{\red{(\textrm{cl},W)}}_{\cV_{T}}) =0$. We now can write
  \begin{multline}\label{eqn:perf-b}
    \underline{J}^{(\cV_t,\red{T})}(\underline{\bw}^{\red{(\textrm{cl},W)}}_{\cV_{t-1}}) = \underline{J}^{(\cV_{t+1},\red{T})}(\underline{\bw}^{\red{(\textrm{cl},W)}}_{\cV_{t}}) + \underline{J}^{(\cV_{t+1},\red{T})}(\widehat{\bw}^{\red{(\textrm{cl},W)}}_{\cV_{t}}) - \underline{J}^{(\cV_{t+1},\red{T})}(\underline{\bw}^{\red{(\textrm{cl},W)}}_{\cV_{t}})\\
    +(\widehat{\bx}^{\red{(\textrm{cl},W)}}_{\cV_{t}})^\top ((1/2)\underline\bQ_{\cV_{t},\cV_{t}} \widehat{\bx}^{\red{(\textrm{cl},W)}}_{\cV_{t}} -\underline\bq_{\cV_{t}})
    +(\widehat{\bu}^{\red{(\textrm{cl},W)}}_{\cV_{t}})^\top ((1/2)\underline\bR_{\cV_{t},\cV_{t}} \widehat{\bu}^{\red{(\textrm{cl},W)}}_{\cV_{t}} -\underline\br_{\cV_{t}}), \quad \forall t\in\cT.
  \end{multline}
  By subtracting \eqref{eqn:perf-b} from \eqref{eqn:perf-a} and noting that $\widehat{\bx}^{\red{(\textrm{cl},W)}}_{\cV_{t}} = \underline{\bx}^{\red{(\textrm{cl},W)}}_{\cV_{t}}$ (they are fixed by the constraint \eqref{eqn:spcd-con-1}),
  \begin{subequations}\label{eqn:perf}
    \begin{align}
      \nonumber \underline{J}^{(W)}_{\cV_t} -\underline{J}^{(\cV_t,\red{T})}(\underline{\bw}^{\red{(\textrm{cl},W)}}_{\cV_{t-1}})
      & =\underline{J}^{(W)}_{\cV_{t+1}}(\overline{w}_{-1}) - \underline{J}^{(\cV_{t+1},\red{T})}(\underline{\bw}^{\red{(\textrm{cl},W)}}_{\cV_{t}}) \\
      &+ (\underline\bu^{\red{(\textrm{cl},W)}}_{\cV_{t}})^\top ((1/2)\underline\bR_{\cV_{t},\cV_{t}} \underline\bu^{\red{(\textrm{cl},W)}}_{\cV_{t}} - \underline\br_{\cV_{t}}) - (\widehat{\bu}^{\red{(\textrm{cl},W)}}_{\cV_{t}})^\top ((1/2)\underline\bR_{\cV_{t},\cV_{t}} \widehat{\bu}^{\red{(\textrm{cl},W)}}_{\cV_{t}} - \underline\br_{\cV_{t}})\label{eqn:term-1}\\
      &+ \underline{J}^{(\cV_{t+1},\red{T})}(\underline{\bw}^{\red{(\textrm{cl},W)}}_{\cV_{t}})-\underline{J}^{(\cV_{t+1},\red{T})}(\widehat{\bw}^{\red{(\textrm{cl},W)}}_{\cV_{t}}),\label{eqn:term-2}
    \end{align}
  \end{subequations}
  for $t\in\cT$. By Theorem \ref{thm:stab}, Lemma \ref{lem:perf}, and the fact that $\rho\in(0,1)$, we have for any $t\in\cT$,
  \begin{align}\label{eqn:plus}
    \|\underline{\bw}^{\red{(\textrm{cl},W)}}_{\cV_{t}}+\widehat{\bw}^{\red{(\textrm{cl},W)}}_{\cV_{t}}\|
    &\leq 2\|\underline{\bw}^{\red{(\textrm{cl},W)}}_{\cV_{t}} \| + \| \underline{\bw}^{\red{(\textrm{cl},W)}}_{\cV_{t}}-\widehat{\bw}^{\red{(\textrm{cl},W)}}_{\cV_{t}}\| \\\nonumber
    &\leq 2\left(2c_2L\rho^{t/2}\|\overline{w}_{-1}\| + \frac{2c_2}{1-\rho^{1/2} }D\right) + \left(c_3D + c_4\rho^{t/2}\|w_{-1}\|\right)\rho^{W}\\\nonumber
    &\leq \left(\frac{4c_2}{1-\rho^{1/2}} + c_3\right)D + (4c_2L+ c_4)\rho^{t/2} \| w_{-1}\|.
  \end{align}
  Therefore, we can simplify the term in \eqref{eqn:term-1} by
  \begin{align*}
    \text{term in \eqref{eqn:term-1}}
    &= (\underline\bu^{\red{(\textrm{cl},W)}}_{\cV_{t}}-\widehat{\bu}^{\red{(\textrm{cl},W)}}_{\cV_{t}})((1/2)\underline\bR_{\cV_{t},\cV_{t}} (\underline\bu^{\red{(\textrm{cl},W)}}_{\cV_{t}}+\widehat{\bu}^{\red{(\textrm{cl},W)}}_{\cV_{t}}) - \underline\br_{\cV_{t}})\\
    &\leq \|\underline\bu^{\red{(\textrm{cl},W)}}_{\cV_{t}}-\widehat{\bu}^{\red{(\textrm{cl},W)}}_{\cV_{t}}\|\left((L/2)\| \underline\bu^{\red{(\textrm{cl},W)}}_{\cV_{t}}+\widehat{\bu}^{\red{(\textrm{cl},W)}}_{\cV_{t}}\| + \| \underline\br_{\cV_{t}}\|\right)\\
    &\leq (c_3D + c_4\rho^{t/2}\|w_{-1}\| )\rho^W\left(\left(\frac{2c_2L}{1-\rho^{1/2}} +c_3 L/2 + 1\right)D + (2c_2L^2+ c_4L/2)\rho^{t/2} \| w_{-1}\|\right)  \\
    & \leq \Bigg[c_3\left(\frac{2c_2L}{1-\rho^{1/2}} +c_3 L/2 + 1\right) D^2\\
    &\quad+ \left(\frac{2c_2c_4L}{1-\rho^{1/2}} +c_3c_4 L + c_4 + 2c_2c_3L^2\right) \rho^{t/2}D \|w_{-1}\| + c_4(2c_2L^2+c_4L/2)\rho^t \|w_{-1}\|^2\Bigg] \rho^W.
  \end{align*}
  Here, the first inequality follows from Assumption \ref{ass:main}\ref{ass:main-bdd}, the second inequality follows from \eqref{eqn:plus} and Lemma \ref{lem:perf}, and the third inequality can be obtained by rearranging terms. Furthermore, for the term in \eqref{eqn:term-2}, we have
  \begin{align*}
    &\text{term in \eqref{eqn:term-2}}\\
    &= (1/2)(\underline\bp_{\cV_{t+1:T}} +\underline\bLambda_{\cV_{t+1:T},\cV_t} \underline{\bw}^{\red{(\textrm{cl},W)}}_{\cV_t} )^\top \underline\bSigma^{(\cV_{t+1},\red{T})} (\underline\bp_{\cV_{t+1:T}} +\underline\bLambda_{\cV_{t+1:T},\cV_t} \underline{\bw}^{\red{(\textrm{cl},W)}}_{\cV_t} )\\
    &\qquad - (1/2)(\underline\bp_{\cV_{t+1:T}} +\underline\bLambda_{\cV_{t+1:T},\cV_t} \widehat{\bw}^{\red{(\textrm{cl},W)}}_{\cV_t} )^\top \underline\bSigma^{(\cV_{t+1},\red{T})} (\underline\bp_{\cV_{t+1:T}} +\underline\bLambda_{\cV_{t+1:T},\cV_t} \widehat{\bw}^{\red{(\textrm{cl},W)}}_{\cV_t} )\\
    &= (1/2)(\underline\bLambda_{\cV_{t+1:T},\cV_t} (\underline{\bw}^{\red{(\textrm{cl},W)}}_{\cV_t} - \widehat{\bw}^{\red{(\textrm{cl},W)}}_{\cV_t}) )^\top \underline\bSigma^{(\cV_{t+1},\red{T})} (2\underline\bp_{\cV_{t+1:T}} +\underline\bLambda_{\cV_{t+1:T},\cV_t} (\underline{\bw}^{\red{(\textrm{cl},W)}}_{\cV_t} +\widehat{\bw}^{\red{(\textrm{cl},W)}}_{\cV_t} ))\\
    &= (1/2)(\underline\bLambda_{\cV_{t+1},\cV_t} (\underline{\bw}^{\red{(\textrm{cl},W)}}_{\cV_t} -  \widehat{\bw}^{\red{(\textrm{cl},W)}}_{\cV_t}))^\top \left[\underline\bSigma^{(\cV_{t+1},\red{T})}_{\cV_{t+1},\cV_{t+1:T}} (2\underline\bp_{\cV_{t+1:T}})+\underline\bSigma^{(\cV_{t+1},\red{T})}_{\cV_{t+1},\cV_{t+1}}\underline\bLambda_{\cV_{t+1},\cV_t} (\underline{\bw}^{\red{(\textrm{cl},W)}}_{\cV_t} +  \widehat{\bw}^{\red{(\textrm{cl},W)}}_{\cV_t}))\right]\\
    &= \sum_{t'\in\cT_{t+1:T}} (1/2)(\underline\bLambda_{\cV_{t+1},\cV_{t}} (\underline{\bw}^{\red{(\textrm{cl},W)}}_{\cV_t}-\widehat{\bw}^{\red{(\textrm{cl},W)}}_{\cV_t}))^\top \underline\bSigma^{(\cV_{t+1},\red{T})}_{\cV_{t+1},\cV_{t'}} (2\underline{\bp}_{\cV_{t'}})\\
    &\qquad +(1/2)(\underline\bLambda_{\cV_{t+1},\cV_{t}} (\underline{\bw}^{\red{(\textrm{cl},W)}}_{\cV_t}- \widehat{\bw}^{\red{(\textrm{cl},W)}}_{\cV_t}))^\top \underline\bSigma^{(\cV_{t+1},\red{T})}_{\cV_{t+1},\cV_{t+1}} (\underline\bLambda_{\cV_{t+1},\cV_{t}} (\underline{\bw}^{\red{(\textrm{cl},W)}}_{\cV_t}+ \widehat{\bw}^{\red{(\textrm{cl},W)}}_{\cV_t}))\\
    &\leq \sum_{t'\in\cT_{t+1:T}} (1/2)\left\|\underline\bLambda_{\cV_{t+1},\cV_{t}}\right\|_{\bpi} \left\|\underline{\bw}^{\red{(\textrm{cl},W)}}_{\cV_t}-\widehat{\bw}^{\red{(\textrm{cl},W)}}_{\cV_t}\right\|_{\bpi} \usigma_{\bpi}\left(\underline\bSigma^{(\cV_{t+1},\red{T})}_{\cV_{t+1},\cV_{t'}}\right) 2 \left\|\underline{\bp}_{\cV_{t'}}\right\|_{\bpi}\\
    &\qquad +(1/2)\left\|\underline\bLambda_{\cV_{t+1},\cV_{t}}\right\|_{\bpi} \left\|\underline{\bw}^{\red{(\textrm{cl},W)}}_{\cV_t}- \widehat{\bw}^{\red{(\textrm{cl},W)}}_{\cV_t}\right\|_{\bpi} \usigma_{\bpi}\left(\underline\bSigma^{(\cV_{t+1},\red{T})}_{\cV_{t+1},\cV_{t+1}}\right) \left\|\underline\bLambda_{\cV_{t+1},\cV_{t}}\right\|_{\bpi} \left\|\underline{\bw}^{\red{(\textrm{cl},W)}}_{\cV_t}+ \widehat{\bw}^{\red{(\textrm{cl},W)}}_{\cV_t}\right\|_{\bpi}\\
    &\leq  \| \underline{\bw}^{\red{(\textrm{cl},W)}}_{\cV_t}-\widehat{\bw}^{\red{(\textrm{cl},W)}}_{\cV_t}\|_{\bpi}\sum_{t'\in\cT_{t+1:T}} (1/2)  (2L) c_1^2 L\rho^{t'-t-1}(t'-t-1 + \dfrac{2}{1-\rho^2})  2D \\
    &\qquad + (1/2)(2L)^2 \frac{2c_1^2L}{1-\rho^2}\| \underline{\bw}^{\red{(\textrm{cl},W)}}_{\cV_t}-\widehat{\bw}^{\red{(\textrm{cl},W)}}_{\cV_t}\|_{\bpi}\|\underline{\bw}^{\red{(\textrm{cl},W)}}_{\cV_t}+\widehat{\bw}^{\red{(\textrm{cl},W)}}_{\cV_t}\|_{\bpi}\\
    &\leq  2c_1^2L^2 \| \underline{\bw}^{\red{(\textrm{cl},W)}}_{\cV_t}-\widehat{\bw}^{\red{(\textrm{cl},W)}}_{\cV_t}\|_{\bpi}\Bigg(\sum_{t'\in\cT_{t+1:T}}  \left( \rho^{t'-t-1}(t'-t-1 + \dfrac{2}{1-\rho^2})D\right) +  \frac{2L}{1-\rho^2}\|\underline{\bw}^{\red{(\textrm{cl},W)}}_{\cV_t}+\widehat{\bw}^{\red{(\textrm{cl},W)}}_{\cV_t}\|_{\bpi}\Bigg)\\
    &\leq  2c_1^2L^2  \rho^W \left(c_3D +  c_4\rho^{t/2} \|\overline{w}_{-1}\|\right) \Bigg( \left(\frac{1}{(1-\rho)^2} +\frac{2}{(1-\rho)(1-\rho^2)} + \frac{2L}{1-\rho^2}\left(\frac{4c_2}{1-\rho^{1/2}}+c_3\right)\right)D\\
    &\qquad  + \frac{2L}{1-\rho^2}(4c_2L+ c_4)\rho^{t/2}\|\overline{w}_{-1}\| \Bigg)\\
    &\leq \Bigg[\frac{2c_1^2 c_3 L^2}{1-\rho^2}\left( -1 + 2c_3 L + \frac{4}{1-\rho} + \frac{8c_2L}{1-\rho^{1/2}} \right)D^2 \\
    &\qquad +  \frac{2c_1^2 L^2}{1-\rho^2}  \left(-c_4 + 2c_3c_4 L + \frac{4c_4}{1-\rho} + \frac{8c_2c_4L}{1-\rho^{1/2}} + 2c_3L(4c_2L + c_4) \right) \rho^{t/2}D\|\overline{w}_{-1}\| \\
    &\qquad +  \frac{4c_1^2 c_4 L^3(4c_2L+c_4)}{1-\rho^2}   \rho^{t}\|\overline{w}_{-1}\|^2\Bigg]\rho^W. 
  \end{align*}
  Here, the first equality follows from the definition of $\underline\bSigma^{(\cV_{t+1},\red{T})}$, the second equality can be obtained by rearranging terms, and the third equality follows from the observation that $\underline\bLambda_{\cV_{t+2:T},\cV_t}=\bzero$; the first inequality follows from Definition \ref{def:norm} and Proposition \ref{prop:norm}; the second inequality follows from Lemma \ref{lem:sig} and \eqref{eqn:lambound}; the third inequality can be obtained by rearranging; the fourth inequality follows from \eqref{eqn:plus} and Lemma \ref{lem:perf}; and the last inequality can be obtained by rearranging terms and noting that
  \begin{equation*}
    \frac{1}{(1-\rho)^2} +\frac{2}{(1-\rho)(1-\rho^2)} + \frac{2L}{1-\rho^2}\left(\frac{4c_2}{1-\rho^{1/2}}+c_3\right) = \frac{1}{1-\rho^2}\left( -1 + 2c_3 L + \frac{4}{1-\rho} + \frac{8c_2L}{1-\rho^{1/2}} \right).
  \end{equation*}
  By taking the summation of \eqref{eqn:perf} over $t\in\cT$, we obtain 
  \begin{align}\label{eqn:final}
    &\underline{J}^{(W)}_{\cV_0}(\overline{w}_{-1})- \underline{J}^{(\cV_0,\red{T})}(\overline{w}_{-1})\nonumber \\
    &\leq
      \Bigg\{
      \left[c_3\left(\frac{2c_2L}{1-\rho^{1/2}} +c_3 L/2 + 1\right) + \frac{2c_1^2 c_3 L^2}{1-\rho^2}\left( -1 + 2c_3 L + \frac{4}{1-\rho} + \frac{8c_2L}{1-\rho^{1/2}} \right)\right]D^2 T \nonumber\\
    &\qquad+\Bigg[\frac{2c_2c_4L}{1-\rho^{1/2}} +c_3c_4 L + c_4 + 2c_2c_3L^2+\\
    &\qquad\qquad\frac{2c_1^2 L^2}{1-\rho^2}  \left(-c_4 + 2c_3c_4 L + \frac{4c_4}{1-\rho} + \frac{8c_2c_4L}{1-\rho^{1/2}} + 2c_3L(4c_2L + c_4) \right)\Bigg]\frac{D\|\overline{w}_{-1}\|}{1-\rho^{1/2}} \nonumber\\ 
    &\qquad+\left[c_4(2c_2L^2+c_4L/2) + \frac{4c_1^2 c_4 L^3(4c_2L+c_4)}{1-\rho^2}\right]\frac{\|\overline{w}_{-1}\|^2}{1-\rho}
      \Bigg\}\rho^W \nonumber\\
    &\leq
      \left\{c_5 D^2 T  + c_6 D\|\overline{w}_{-1}\|+ c_7 \|\overline{w}_{-1}\|^2\right\}\rho^W,
  \end{align}
  where the second inequality follows from the definition of $c_5, c_6, c_7$ in \eqref{eqn:cs}. We observe from Proposition \ref{prop:scen}\ref{prop:scen-c} ($\pi_{i|k}$ is the conditional probability), Theorem \ref{thm:decay-ext-sub} (a unique solution of \eqref{eqn:spcd} exists), Lemma \ref{lem:equiv} (the solution of \eqref{eqn:spcd} is the solution of \eqref{eqn:spc}), and the definition in \eqref{eqn:newJ} that
  \begin{equation}\label{eqn:Jequiv-1}
    {J}^{(W)}(\xi_0;\overline{w}_{-1})
    = \underline{J}^{(W)}_{\cV_0}(\overline{w}_{-1}).
  \end{equation}
  By the definitions in \eqref{eqn:orig} and \eqref{eqn:origd}, their equivalence (Lemma \ref{lem:equiv}), the existence of unique solutions (Theorem \ref{thm:decay-ext}), we also have 
  \begin{equation}\label{eqn:Jequiv-2}
    {J}^{\star}(\xi_0;\overline{w}_{-1}) = \underline{J}^{(\cV_0,\red{T})}_{}(\overline{w}_{-1}).
  \end{equation}
  Combining \eqref{eqn:final}, \eqref{eqn:Jequiv-1}, and \eqref{eqn:Jequiv-2}, we complete the proof.
  \qedhere
\end{proof}

\section*{Acknowledgment}
This material is based upon work supported by the U.S. Department of Energy, Office of Science, Office of Advanced Scientific Computing Research (ASCR) under Contract DE-AC02-06CH11347. 

\bibliographystyle{informs2014}
\bibliography{main,smpc}

\vspace{0.1cm}
\begin{flushright}
	\scriptsize \framebox{\parbox{2.5in}{Government License: The
			submitted manuscript has been created by UChicago Argonne,
			LLC, Operator of Argonne National Laboratory (``Argonne").
			Argonne, a U.S. Department of Energy Office of Science
			laboratory, is operated under Contract
			No. DE-AC02-06CH11357.  The U.S. Government retains for
			itself, and others acting on its behalf, a paid-up
			nonexclusive, irrevocable worldwide license in said
			article to reproduce, prepare derivative works, distribute
			copies to the public, and perform publicly and display
			publicly, by or on behalf of the Government. The Department of Energy will provide public access to these results of federally sponsored research in accordance with the DOE Public Access Plan. http://energy.gov/downloads/doe-public-access-plan. }}
	\normalsize
\end{flushright}

\end{document}